\documentclass[11pt]{article}
\usepackage{geometry}
\usepackage{mathtools, amsthm, amsmath, amssymb, amsfonts, bbm}
\usepackage{syntonly}
\usepackage{etoolbox}
\usepackage{hyperref}
\newgeometry{margin=2cm}

\usepackage{graphicx, framed}
\graphicspath{{Images/}}
\usepackage{afterpage, xcolor}

\newcommand{\rn}{\mathbb{R}^n}
\newcommand{\R}{\mathbb{R}}
\newcommand{\sn}{\mathbb{S}^{n-1}}
\newcommand{\tn}{\mathbb{T}^n}
\newcommand{\h}[1]{\mathbb{H}^{e'}_{#1}}
\newcommand{\he}[1]{\mathbb{H}^{e}_{#1}}

\newcommand{\cs}{c^*}
\newcommand{\ws}{w^*}
\newcommand{\ls}{\lambda^*}
\newcommand{\e}[1]{e^{#1}}
\newcommand{\cd}{{\cdot}}

\newcommand{\ds}{\displaystyle}
\newcommand{\ih}[1]{\int\limits_{\mathbb{H}^{e'}_{#1}}}
\newcommand{\ihe}[1]{\int\limits_{\mathbb{H}^{e}_{#1}}}
\newcommand{\disk}[1]{D_{#1}}

\newtheorem{theorem}{Theorem}[section]
\newtheorem{defn}[theorem]{Definition}
\newtheorem{lemma}[theorem]{Lemma}
\newtheorem{prop}[theorem]{Proposition} 
\newtheorem{cor}[theorem]{Corollary}

\newcommand{\eal}[3]{\begin{equation}\label{#1}\left\{\hspace{20pt}
\begin{array}{#2}
#3
\end{array}\right.
\end{equation}}

\newcommand{\ea}[2]{\begin{equation*}\left\{\hspace{20pt}
\begin{array}{#1}
#2
\end{array}\right.
\end{equation*}}

\title{Logarithmic Bramson correction for multi-dimensional periodic Fisher-KPP equations}
\author{Beniada Shabani}
\date{}

\begin{document}
\maketitle


\abstract{
We study the long time behavior of solutions of periodic Fisher-KPP type equations in $\rn$ that arise from compactly supported initial data. We prove that propagation along a fixed direction $e\in\sn$ is completely determined by an associated minimizing direction $e'$ which is unique and in a bijective correspondence with $e$. This correspondence allows us to determine the correct Bramson shift in higher dimensions and show that that the propagation along each $e$ occurs asymptotically at the speed $w_*(e)t-\frac{n/2+1}{\lambda_*(e')\,e\cdot e'}\log t+O(1)$. 
}

\section{Introduction}

We consider the Cauchy problem for the reaction-diffusion equation
\eal{kpp}{ll}{
u_t=\Delta u+\mu(x)f(u), & t>0, x\in \rn,\\
u(0,x)=u_0(x), & x\in \rn, 
}
where the non-linearity $f$ is a $C(\R)$ function of Fisher-KPP type, i.e. 
\begin{equation}\label{kpptypef}
f(0)=f(1)=0, \ f'(0)>0, \ f'(1)<0, \ 0<f(s)\le f'(0)s  \text{ for all } s\in(0,1),
\end{equation}
and there exists $M>0$, $\alpha>0$ and $s_0\in(0,1)$ such that
\begin{equation*}
f(s)\ge f'(0)s-Ms^{1+\alpha} \ \text{ for all } s\in[0, s_0].
\end{equation*}
Typical examples of this are concave functions vanishing at $0$ and $1$, like $f(s)=s(1-s)$. By rescaling $\mu$ if necessary, we may assume that 
\begin{equation*}
f'(0)=1.
\end{equation*}
On the other hand we take $\mu\in C^\infty(\rn)$ to be uniformly positive and 1-periodic in the standard directions,
\begin{equation*}
\mu(x+k)=\mu(x) \ \text{ for all } k\in\mathbb{Z}^n, \ 0< \underline{\mu}\le \mu(x)\le \bar{\mu}<\infty.
\end{equation*}
We are interested in the level sets of the solutions of \eqref{kpp} arising from 'localized' initial data, in the sense that $u_0\in L^{\infty}(\rn)$ is compactly supported and satisfies
\begin{equation*}
0\le u_0(x)\le 1, \ u_0\not\equiv 0.
\end{equation*}
In that case the solution $u$ is classical for $t>0$, and the strong maximum principle implies that $0<u(t,x)<1$ for all $t>0$.\\

These equations arise naturally in many mathematical models in biology, ecology, chemical kinetics, etc., see e.g. \cite{xin}, with $u(t,x)$ usually denoting a local population density. They appeared originally in the works of Fisher \cite{fisher} and Kolmogorov, Petrovsky, Piskunov \cite{kpp} in the context of population genetics, and have ever since been extensively studied with both PDE and probabilistic tools, due to their connection to branching Brownian motion (see \cite{mck}). The solutions of \eqref{kpp} are characterized by the invasion of the unstable steady state by the stable one, in that case $0$ and $1$ respectively, as determined by the conditions on $f$. In this paper we study precisely the geometry and speed at which this invasion occurs.

Before stating the main theorem, we'll first recall some previous results, starting with the 1 dimensional homogeneous case studied in \cite{kpp}, building up towards higher dimensions and periodicity.

\subsubsection{Homogeneous equation in $\R$}

Let's first consider \eqref{kpp} when $\mu\equiv 1$ and $n=1$. In this case, the equation
\begin{equation*}
u_t=u_{xx}+f(u)
\end{equation*}
has traveling wave solutions of the form $U_c(t,x)=\phi_c(x-ct)$ for all $c\ge\cs=2$, which satisfy
\begin{equation*}
-c\phi_c'=\phi_c''+f(\phi_c), \ \phi_c(-\infty)=1, \ \phi_c(+\infty)=0.
\end{equation*}
These functions decay exponentially as $x\to \infty$ like $\phi_c(x)\sim C\e{-\lambda_cx}$ if $c>\cs$, and $\phi_{\cs}(x)\sim Cx\e{-\ls x}$ if $c=\cs$. The exponent is determined by exponential solutions $\e{-\lambda x}$ of the linearized equation at the right side tail, i.e. it's the smallest $\lambda$ that solves 
\begin{equation*}
\lambda c = \lambda^2+1.
\end{equation*}
In \cite{kpp} it is shown that the solutions of \eqref{kpp} approach asymptotically to the shifted slowest traveling wave on $\R^+$, i.e.
\begin{equation*}
\lim\limits_{t\to\infty} |u(t,x)-\phi_{\cs}(x-\cs t + s(t))| = 0 \ \text{ uniformly in } x\in\R^+,
\end{equation*}
where the shift $s(t)=o(t)$ as $t\to\infty$. In \cite{uchiyama} Uchiyama provides the first asymptotics on $s(t)$ via a refinement of the \cite{kpp} techniques: $s(t)=\frac{3}{2\ls}\log t + O(\log\log t)$. Bramson refined this result furthermore in \cite{bram1, bram2}, where he shows by entirely probabilistic tools to show that the error is in fact of $O(1)$. More precisely, he showed that the level sets $L_m(t)=\{x: u(t,x)=m\}$ of $u$ for $m\in(0,1)$ lie in
\begin{equation*}
L_m(t)\cap\R^+ \subset \left[\cs t-\frac{3}{2\ls}\log t + x_m - \frac{C}{\sqrt{t}}, \cs t-\frac{3}{2\ls}\log t + x_m + \frac{C}{\sqrt{t}}\right],
\end{equation*}
where the constant shift $x_m$ depends on $m$ and the initial data. The exact value of the constant $C$ has been studied in \cite{lau} and \cite{es1}, \cite{es2}, with the authors in the last two papers formally deriving that $C=3\sqrt{\pi}$. This prediction was verified for a linearized boundary value problem in \cite{henderson}. A new proof of the sharp asymptotics, together with more refined error estimates for the Fisher-KPP case is provided in the works \cite{nrr1}, \cite{nrr2} of Nolen, Roquejoffre and Ryzhik. The sharpest results to date are due to J. Berestycki, Brunet and Derrida in \cite{bbd} and state that the front $u(t,x)=m$ is located at
\begin{equation*}
2t-\frac{3}{2}\log t + x_m - \frac{3\sqrt{\pi}}{\sqrt{t}}+\frac{9}{8}(5-6\log 2)\frac{\log t}{t} + O\left(\frac{1}{t}\right).
\end{equation*}

\subsubsection{Periodic equation}\label{periodicequation}

In the periodic case the notion of traveling waves is replaced by that of periodic pulsating fronts $U_c(t,x)$ associated to a unit direction $e$. These are functions of the form $U_c(t,x)=\phi_c(x\cd e - ct, x)$, with $\phi_c$ periodic in $x$, with boundary conditions
\begin{equation*}
\lim_{s\to -\infty}\phi_c(s,x) = 1, \ \lim_{s\to\infty} \phi_c(s, x)=0, \ \text{ uniformly in } x.
\end{equation*}
$U_c$ itself satisfies the equation
\begin{equation*}
(U_c)_t=\Delta U_c+\mu(x)f(U_c), \ t>0, x\in \rn.
\end{equation*}
Their existence and properties have been established in \cite{bh2}, \cite{bhn1}, and the main result is that for each direction, there exist positive minimal speeds $c^*(e)>0$ such that $U_c$ exists and is unique for all $c\ge c^*(e)$, and no such fronts exist otherwise. These minimal speeds can be characterized using solutions in the form of periodic perturbations of exponential traveling fronts for the linearized equation. Indeed, $\bar{u}(t,x)=e^{-\lambda(x\cd e-ct)}\psi(x; e, \lambda)$ is a positive solution of the linearized equation 
\begin{equation}\label{linear}
\bar{u}_t=\Delta \bar{u} +\mu(x)\bar{u}, \ x\in\rn
\end{equation}
only if $\psi$ satisfies
\begin{equation*}
\lambda c\psi = \Delta \psi - 2\lambda e \cd \nabla \psi + (\lambda^2+\mu(x))\psi.
\end{equation*}
Given $\lambda>0$, let $\psi(x;e, \lambda)$ be the positive eigenfunction of the 1-periodic eigenvalue problem
\begin{equation}\label{eigenvalue}
\gamma(e,\lambda)\psi = \Delta \psi - 2\lambda e \cd \nabla \psi + (\lambda^2+\mu(x))\psi, \ \psi(x+e_i;e, \lambda)=\psi(x;e, \lambda), \ x\in\rn,
\end{equation}
with eigenvalue $\gamma(\lambda, e)$, normalized such that 
\begin{equation}\label{normalization}
\int\limits_{x\in\mathbb{T}^n} \psi(x;e,\lambda)=1.
\end{equation}
In the future we will simplify the notation by dropping both $e$ and $\lambda$, if this introduces no confusion.\\
Then $\bar{u}(t,x)$ solves \eqref{linear} if $c=\gamma(e,\lambda)/\lambda$, so we can in particular define the minimal speed of propagation to be 
\begin{equation}\label{minspeed}
c^*(e)=\inf_{\lambda>0} \frac{\gamma(e,\lambda)}{\lambda},
\end{equation}
achieved at $\lambda=\ls(e)$. This $\cs(e)$ is precisely the speed of the slowest pulsating front along $e$, and the exponential solution corresponding to it will play a crucial role in the analysis below.

\subsubsection{Propagation speed for $u_0\in C_c(\rn)$}
The question of spreading for localized data in higher dimensions was first addressed by Aronson and Weinberger in \cite{aw2} in 1978 in the homogeneous setting $\mu\equiv 1$. Here it was shown that the solutions propagate at a uniform speed $\cs=2$ in every direction, i.e.
\ea{ll}{
\ds \lim\limits_{t\to\infty}\sup\limits_{||x||\ge ct}u(t,x)= 0  & \ \text{ for } c>\cs, \\
\ds \lim\limits_{t\to\infty}\inf\limits_{||x||\le ct}u(t,x)= 1 & \ \text{ for } 0\le  c<\cs. 
}

In the periodic medium, the first results were due to Freidlin and G\"artner in \cite{fg} in 1979, who used probabilistic tools to determine that linear in time spreading along a fixed $e$ happens with speed $\ws(e)$, given by the so-called \textit{Freidlin-G\"artner formula}, 
\begin{equation}\label{wstar}
\ws(e)=\inf\limits_{e\cd\bar{e}>0}\frac{\cs(\bar{e})}{e\cd\bar{e}}.
\end{equation} 
This means that 
\begin{equation*}
\lim\limits_{t\to\infty} u(t, x+ct\,e) = \left\{\begin{array}{l}
0 \qquad \text{ if } c>\ws(e)\\
1 \qquad \text{ if } 0\le c < \ws(e) .
\end{array}  \right.
\end{equation*}
Note that $\ws$ in general can be much lower than $\cs$ (see figure \ref{fig1}). Moreover, the infimum for $\ws(e)$ is achieved at a \textit{minimizing direction} $e'$, which plays a principal role for the propagation in $e$ and to which we devote a full section below. 

\begin{figure}[h]
\centering
\includegraphics[width=0.52\textwidth]{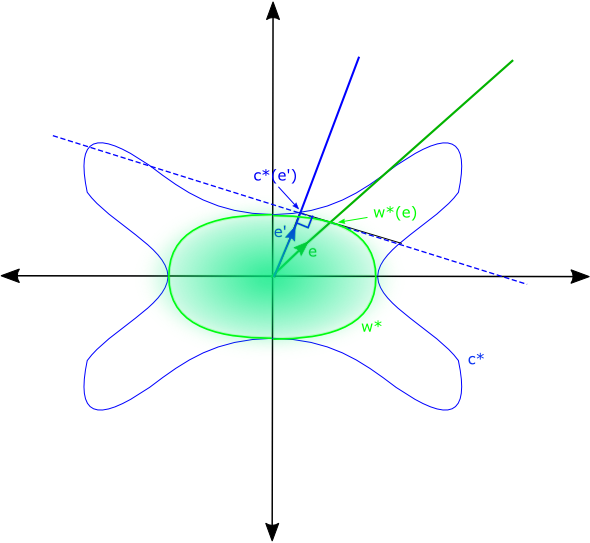} 
\caption[b]{\footnotesize 
The true propagation speed $\ws(e)$ may be much lower than the traveling front speed $\cs(e)$ and it is determined uniquely by its \textit{minimizing direction} $e'$. 
}
\label{fig1}
\end{figure}

This propagation result has been revisited many times in the past, with new methods of proof including ideas from large deviations in probability due to Freidlin in \cite{freidlin}, viscosity solutions due to Evans and Souganidis in \cite{evsoug}, monotone dynamical systems due to Weinberger in \cite{weinberger}, and PDE methods from H. Berestycki, Hamel and Nadin in \cite{bhn2}.

\subsection{Main results}

The first important result of the paper regards the \textit{minimizer} $e'$ of $e$, which for the sake of clarity we define properly here.

\begin{defn}\label{min.def}
Let $e$ be a unit direction. We call $e'\in\sn$ a minimizer for $e$ if $\ws(e)=\frac{\cs(e')}{e\cd e'}$, i.e. if it achieves the minimum in \eqref{wstar}.
\end{defn}
\begin{theorem} \label{unique}
Given any $e\in\sn$, the minimizer $e'\in\sn$ is unique.  Moreover, there is a one-to-one correspondence between $e\in\sn$ and $e'\in\sn$.
\end{theorem}

This result is surprising at first, because $e'$ itself simply minimizes a quantity that is not a priori convex, and therefore there is no reason why the it should be unique, and moreover that every direction is itself a minimizer. The key point in obtaining this was to find the exact way in which $e$ arises from $e'$, which to the best of our knowledge had not been known so far. Furthermore, this indicates that the profile of the pulsating speeds $\cs$ may be more regular than previously thought.\\

The second main result concerns the asymptotics of propagation of the solution $u$ and can be stated as follows:
\begin{theorem}\label{main}
Let $u(t,x)$ be the solution of \eqref{kpp}, with $u_0$ compactly supported in $\rn$, and satisfying $0\le u_0(x)\le 1, u_0\not\equiv 0$. Then for any direction $e\in\sn$ and $\epsilon>0$, there exists a time $t_0(\epsilon, e)$ and a length $L(\epsilon, e)$ such that 
\begin{equation}\label{lower}
u\ge 1-\epsilon \ \text{ for all } t>t_0(e,\epsilon)\ \text{ and } x=R\, e,\ R\in \left[0, \ws(e) t - \frac{n/2+1}{\ls(e') e\cd e'}\log t - L(e, \epsilon)  \right],
\end{equation}
and
\begin{equation}\label{upper}
u\le \epsilon \ \text{ for all } t>t_0(e,\epsilon)\ \text{ and } x=R\, e,\ R\in \left[\ws(e) t - \frac{n/2+1}{\ls(e') e\cd e'}\log t + L(e, \epsilon), \infty \right).
\end{equation}
\end{theorem}

Here $\ls(e')$ is the decay rate of the slowest exponential solution along $e'$ in the whole space, which will be defined in Section \ref{uniqueness}.

Note that while the analogous problem in $\R$ has been revisited many time in the past and is by now well understood, the only available result on sharp asymptotics of multi-directional spreading is due to G\"artner in \cite{gartner} in 1982. Here he uses probabilistic methods to control the first exit densities of Brownian paths from convex and concave curves to obtain an integral relation for the location of the fronts. This in turn determines the location up to a quantity of order $O(1)$. His result is however restricted to the homogeneous case $\mu(x)\equiv 1$, with $f(u)=u(1-u)$ and spherically symmetric initial data.

\subsubsection{Remarks}

First, the correction is fully determined by the minimizer, and as such it leads to the question of whether that or at least the quantity $\ls(e')e\cd e'$ is unique. This is indeed proved in Theorem \ref{unique} in the next section. More importantly, in the same theorem we prove the stronger statement: there is a one-to-one correspondence between directions $e\in\sn$ and minimizers $e'\in\sn$. Second, the factor $\frac{n}{2}+1$ is universal and does not depend on reaction term and the initial data. $\ws(e)$ and $\ls(e')e\cd e'$ on the other hand depend on $\mu(x)$ and $f'(0)$, and the only part depending on $u_0$ is the term of order $O(1)$, an observation that agrees with previous results in one dimension.

Third, the correction is dimension dependent. This suggests that the phenomena of multi-directional spreading can not be captured by the one dimensional pulsating fronts, which are crucial in the PDE proofs in $\R$. Indeed, working in $\rn$ is very different from $\R$, because there are few existing tools to deal with the difficulties arising from geometric considerations. Previous research focuses on homogeneous media (e.g. \cite{aw1},\cite{aw2},\cite{gartner}), front propagation (e.g. \cite{bh1},\cite{bh2}) and spreading in $\R$ (e.g. \cite{hnrr1},\cite{hnrr2}), and all of these are essentially one dimensional phenomena. Our main contribution has been to develop new methods for analyzing the combined effect of multiple directions in heat kernel estimates and in construction of super-solutions controlling $u$.

Finally, the higher the dimension, the greater the lagging behind the linear term. This can be understood at least in the homogeneous case from McKean's probabilistic interpretation of the equation via branching Brownian motion \cite{mck}.

\subsection{Strategy of the proof}

The plan for the rest of the paper is as follows: Section \ref{uniqueness} is devoted to the proof of the uniqueness result. This requires a connection of $e$ to $e'$ which follows from an analysis of the linearization of \eqref{kpp}, carried in the first part of the section

The idea of the proof for Theorem \ref{main}  is to compare $u$ in the regions where it is small to the solutions of the linearized equation. To make up for the growing effect of the linearized equation, we impose Dirichlet conditions in a properly moving boundary. Since the main blocking of the propagation in $e$ comes from $e'$, we first look at bounds for the linearized equation in the half-space moving linearly in the direction $e'$, with speed $\cs(e')$. However as shown in 
proposition \ref{linearbounds} these functions will still display growth in the direction of $e$. To deal with this, we introduce a logarithmic delay in the moving boundary that will in turn produce the correct logarithmic delay of the propagation in $e$. The results on the asymptotics of the solutions in both frames are stated in Section \ref{asymptotics}; however, as they are quite technical, their proofs are postponed till \ref{linearspace} and \ref{logspace}. Assuming these linear bounds, we prove Theorem \ref{main} in Section \ref{mainproof} by using the solutions of the Dirichlet problem to construct sub- and super-solutions that control the behavior of $u$ close to the spreading front.

The idea of comparing the equation to its linearization in a moving frame, together with the analysis, comes from \cite{hnrr2}, where it is used to find the logarithmic delay in $\R$. The main difficulty in our case arises due to the discrepancy between the two involved directions ($e$ and $e'$), as well as in the proof of the upper bound in the main theorem, in which we need to consider simultaneously all the solutions of the Dirichlet problem arising from multiple directions in $\rn$.

\subsection*{Acknowledgment}
The results presented here are part of the PhD thesis of the author, and would not have been possible without the invaluable guidance and encouragement of her advisor Lenya Ryzhik at Stanford University, and of Jean-Michel Roquejoffre. The author is immensely grateful to both of them. She would also like to thank Henri Berestycki for inviting her to spend the academic year 2017-2018 at the Centre d'Analyse et de Math\'ematique Sociales at EHESS, Paris, funded by the ERC under the European Union’s Seventh Framework Programme (FP/2007-2013) / ERC Grant Agreement 321186 - ReaDi, as well as acknowledge the hospitality of the Centre. 


\section{Uniqueness of the minimizer} \label{uniqueness}

\subsection{Linearization}
Let us focus first in understanding the linearized equation
\begin{equation}\label{linearequation}
v_t=\Delta v+\mu(x)v.
\end{equation}
As mentioned in Section \ref{periodicequation}, this equation has a family of global exponential solutions of the form $\e{-\lambda(x\cd e-ct)}\psi(x;e,\lambda)$ along each direction $e$. We are ultimately interested in studying the asymptotics of $v$ in a domain that moves in the direction of the minimizer $e'$, therefore it makes sense to write $v$ as a perturbation of the slowest exponential solution along $e'$:

\begin{equation}\label{perturbed}
v(t,x)=\e{-\ls(e')(x\cd e'-\cs(e')t)}\psi(x;e',\ls(e'))p(t,x;e')
\end{equation}
The $e'$'s in the arguments of $v, \psi$ and $p$ will be subsequently dropped to avoid lengthy notation. $p$ satisfies
\begin{equation}\label{linearframe}
p_t=\Delta p +2\left(\frac{\nabla \psi}{\psi} - \ls(e')e'\right)\cd \nabla p
\end{equation}
Denote the advective coefficient by
\begin{equation}\label{kappa}
2\left(\frac{\nabla \psi(x;e',\ls(e'))}{\psi(x;e',\ls(e'))} - \ls(e')e'\right) :=\kappa(x;e').
\end{equation}

\noindent The first step is to write \eqref{perturbed} in a self-adjoint form via the following crucial lemma.

\begin{lemma} \label{nu.existance}
There exist a unique positive, periodic scalar function $\nu(x;e')$ such that 
\begin{equation} \label{normalized}
\int\limits_{\mathbb{T}^n}\nu(x;e)dx=1,
\end{equation}
and a divergence-free, vector-valued $F(x;e')$ satisfying
\begin{equation}\label{drift}
\int\limits_{\mathbb{T}^n} F(x;e')dx = \ws(e)e,
\end{equation}
with the property that for any function $p(x)$,
\begin{equation} \label{selfadj}
\Delta p +\kappa(x;e')\cd \nabla p
=\frac{1}{\nu(x;e')}\nabla \cd (\nu(x;e')\nabla p)-\frac{F(x;e')}{\nu(x;e')}\cd\nabla p.
\end{equation}
\end{lemma}

This means that equation \eqref{linearframe} can be rewritten as
\begin{equation*}
p_t=\frac{1}{\nu(x)}\nabla\cd(\nu(x)\nabla p) - \frac{F(x)}{\nu(x)}\cd \nabla p
\end{equation*}
with a total weighted drift $F$ in the direction $e$, with the mass $\ws(e)$. This is very important because it indicates that when looking in a direction of the minimizer $e'$, any mass that survives moves towards $e$ and not $e'$, with speed precisely $\ws(e)t$. So while $e'$ comes naturally from $e$ by definition \ref{min.def}, here we see the way $e$ arises from $e'$ as the preferred direction of travel for the perturbations of the exponential solutions along $e'$.

\begin{proof}
Dropping the $e'$, \eqref{selfadj} is equivalent to 
\begin{equation}\label{nu}
\nabla\nu(x)-\nu(x)\kappa(x)=F(x),
\end{equation}
and the divergence-free condition on $F$ implies
\begin{equation} \label{kr}
\Delta\nu(x)-\nabla\cd(\nu(x)\kappa(x))=0.
\end{equation}
Denote by $L(e)$ the operator $\Delta+\kappa(x;e)\cd\nabla$, and $L^*(e)$ its adjoint -- these operators will play an important role in the future analysis. Then $\Delta \nu-\nabla\cd(\nu\kappa(x;e'))=L^*(e')\nu$. Now
the principal eigenvalue of $L$ is $0$, with corresponding eigenvector $1$, so Krein-Rutman theorem (see \cite{kr}) implies that $\nu(x)>0$ exists as the principal eigenvector of $L^*$, and it's unique up to a multiplicative factor. The normalization \eqref{normalized} assures further that the solution is unique.

Define $F(x)$ as in \eqref{nu}. The divergence-free condition is clear from \eqref{kr}, so all we need to evaluate is $\int F(x)\,dx=-\int\nu(x;e')\kappa(x;e')\,dx$.\\

Extend all the functions of $e\in\sn$ to all $e\in\rn$: let $e\in\rn\backslash\{0\}$. 
\begin{equation*}
\e{-\lambda(x\cd e-c(e,\lambda)t)}\psi(x;e,\lambda)=\e{-\lambda||e||(x\cd\frac{e}{||e||}-\frac{c(e,\lambda)}{||e||}t)}\psi(x;e,\lambda),
\end{equation*}
solves the linearized equation, and so does
\begin{equation*} 
\e{-\lambda||e||(x\cd\frac{e}{||e||}-c(\frac{e}{||e||},\lambda||e||)t)}\psi(x;\frac{e}{||e||},\lambda||e||),
\end{equation*}
so we can define
\begin{equation*} 
\psi(x;e,\lambda)=\psi(x;\frac{e}{||e||},\lambda||e||), \quad c(e,\lambda)=||e||c(\frac{e}{||e||},\lambda||e||).
\end{equation*}
$\psi(x;e,\lambda)$ solves 
\begin{equation}\label{psiinrn} 
\Delta\psi-2\lambda e\cd\nabla\psi+(\mu(x)+(\lambda||e||)^2)\psi=\gamma(e,\lambda)\psi.
\end{equation}
\begin{equation*} 
\gamma(e,\lambda)=\lambda c(e,\lambda)=\lambda ||e||c(\frac{e}{||e||},\lambda||e||)=\gamma(\frac{e}{||e||},\lambda||e||).
\end{equation*}
\begin{equation*} 
\cs(e)=\inf\limits_{\lambda>0}\frac{\gamma(e,\lambda)}{\lambda}=||e||\cs(\frac{e}{||e||}), \quad \text{achieved at} \quad \ls(e)=\frac{1}{||e||}\ls(\frac{e}{||e||}).
\end{equation*}
\begin{equation*} 
\ws(e)=\inf\limits_{e\cd\bar{e}>0}\frac{\cs(\bar{e})}{e\cd\bar{e}}=\inf\limits_{e\cd \bar{e}>0}\frac{||\bar{e}||\cs(\frac{\bar{e}}{||\bar{e}||})}{||e||\frac{e}{||e||}\cd||\bar{e}||\frac{\bar{e}}{||\bar{e}||}}=\frac{1}{||e||}\ws(\frac{\bar{e}}{||\bar{e}||}), \quad \text{achieved at} \quad e'=(\frac{e}{||e||})'.
\end{equation*}
First differentiate \eqref{psiinrn} in $e_j$ and evaluate it at $\lambda=\ls(e)$. Calling $\Psi_j(x,e)=\left.\frac{\partial}{\partial e_j}\psi(x;e,\lambda)\right| _{\lambda=\ls(e)}$, $\Gamma_j(e)=\left.\frac{\partial}{\partial e_j}\gamma(e,\lambda)\right| _{\lambda=\ls(e)}$, where $e_1,\ldots,e_n$ are the standard unit vectors in $\rn$, we obtain
\begin{equation*} 
\Delta\Psi_j-2\ls(e)e_j\cd\nabla\psi-2\ls(e)e\cd\nabla\Psi_j+ 2(\ls(e))^2e\cd e_j \psi +(\mu(x)+(\ls(e))^2||e||^2)\Psi_j=\Gamma_j\psi+\gamma(e,\ls(e))\Psi_j.
\end{equation*}
Define the functions $\chi_j(x;e):=\Psi_j(x;e)/\psi(x;e,\ls(e))$. Plugging this into the above equation and combining it with the equation for $\psi$, we get
\begin{equation*} 
\Delta\chi_j+2\frac{\nabla\psi}{\psi}\cd\nabla\chi_j-2\ls(e)e\cd\nabla\chi_j-2\ls(e)e_j\cd\frac{\nabla\psi}{\psi}+2(\ls(e))^2e\cd e_j=\Gamma_j.
\end{equation*}
This can be expressed more compactly as
\begin{equation} \label{chicomp}
L(e)\chi_j =\ls(e)\kappa(x;e) \cd e_j+\Gamma_j(e)
\end{equation}
This equality is true for every vector $e$, in particular for minimizers $e'$, so if we multiply it by $\nu(x)$ and integrate in the torus, we get
\begin{equation*} 
\int\limits_{\mathbb{T}^n}\nu(x;e')L(e')\chi_j(x;e')\,dx=\ls(e')\left(\int\limits_{\tn}\nu(x;e')\kappa(x;e')\,dx\right)\cd e_j+\Gamma_j(e'),
\end{equation*}
as $\int\nu=1$. But the left hand side is $0$ since $L^*\nu=0$, so to evaluate the integral we only need $\Gamma_j(e')$.\\
Now $\cs(e)=\frac{\gamma(e,\ls(e))}{\ls(e)}$, so differentiating this along $e_j$ gives
\begin{equation*} 
\frac{\partial}{\partial_{e_j}}\cs(e)
=\frac{\Gamma_j(e)}{\ls(e)}+\left(\frac{\partial\gamma(e,\ls(e))/\partial\lambda}{\ls(e)}-\frac{\gamma(e,\ls(e))}{(\ls)^2}\right)\frac{\partial}{\partial_{e_j}}\ls(e)
=\frac{\Gamma_j(e)}{\ls(e)}+\frac{\partial}{\partial_{\lambda}}\frac{\gamma(e,\ls(e))}{\ls(e)}\cd\frac{\partial}{\partial_{e_j}}\ls(e).
\end{equation*}
On the other hand, the expression in the formula for $\cs(e)$ achieves its minimum at $\ls(e)$, therefore its derivative with respect to $\lambda$ vanishes there. This makes the second term in the right hand-side $0$, so
\begin{equation*} 
\Gamma_j(e)=\ls(e)\frac{\partial}{\partial_{e_j}}\cs(e).
\end{equation*}
Finally, $\ws(e)=\min_{e\cd\bar{e}>0}\frac{\cs(\bar{e})}{e\cd\bar{e}}$, with the minimum achieved at $e'$, so $\nabla_{\bar{e}}(\frac{\cs(\bar{e})}{e\cd\bar{e}})$ must vanish at $e'$. Carrying out the differentiation we find that $\nabla \cs(e')=\frac{\cs(e')}{e\cd e'}e=\ws(e)e$, and hence $\Gamma(e')=\ls(e')\ws(e)e$. Going back to the original problem,
\begin{equation*} 
\int F(x)\,dx=-\int\nu(x;e')\kappa(x;e')\,dx=\frac{1}{\ls(e')}\Gamma(e')=\ws(e)e.
\end{equation*}
\end{proof}

\noindent  Note that now that we know $\Gamma_j$, \eqref{chicomp} gives the following equation
\begin{equation} \label{chi}
L(e')\chi(x;e')=\ls(e')\kappa(x;e')+\ls(e')\ws(e)e,
\end{equation}
which will be important again later in the proofs.

\subsection{Proof of Theorem \ref{unique}}

\textbf{Uniqueness of minimizer.} Let $e'$ be a minimizer for $e$. From the computations above,
\begin{equation} \label{gamma.at.min}
\nabla\gamma(e')=\Gamma(e')=\ls(e')\ws(e)e.
\end{equation}
Let $e^\perp$ be a vector orthogonal to $e$. Then for any minimizer $e'$,
$$\partial_{e^\perp}\gamma(e')=\nabla\gamma(e')\cd e^\perp=0.$$
If there are two minimizers $e_1$ and $e_2$ for $e$, then 
$$\partial_{e^\perp}\gamma(e_1)=\partial_{e^\perp}\gamma(e_2)=0.$$
It is easy to show that $\gamma$ is convex and super-linear at infinity, so $\partial_{e^\perp}\gamma=0$ must hold in the segment joining $e_1$ and $e_2$. But $\gamma$ is analytic in $e$, so it's Hessian can not have a $0$ determinant along a segment unless $\partial_{e^\perp}\gamma\equiv 0$, which contradicts the super-linear growth condition for $\gamma$. Hence $e_1=e_2$.\\

\noindent\textbf{Every direction is a minimizer.} Let $\bar{e}\in\sn$ be any direction. $\nabla\gamma(\bar{e})$ will be parallel to some unit direction, say $e$, with $e\cd\bar{e}>0$. If $e'$ minimizes $e$, then \eqref{gamma.at.min} holds, so $\nabla\gamma$ are parallel for $\bar{e}$ and $e'$. Then the same argument as the uniqueness one above gives $\bar{e}=e'$, so $\bar{e}$ is the minimizer for $e$.\\

\noindent$\mathbf{e'}$\textbf{ minimizes at most one }$\mathbf{e.}$ This follows from an elementary geometric argument. Indeed, let $e_1$ and $e_2$ have $e'$ as a minimizer, and let $OE_1,OE_2, OE'$ be the segments centered at the origin, of length $\ws(e_1), \ws(e_2), \cs(e')$, in the direction of $e_1, e_2, e'$ respectively (see figure \ref{min}). $O, E, E_i$ lie in a sphere $\Gamma_i$ centered at the middle of $OE_i$. Since
$$\ws(e_i)=\frac{\cs(e')}{e'\cd e_i}\le\frac{\cs(e)}{e\cd e_i}$$
for every $e\in\sn$, the segment of length $\cs(e)$ along $e$ will be outside the sphere $\Gamma_i$. This means that $\cs(e)$ lies outside of two different spheres $\Gamma_1$ and $\Gamma_2$ that intersect at $E'$, and it coincides with the intersection at $e'$ ($\cs(e')=OE'$), and hence, it can not be Lipschitz.

\begin{figure}[h]
\centering
\includegraphics[width=0.45\textwidth]{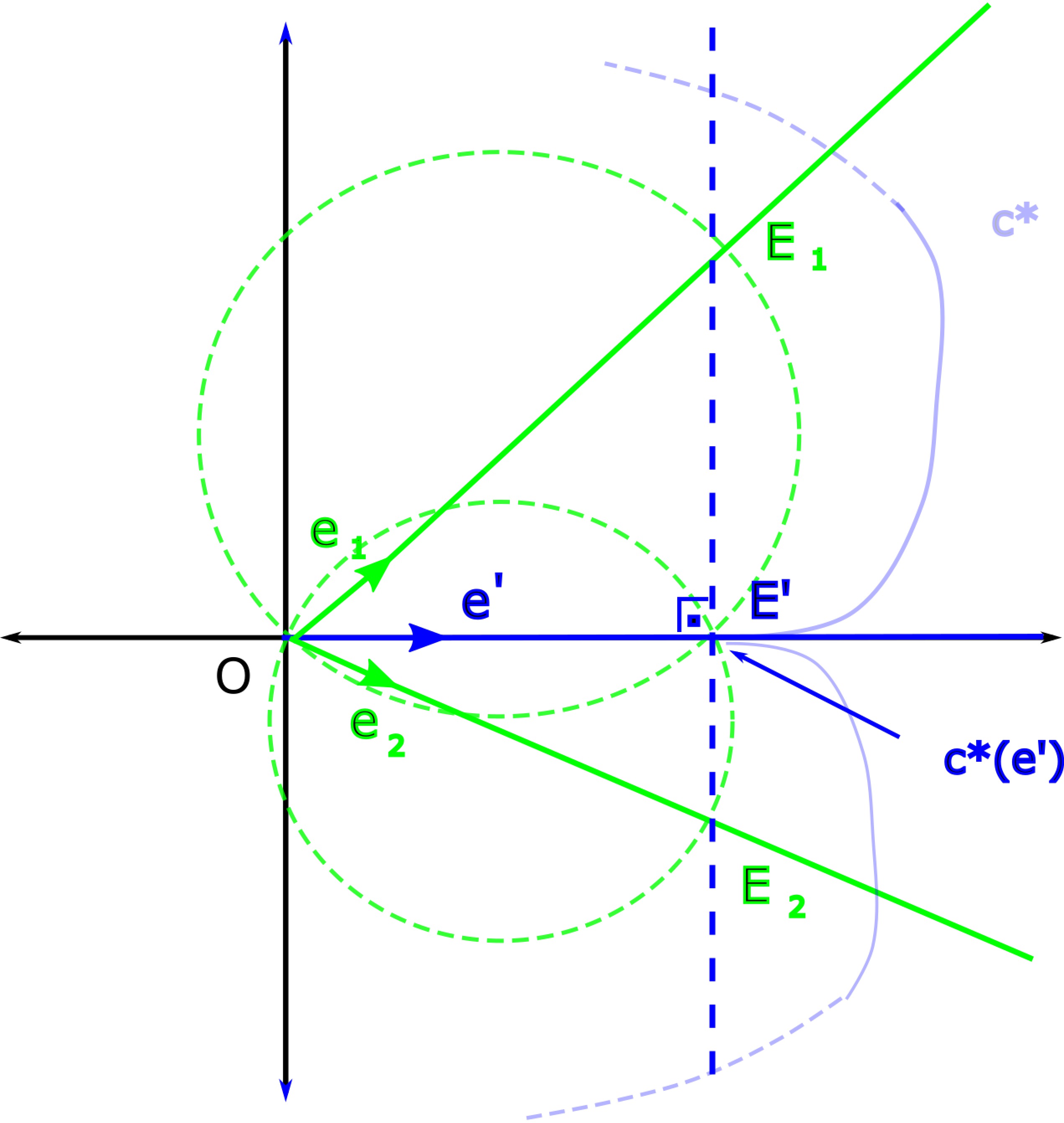}
\caption[b]{\footnotesize 
The surface $\cs:\sn\to\R^+$ lies outside of the spheres determined by $O, E_i$ and $E'$, and equals $|OE'|$ at $e'$, therefore it can not be Lipschitz at that point.
}
\label{min}
\end{figure}


\section{Asymptotics of the linearized equation: the main results}\label{asymptotics}

The proof of Theorem \ref{main} relies on the asymptotics of the Dirichlet problem for the linearized equation
\begin{equation*}
v_t=\Delta v+\mu(x)v
\end{equation*}
in an appropriately moving half-space. Because the main blocking of propagation in the direction $e$ comes from the minimizer $e'$, we look at the problem in the half-space moving along $e'$, first with speed $\cs(e')t$, and then in a slowed down frame $\cs(e')t-\alpha\log t$ with an appropriately chosen $\alpha$. We state the related results below here but postpone their proofs till Sections \ref{linearspace} and \ref{logspace} respectively.

\subsection{Linearly moving half-space}
So let $v(t,x):=v(t,x;e')$ solve
\eal{linearkilling}{ll}{
v_t=\Delta v + \mu(x)v, \quad & x\cd e'>\cs(e')t, \\
v(t,x)=0, \quad & x\cd e'=\cs(e')t, \\
v(0,x)=v_0(x), \quad & x\cd e'>0,
}
with $v_0\in C_c(\{x\cd e'>0\})$ non-negative and $v_0\not\equiv 0$. 
To avoid the potential growth coming from the lowest order term, we express $v$ as a perturbation of the exponential solutions just as in \eqref{perturbed}:
\begin{equation}\label{perburbation}
v(t,x)=\e{-\ls(e')(x\cd e'-\cs(e')t)}\psi(x)p(t,x)
\end{equation}
Following the result of Lemma \ref{selfadj}, $p$ satisfies
\eal{p.linear}{ll}{
p_t=\frac{1}{\nu(x)}\nabla\cd(\nu(x)\nabla p) - \frac{F(x)}{\nu(x)}\cd \nabla p, \quad & x\cd e'>\cs(e')t, \\
p(t,x)=0, \quad & x\cd e'=\cs(e')t, \\
p(0,x)=p_0(x)=v_0(x)\e{\ls(e')x\cd e'}\psi(x)^{-1}, \quad & x\cd e'>0,
}
Solutions of this equation satisfy the following bounds:

\begin{prop}\label{linearbounds}
Let $p(t,x;e')$ solve \eqref{p.linear}. There exist constants $C_u, C_l, \sigma, \rho$ and $T_0$ depending only on $v_0$ such that $p$ satisfies the upper bound
\begin{equation*}
|p(t,\ws(e)t\,e+x)|\le C_u\frac{x\cd e'}{t^{\frac{n}{2}+1}}\int\limits_{x\cd e'\ge 0}x\cd e'p_0(x)\,dx \quad \text{ for all } t\ge T_0,
\end{equation*}
and the lower bound
\begin{equation*}
|p(t,\ws(e)t\,e + \sigma \sqrt{t}\,e' +x)|\ge C_l\frac{1}{t^{\frac{n+1}{2}}} \quad \text{ for all } x\in B_{\rho\sqrt{t}}(0) 
\text{ and } t\ge T_0.
\end{equation*}
\end{prop}

Note that because we are interested in what happens while moving along $e$ with speed $\ws(e)$, we choose to observe the solution in a reference frame centered at the tip of $\ws(e)e$. Since $\psi$ is uniformly bounded above and below by positive constants,
\begin{equation*}
v(t,\ws(e)t\,e+x)\sim\e{-\ls(e')x\cd e'}p(t,\ws(e)t\,e+x).
\end{equation*}
Hence the proposition implies that $v$ in this reference frame decays exponentially at a rate proportional to the distance of the point to the boundary, while staying of $O\left(\frac{1}{t^{\frac{n+1}{2}}}\e{-\ls(e')\sigma\, e\cd e'\,\sqrt{t}}\right)$ when $x$ is in the $n-1$ dimensional ball $B(\sigma \sqrt{t}\,e,\rho\sqrt{t})$ perpendicular to $e'$, which we'll denote by $B_{\perp}(\sigma\sqrt{t}\,e,\sqrt{t}; e')$.

\begin{figure}[h]
\centering
\includegraphics[width=0.55\textwidth]{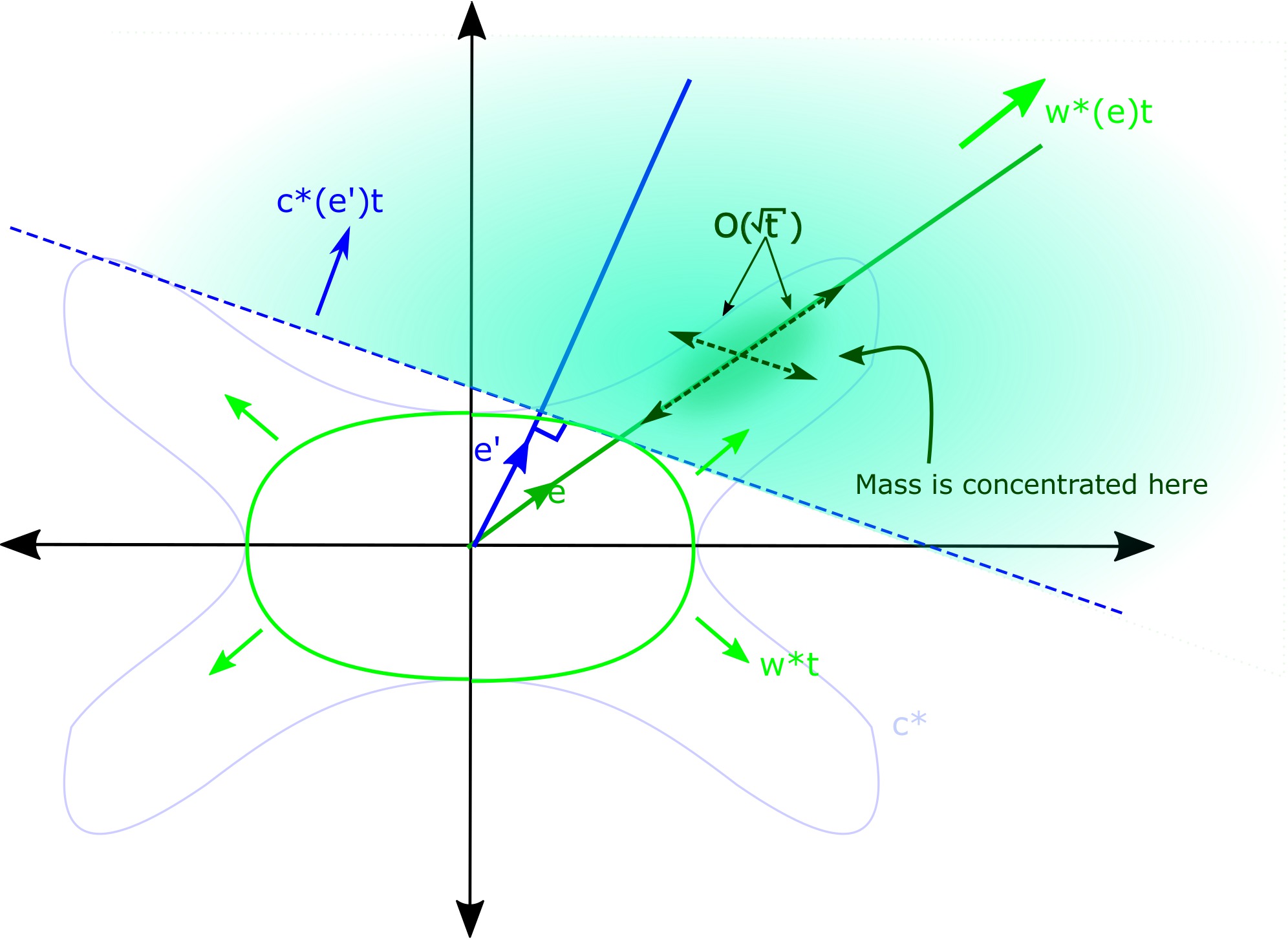}
\caption[b]{\footnotesize 
\textbf{(b)} In the problem with a Dirichlet boundary moving linearly in the direction of $e'$, the mass is concentrated in a region of diameter $\sqrt{t}$ along $e$.}
\label{fig2}
\end{figure}

To see how these bounds come about, we can think of the homogenized equation with $\nu(x)=1, F(x)=\ws(e)e$ and $p$ solving
\ea{ll}{
p_t=\Delta p - \ws(e)e\cd \nabla p, \quad & x\cd e'\ge \cs(e')t, \\
p(t,x)=0, \quad & x\cd e'= \cs(e')t,\\
p_0 \text{ compactly supported}, & x\cd e'>0.
}
Fixing the frame by $q(t,x):=p(t,\ws(e)t\,e+x)$ results in the heat equation in a half space
\ea{ll}{
q_t=\Delta q, \quad & x\cd e'>0, \\
q(t,x)=0, \quad & x\cd e'=0,\\
q_0 \text{ compactly supported}, & x\cd e'>0,
}
whose asymptotics can be computed easily to give
\begin{equation*}
p(t,\ws(e)t\,e+x)=q(t,x)\sim \frac{x\cd e'}{t^{\frac{n}{2}+1}}\e{-\frac{|x|^2}{4t}}.
\end{equation*}
From here we immediately notice the upper bound in the whole space and the lower bound on positions of size $\sqrt{t}e'$.

\subsubsection{Implications for $v$}

Using these bounds in 
\begin{equation*}
v(t,\ws(e)t\,e+x)=\e{-\ls(e')x\cd e'}\psi(\ws(e)t\,e+x)p(t,\ws(e)t\,e+x),
\end{equation*}
we get that 
\begin{equation*}
v(t,\ws(e)t\,e+x)\le C\frac{x\cd e'}{t^{\frac{n}{2}+1}}\e{-\ls(e')x\cd e'}
\end{equation*}
for the upper bound, and
\begin{equation*}
v(t,\ws(e)t\,e+x)\ge C'\frac{1}{t^{\frac{n+1}{2}}}\e{-\ls(e')\sigma\sqrt{t}}
\end{equation*}
at positions around $x\sim \sigma\sqrt{t}e'\,e$. Hence the $L^\infty$ norm of $v$ decays algebraically in time, with a portion of the mass accumulated in front of the direction of interest $e$. However our strategy was to compare the solution $u$ of \eqref{kpp} to $v$ solving the linear equation \eqref{linearkilling}, and the decay of $v$ makes it futile in construction of sub and super-solutions for $u$. In the following section we correct for the over-absorption by devising a new moving frame that is shifted logarithmically by $\alpha$, with $\alpha$ chosen such that the solutions to the new problem remain bounded by from above and below by a constant close to the boundary.

\subsection{Logarithmically delayed half-space}

\noindent Consider the equation
\eal{logkilling}{ll}{
\ds w_t=\Delta w+\mu(x)w, & x\cd e'>\cs(e')t-\alpha\log(\frac{t+T}{T}),\ t>0, \\
\ds w(t,x)=0, & x\cd e'=\cs(e')t-\alpha\log(\frac{t+T}{T}),\ t>0, \\
\ds w(0,x)=w_0(x), & x\cd e'>0,
}
with $w_0$ continuous, compactly supported, non-negative and $w_0\not\equiv 0$. Because the boundary here is quite complicated, it is convenient to pass first to logarithmically shifted new time variable that shifts the boundary to $\{x\cd e'>\cs(e')t\}$. So we start first by studying the asymptotic connection of this time variable to $t$, as well as the form of the resulting new linearized equation. We then write this function as a perturbation of the global exponential solution along $e'$, like we did in Section \ref{linearspace}. Here we notice that the resulting equation differs from \eqref{p.linear} by just an $O(\frac{1}{t})$ term in the coefficient of the time derivative, which allows us to adapt some of the earlier proofs in this setting. The main results are the estimates close to the boundary in Proposition \ref{p.bar.est}.

\subsubsection{Log-shifted time variable}\label{log.var}
In order to simplify the moving rate of the boundary and to be able to use the results from the previous chapter, we make the following convenient change of variables: let
\begin{equation*}
\bar{w}(t,x):=w(a(t),x),
\end{equation*}
where $a(t)$ is such that 
\begin{equation*}
\cs(e')a(t)-\alpha\log\left(\frac{a(t)+T}{T}\right):=\cs(e')t,
\end{equation*}
i.e.
\begin{equation}\label{varchange}
t=a(t)-\frac{\alpha}{\cs(e')}\log\left(\frac{a(t)+T}{T}\right).
\end{equation}
This function exists and is monotonic for $T$ large enough by the implicit function theorem, so $a'(t)>0$ for all $t$. The equation for the new function $\bar{w}$ then becomes
\eal{newvar}{ll}{
\frac{1}{a'(t)}\bar{w}_t=\Delta \bar{w}+\mu(x)\bar{w}, & x\cd e'>\cs(e')t,\ t>0, \\
\ds \bar{w}(t,x)=0, & x\cd e'=\cs(e')t,\ t>0, \\
\ds \bar{w}(0,x)=w_0(x), & x\cd e'>0.
}
Next express $\bar{w}$ in terms of the exponential solution along $e'$ as
\begin{equation*}
\bar{w}(t,x)=b(t)\e{-\ls(e')(x\cd e'-\cs(e')t)}\psi(x;e',\ls(e'))\bar{p}(t,x),
\end{equation*}
with $b(t)>0$ to be determined later. The equation for $\bar{p}$ becomes
\begin{equation}\label{pbar}
\ds\frac{1}{a'(t)}\bar{p}_t=\Delta\bar{p}+\kappa(x)\cd\nabla\bar{p}
+\left(\ls(e')\cs(e')\left(1-\frac{1}{a'(t)}\right)-\frac{1}{a'(t)}\frac{b'(t)}{b(t)} \right)\bar{p}
\end{equation}
for $x\cd e'>\cs(e')t$, and $\bar{p}(t,x)=0$ when $x\cd e'=\cs(e')t$. The function $$\kappa(x)=\kappa(x;e')=2\left(\frac{\nabla \psi}{\psi}-2\ls(e')e'\right)$$
is defined in \eqref{kappa}. We will choose $b(t)$ in a way that eliminates the lower order term in \eqref{pbar}. For this we need to understand $a'(t)$, so differentiate \eqref{varchange} in $t$ to obtain
\begin{equation*}
\alpha'(t)=1+\frac{\alpha}{\cs(e')(a(t)+T)-\alpha}.
\end{equation*}
Then
\begin{equation*}
\frac{b'(t)}{b(t)}=(a'(t)-1)\ls(e')\cs(e')=\frac{\alpha\ls(e')\cs(e')}{\cs(e')(a(t)+T)-\alpha}=\frac{\alpha\ls(e')}{t+T}-\frac{\alpha^2\lambda(e')}{\cs(e')}\frac{\log(\frac{a(t)+T}{T})-1}{(t+T)(a(t)+T-\frac{\alpha}{\cs(e')})},
\end{equation*}
and since $a(t)=t+O(\log t)$, 
\begin{equation*}
\frac{b'(t)}{b(t)}=\frac{\alpha\ls(e')}{t+T}+O\left(\frac{1}{(t+T)^{3/2}}\right).
\end{equation*}
This implies that 
\begin{equation}\label{b-asymptotics} 
b(t)=\e{\alpha\ls(e')\log(t+T)+O(t^{-1/2})}=(t+T)^{\alpha\ls(e')}(1+O(t^{-1/2})).
\end{equation}
On the other hand,
\begin{equation*}
\ds\frac{1}{a'(t)} =\frac{\cs(e')(a(t)+T)-\alpha}{\cs(e')(a(t)+T)}=1-\beta(t),
\end{equation*}
where
\begin{equation*} 
\begin{array}{ll}
\ds\beta(t) &\ds =\frac{\alpha}{\cs(e')(a(t)+T)} =\frac{\alpha}{\cs(e')(t+T)}-\frac{\alpha}{\cs(e')}\frac{a(t)-t}{(t+T)(a(t)+T)}\\ & \ds=\frac{\alpha}{\cs(e')(t+T)}-\frac{\alpha^2}{\cs(e')^2}\frac{\log(a(t)+T)}{(t+T)\left(t+\frac{\alpha}{\cs(e')}\log\left(\frac{a(t)+T}{T}\right)+T\right)}.
\end{array}
\end{equation*}
So $\beta(t)$ satisfies
\begin{equation} \label{beta.asymp}
\beta(t)=\frac{\alpha}{\cs(e')t}+O\left(\frac{1}{t^{3/2}} \right)\ \text{ as } \ t\to\infty, \quad \text{ and } \ |\beta(t)|\le\frac{C}{t}, \ |\beta'(t)|\le\frac{C}{t^2} \ \text{ for } t\ge T_0.
\end{equation}
In the other hand, since $a(t)\ge 0$ we may choose $T$ large enough so that $|\beta(t)|\le\frac{1}{2}$ for all $t$. With these asymptotics for $1/a'(t)$ $\bar{p}$ turns out to satisfy similar bounds to those of $p$ in Proposition \ref{linearbounds}. These are stated in the following proposition, whose proof will be postponed to the last section of this chapter.
%
\begin{prop}\label{p.bar.est}
Let $\bar{p}$ solve 
\begin{equation}\label{pbareq}\left\{\hspace{20pt} 
\begin{array}{ll}
\ds(1-\beta(t))\bar{p}_t=\Delta\bar{p}+\kappa(x)\cd\nabla\bar{p},
 & x\cd e'>\cs(e')t, \\
\bar{p}(t,x)=0,  & x\cd e'=\cs(e')t, \\
\bar{p}(0,x)=p_0(x)=w_0(x)\e{\ls(e')x\cd e'}\psi(x)^{-1}, & x\cd e'>0,
\end{array}\right.
\end{equation}
Then there exist positive constants $M$ and $T_0$ such that for all $t\ge T_0$ and $x\in B_{\sigma\sqrt{t}}(0)\cap\{x\cd e'\ge 0\}$,
\begin{equation*}
M^{-1}\frac{x\cd e'}{t^{\frac{n}{2}+1}}\le \bar{p}(t,\ws(e)t\,e+x)\le M\frac{x\cd e'}{t^{\frac{n}{2}+1}}.
\end{equation*}
\end{prop}

\subsubsection{Implications for $w$: the correct log shift}
We will now use this proposition to infer about the bounds of $w$ at locations of order $\ws(e)t-\frac{\alpha}{\ls(e')e\cd e'}\log t$. Going back to the original variables, using the explicit form of $\bar{w}, a(t)$ and $b(t)$,
\begin{equation*}
\begin{array}{l}
w(t,x)\ds=\bar{w}(a^{-1}(t),x)=\bar{w}\left(t-\frac{\alpha}{\cs(e')}\log\left(\frac{t+T}{T}\right),x \right)\\
\ds=(t+T)^{\alpha\ls(e')}\left(1+O\left(\frac{1}{\sqrt{t}}\right)\right) \e{-\ls(e')\left(x\cd e'-\cs(e')t+\alpha\log\left(\frac{t+T}{T}\right)\right)} \psi(x) \bar{p}\left(t-\frac{\alpha}{\cs(e')}\log\left(\frac{t+T}{T}\right),x\right)\\
\ds=\left(1+O\left(\frac{1}{\sqrt{t}}\right)\right) \e{-\ls(e')\left(x\cd e'-\cs(e')t\right)} \bar{p}\left(t-\frac{\alpha}{\cs(e')}\log\left(\frac{t+T}{T}\right),x\right),
\end{array}
\end{equation*}
where we used the fact that $\psi(x)$ is uniformly positive. Increasing $M$ if necessary, Proposition \ref{p.bar.est} gives the following estimates for $\bar{p}$ at the logarithmically shifted time $t-\frac{\alpha}{\cs(e')}\log\left(\frac{t+T}{T}\right)$:
\begin{equation*}
M^{-1}\frac{x\cd e'}{(t+T)^{\frac{n}{2}+1}}\le \bar{p}\left(t-\frac{\alpha}{\cs(e')}\log\left(\frac{t+T}{T}\right),\left(\ws(e)t-\frac{\alpha}{e\cd e'}\log\left(\frac{t+T}{T}\right)\right)\,e+x\right)\le M\frac{x\cd e'}{(t+T)^{\frac{n}{2}+1}},
\end{equation*}
for $t\ge T_0$ and $x\in B_{\sigma\sqrt{t}/2}(0)\cap\{x\cd e'\ge 0\}$. Increasing $M$ even more and shifting $x$ by $L\,e$ to account for $O(1)$ corrections gives 
\begin{equation*}
M^{-1}\frac{x\cd e'}{t^{\frac{n}{2}+1}}\le \bar{p}\left(t-\frac{\alpha}{\cs(e')}\log\left(\frac{t+T}{T}\right),\left(\ws(e)t-\frac{\alpha}{e\cd e'}\log t+L\right)\,e+x\right)\le M\frac{x\cd e'}{t^{\frac{n}{2}+1}}.
\end{equation*}
This implies that at the same locations, $w$ satisfies
\begin{equation*}
\begin{array}{l}
\ds w\left(t,\left(\ws(e)t-\frac{\alpha}{e\cd e'}\log t+L\right)\,e+x\right)\\
\ds\ge C \e{-\ls(e')x\cd e'}\e{-\ls(e')L e\cd e'}\e{\ls(e')\alpha\log t}\bar{p}\left(t-\frac{\alpha}{\cs(e')}\log\left(\frac{t+T}{T}\right),\left(\ws(e)t-\frac{\alpha}{e\cd e'}\log t+L\right)\,e+x\right)\\
\ds\ge c_L\e{-\ls(e')x\cd e'}t^{\ls(e')\alpha-(\frac{n}{2}+1)}x\cd e'\ge c_Lt^{\ls(e')\alpha-(\frac{n}{2}+1)} =c_L
\end{array}
\end{equation*}
when
$$\alpha=\frac{n/2+1}{\ls(e')}$$
for all $x\in B_{\epsilon}(0)$. Similarly for this $\alpha$ and $x$ in the same range,
\begin{equation*}
w\left(t,\left(\ws(e)t-\frac{\alpha}{e\cd e'}\log t+L\right)\,e+x\right)\le C_L.
\end{equation*}
This means that for this particular shift, $w$ remains bounded above and below by a constant at positions of order $1$ from the boundary, in the direction of $e$. The estimate above can be summarized in the following corollary.

\begin{cor}\label{order.one}
Given $L>L_0$, there exist $\epsilon>0$ and $T>0$ such that for all $t\ge T$ and $x\in B_{\epsilon\sqrt{t}}(0)$,
\begin{equation*}
c_L\le w\left(t,\left(\ws(e)t-\frac{n/2+1}{\ls(e')e\cd e'}\log t+L\right)\,e+x\right)\le C_L.
\end{equation*}
\end{cor}

To finish the proof of the upper bound of the main theorem in the last chapter, we will need a decay estimate $w(t,x)$ in the whole half-space. For this, recall  the heat kernel estimates of Norris in the entire space \eqref{heat.rn}. Applying them to $\bar{p}$ implies that there exists $K>0$ such that for $t\ge 1$,  
\begin{equation*}
\bar{p}(t,x+\ws(e)t\,e)\le \frac{K}{t^{\frac{n}{2}}}\e{-\frac{|x|^2}{Kt}}.
\end{equation*}
This implies for $w$ that 
\begin{equation*}
\begin{array}{l}
\ds w\left(t,\left(\ws(e)t-\frac{n/2+1}{\ls(e')e\cd e'}\log t\right)\,e+x\right)\\
\ds\le C \e{-\ls(e')x\cd e'}\e{\ls(e')\alpha\log t}\bar{p}\left(t-\frac{\alpha}{\cs(e')}\log t,\left(\ws(e)t-\frac{\alpha}{e\cd e'}\log t\right)\,e+x\right)\\
\ds= C\e{-\ls(e')x\cd e'}t^{n/2+1} \bar{p}\left(t-\frac{\frac{n}{2}+1}{\ls(e')\cs(e')}\log t,\left(\ws(e)t-\frac{n/2+1}{\ls(e')e\cd e'}\log t\right)\,e+x\right)\\
\ds=C\e{-\ls(e')x\cd e'}t^{n/2+1} \bar{p}\left(t-\frac{\frac{n}{2}+1}{\ls(e')\cs(e')}\log t,\ws(e)\left(t-\frac{n/2+1}{\ls(e')\cs(e')}\log t\right)\,e+x\right)\\
\ds\le C\e{-\ls(e')x\cd e'}t^{n/2+1} \frac{1}{t^{\frac{n}{2}}}\e{-\frac{|x|^2}{Kt}} =Ct\e{-\ls(e')x\cd e'}\e{-\frac{|x|^2}{Kt}}.
\end{array}
\end{equation*}
So for all $x$ such that $x\cd e'>0$ we have the following Gaussian bounds 
\begin{equation}\label{gaussian.decay}
w\left(t,\left(\ws(e)t-\frac{n/2+1}{\ls(e')e\cd e'}\log t\right)\,e+x\right) \le Ct\e{-\ls(e')x\cd e'}\e{-\frac{|x|^2}{Kt}}.
\end{equation}

\noindent On the other hand, the linear bound 
\begin{equation*}
\bar{p}(t,x+\ws(e)t\,e)\le C\frac{x\cd e'}{t^{\frac{n}{2}+1}}
\end{equation*}
implies
\begin{equation*}
\begin{array}{l}
\ds w\left(t,\left(\ws(e)t-\frac{n/2+1}{\ls(e')e\cd e'}\log t\right)\,e+x\right)\\
\ds\le C\e{-\ls(e')x\cd e'}t^{n/2+1} \bar{p}\left(t-\frac{\frac{n}{2}+1}{\ls(e')\cs(e')}\log t,\ws(e)\left(t-\frac{n/2+1}{\ls(e')\cs(e')}\log t\right)\,e+x\right)\\
\ds\le C\e{-\ls(e')x\cd e'}t^{n/2+1}\frac{x\cd e'}{t^{\frac{n}{2}+1}}= C x\cd e'\e{-\ls(e')x\cd e'}.
\end{array}
\end{equation*}
So again, for all $x$ such that $x\cd e'>0$, $w$ has the exponential decay
\begin{equation}\label{exp.decay}
w\left(t,\left(\ws(e)t-\frac{n/2+1}{\ls(e')e\cd e'}\log t\right)\,e+x\right) \le C x\cd e'\e{-\ls(e')x\cd e'}.
\end{equation}

\section{Proof of the main theorem assuming the linear bounds}\label{mainproof}

This section is dedicated to proving the result of Theorem \ref{main}. In the first part we construct a sub-solution for $u$ using the solution $w$ to the linearized Dirichlet problem in the logarithmically shifted domain, together with the bounds from Corollary \ref{order.one}. The second part is concerned with the proof of the upper bound. This is in turn more complicated and requires combining the effect of multiple directions simultaneously, therefore the super-solution we construct is the result of adding the solutions of the linearized problem arising from finitely many $e$'s.

\subsection{The lower bound}

We want to show show that fixing $e\in\sn$, for every $\epsilon>0$, there exist $t_0$ and $L$ depending on $\epsilon$ and $e$ such that 
\begin{equation*}
u\ge 1-\epsilon \ \text{ for all } t>t\ \text{ and } x=R\, e,\ R\in \left[0, \ws(e) t - \frac{n/2+1}{\ls(e') e\cd e'}\log t - L  \right].
\end{equation*}
Let the initial data $u_0$ be supported on a ball of radius $R_0$ centered at the origin, and consider the solution $w$ of the linearized Dirichlet problem \eqref{logkilling} on the domain
$$x\cd e'>\cs(e')t-\frac{n/2+1}{\ls(e')}\log\left(\frac{t+T}{T}\right)-R_0,$$
with initial data $w_0=u_0$. Then from chapter \ref{logspace} we deduce that $w$ is uniformly bounded above. Indeed $\bar{p}$ satisfies 
\begin{equation*}
\bar{p}(t,\ws(e)t\,e+x)\le C\frac{x\cd e'}{t^{\frac{n}{2}+1}},
\end{equation*}
so similarly to the proof of Corollary \ref{order.one},
\begin{equation*}
\begin{array}{l}
\ds w\left(t,\left(\ws(e)t-\frac{n/2+1}{\ls(e')e\cd e'}\log t\right)\,e+x\right)\\
\ds\le C\e{-\ls(e')x\cd e'}t^{n/2+1} \bar{p}\left(t-\frac{\frac{n}{2}+1}{\ls(e')\cs(e')}\log t,\left(\ws(e)t-\frac{n/2+1}{\ls(e')e\cd e'}\log t\right)\,e+x\right)\\
\ds=C\e{-\ls(e')x\cd e'}t^{n/2+1} \bar{p}\left(t-\frac{\frac{n}{2}+1}{\ls(e')\cs(e')}\log t,\ws(e)\left(t-\frac{n/2+1}{\ls(e')\cs(e')}\log t\right)\,e+x\right)\\
\ds\le C\e{-\ls(e')x\cd e'}t^{n/2+1} \frac{x\cd e'}{t^{\frac{n}{2}+1}} =Cx\cd e'\e{-\ls(e')x\cd e'}\\
\ds\le C\sup\limits_{x\cd e'>0} x\cd e'\e{-\ls(e')x\cd e'}\le M.
\end{array}
\end{equation*}
Therefore multiplying $w$ by a small constant $a$ if necessary and using the bounds of $f(s)$ for small $s$ we can guarantee that $aw(t,x)$ is a sub-solution for \eqref{kpp} and satisfies $aw(t,x)\le\delta$ for all $t,x$. So
\begin{equation*}
u(t,x)\ge aw(t,x) \quad \text{ for all } t\ge 0, \ x\cd e'>\cs(e')t-\frac{n/2+1}{\ls(e')}\log\left(\frac{t+T}{T}\right)-R_0.
\end{equation*}
From Corollary \ref{order.one},
\begin{equation*}
\begin{array}{c}
\ds u\left(t,\left(\ws(e)t-\frac{n/2+1}{\ls(e')e\cd e'}\log t+L\right)\,e+x\right)\ge aw\left(t,\left(\ws(e)t-\frac{n/2+1}{\ls(e')e\cd e'}\log t+L\right)\,e+x\right)\\
\ds\ge ac_L=\delta'
\end{array}
\end{equation*}
for all $t\ge t_0$ and $x\in B_\epsilon(0)$. \\

Note that by parabolic Harnack inequality, there exists a time $s_0$ such that if $u(s,x)\ge\delta'$ for all $x\in B_\epsilon(0)$, then $u(t,x)\ge 1-\epsilon$ for all $x$ in the same $B_\epsilon(0)$ and $t\ge s+s_0$. Since 
\begin{equation*}
u\left(s,\left(\ws(e)s-\frac{n/2+1}{\ls(e')e\cd e'}\log s+L\right)\,e+x\right)\ge\delta', \quad \forall s\ge t_0, \ x\in B_\epsilon(0),
\end{equation*}
this implies that 
\begin{equation*}
u\left(t,\left(\ws(e)s-\frac{n/2+1}{\ls(e')e\cd e'}\log s+L\right)\,e+x\right)\ge 1-\epsilon, \quad \forall s\ge t_0,\ t\ge s+s_0, \ x\in B_\epsilon(0),
\end{equation*}
or equivalently for all $t_0\le s\le t-t_0$. Hence, $u$ is bounded from below by $1-\epsilon$ for all $x=Re$, where 
\begin{equation*}
\ws(e)t_0-\frac{n/2+1}{\ls(e')e\cd e'}\log t_0+L \le R\le \ws(e)(t-t_0)-\frac{n/2+1}{\ls(e')e\cd e'}\log (t-t_0)+L,
\end{equation*}
for all $t$ larger than some $t_0'>t_0$, chosen such that the quantity on the left in the above equation is always smaller than the quantity on the right. Choosing $L_0$ such that 
\begin{equation*}
\ws(e)t-\frac{n/2+1}{\ls(e')e\cd e'}\log t -L_0\le \ws(e)(t-t_0)-\frac{n/2+1}{\ls(e')e\cd e'}\log (t-t_0)+L,
\end{equation*}
we obtain the lower bound for all $C_0\le R\le \ws(e)t-\frac{n/2+1}{\ls(e')e\cd e'}\log t -L_0$. To extend the result to $R\in[0,C_0]$, note that by parabolic regularity, $u(1,x)\ge\alpha>0$ on $B_{C_0}(0)$, therefore $u(t,x)\ge\alpha>0$ on $B_{C_0}(0)$ for all $t\ge t_1$ for some $t_1\ge 1$. So increasing $t_0$ if necessary to ensure $t_0\ge t_1$, we obtain $u(t,Re)\ge 1-\epsilon$ for $R\in \left[0, \ws(e) t - \frac{n/2+1}{\ls(e') e\cd e'}\log t - L_0 \right]$ for all $t\ge t_0$. This concludes the proof of the lower bound.

\subsection{The upper bound}

In this section we establish the existence of $t_0$ and $L_0$ depending on $\epsilon$ (but not on $e$) such that
\begin{equation*}
u\le \epsilon \ \text{ for all } t>t_0\ \text{ and } x=R\, e,\ R\in \left[\ws(e) t - \frac{n/2+1}{\ls(e') e\cd e'}\log\left(\frac{t+T}{T}\right) + L, \infty \right).
\end{equation*}
The solution $w(t,x)$ of the linearized equation \eqref{logkilling}, extended by $0$ to the whole $\rn$, comprises a super-solution for the equation \eqref{kpp} in the logarithmically shifted half-space. From the bounds in Corollary \ref{order.one}, we have that multiplying it by some large constant $A$,
\begin{equation*}
2\le Aw\left(t,\left(\ws(e)t-\frac{n/2+1}{\ls(e')e\cd e'}\log\left(\frac{t+T}{T}\right)+L\right)\,e+x\right)\le M
\end{equation*}
for all $x\in B_{\epsilon\sqrt{t}}(0)$ 
and $t>t_0$. Call this
$$W_e(t,x):=Aw(t,x;e)$$
to indicate the function arising from the direction $e$. Let $\mathcal{W}_t$ be the $n-1$-dimensional convex manifold 
\begin{equation*}
\mathcal{W}_t=\left\{Re: \ e\in\sn, \ R=\ws(e)t-\frac{n/2+1}{\ls(e')e\cd e'}\log\left(\frac{t+T}{T}\right)+L \right\}.
\end{equation*}
The discussion above implies that at any point of $\mathcal{W}_t$, $W_e\ge 2$ for the direction $e$ corresponding to that point, so the hope is to 'patch' all the $W_e$'s together to obtain a function that, when intersected with the constant solution $1$, will produce a super-solution for $u$ in the entire $\rn$. However it is impossible to deal with an infinity of functions all at once, therefore instead we approximate $\mathcal{W}_t$ by a convex polyhedral surface $\mathcal{W}_t'$ that circumscribes and is tangent to $\mathcal{W}_t$. We choose this object in such a way that every side corresponds to a surface of diameter of order $O(\sqrt{t})$
in the direction of some $e_i\in\sn$, where $W_{e_i}(t,x)\ge 2$. This means that there are $N(t)$ directions $e_1,e_2,\ldots, e_{N(t)}$ associated to $\mathcal{W}_t'$, and at every point of this polyhedron, at least one of the solutions $W_{e_i}$ is bounded from below by $2$ (See Figure \ref{supersol} for a $2$ dimensional picture). Moreover, if we are interested in the behavior along a specific $e\in\sn$, then the $e_i$'s are chosen so that $e$ is one of them.

\begin{figure}[h]
\centering
\includegraphics[width=0.55\textwidth]{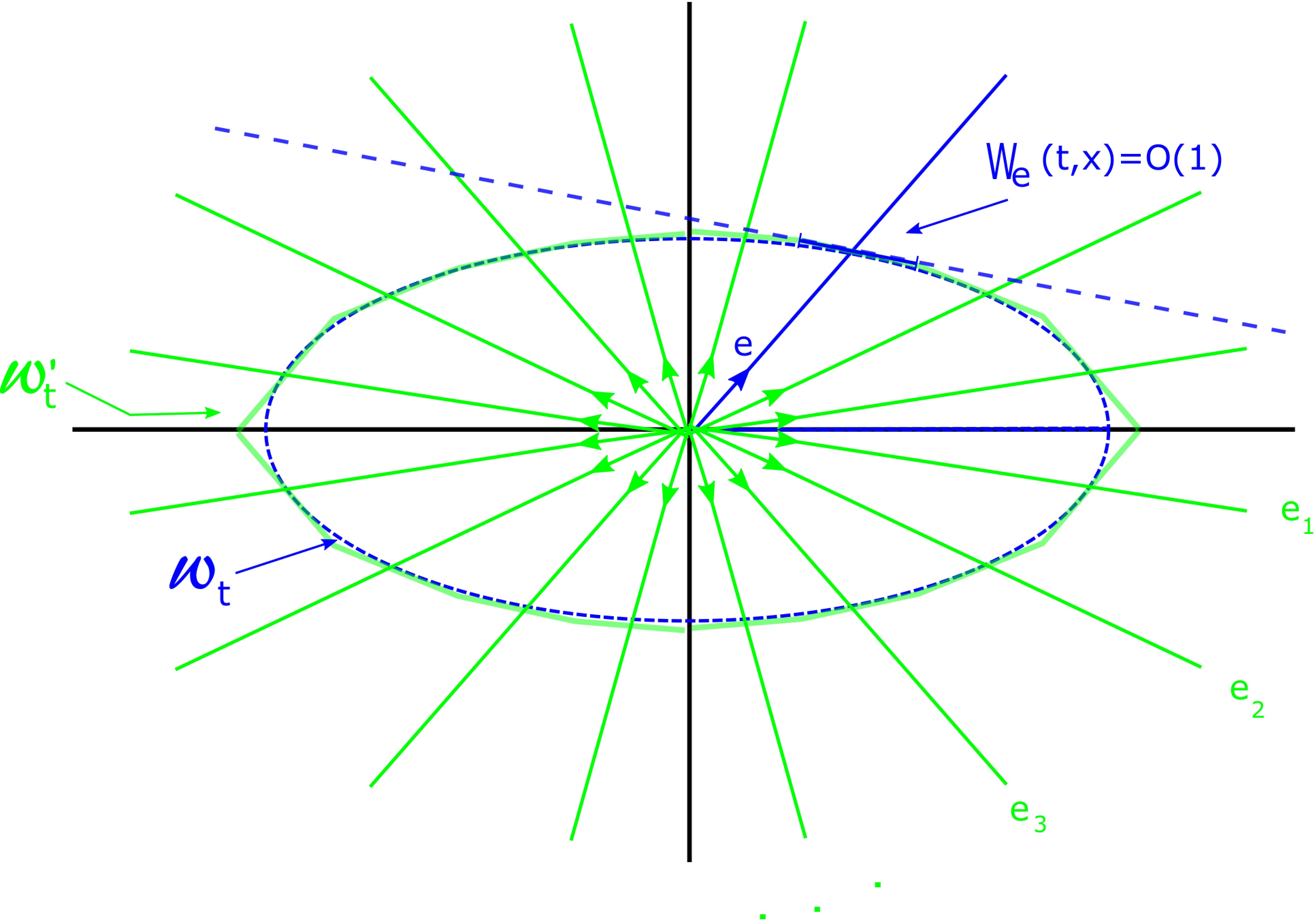} 
\caption[b]{\footnotesize Adding the functions coming from directions $e_1, e_2,\ldots, e_{N(t)}$. For a typical $e$ (blue) among these, $W_e(t,x)\ge 2$ on a segment of length $O(\sqrt{t})$
at position $\ws(e) t - \frac{n/2+1}{\ls(e') e\cd e'}\log\left(\frac{t+T}{T}\right) + L$. Patching all these segments creates the polygonal path $\mathcal{W}'_t$ (green) tangent to $\mathcal{W}_t$ (blue).}
\label{supersol}
\end{figure}

\noindent Let 
\begin{equation*}
W(t,x)=W_{e_1}(t,x)+W_{e_2}(t,x)+\ldots+W_{e_{N(t)}}(t,x).
\end{equation*}
Since each $W_{e_i}$  is a super-solution for the Fisher-KPP equation, so is their sum $W$, and moreover, 
$$W(t,x)\ge 2 \quad \text{ for all } x\in\mathcal{W}_t'.$$

To guarantee that the sides of $\mathcal{W}_t'$ remain of order $O(\sqrt{t})$
, we need to adjust the number of directions considered with respect to the time of interest. If we want information at time $\tau$, then we patch $N(\tau)$ directions so that at time $2\tau$, this condition is satisfied. Therefore to not abuse notation, we should denote by $\mathcal{W}_{t,\tau}'$ and $W(t,x;\tau)$ the polyhedral surface and function coming from these specific $N(\tau)$ directions, at time $t$. However, we will keep this in mind, but stick to the previous notation for the sake of clarity.\\

Now consider
\begin{equation*}
U(t,x) = \left\{\begin{array}{cl}
1, &\quad \text{ if } x \text{ lies inside } \mathcal{W}'_t, \\
\min\{1, W(t,x)\} &\quad \text{ if } x \text{ lies outside } \mathcal{W}'_t .
\end{array}  \right.
\end{equation*}
By increasing $L$ if necessary to keep the support of $u_0$ inside $B_L(0)$, we get that $U(0,x)\ge u_0(x)$ on $\rn$, therefore $U$ is a super-solution for \eqref{kpp}. Hence, to prove the upper bound in Theorem \ref{main}, it is enough to control $W$ at a fixed distance from $\mathcal{W}'_t$.

The first step is to find $N(t)$. Recall that this number was chosen so that $\mathcal{W}'_{2t}$ has $N(2t)$ sides of order $O(\sqrt{t})$
. Note that $\mathcal{W}'_t$ has a linear diameter, therefore its surface is of order $O(t^{n-1})$. Since the area of each surface itself is of order $O(\sqrt{t}^{n-1})$
, then $N(2t)O(\sqrt{t}^{n-1})=O(t^{n-1})$,
so
$$N(t)=O(t^{\frac{n-1}{2}}).$$ 

Fix $e\in\sn$. To control $W$ at locations of size larger than $\ws(e)t-\frac{n/2+1}{\ls(e')e\cd e'}\log\left(\frac{t+T}{T}\right)+L$ along $e$, let us recall the bounds \eqref{gaussian.decay} and \eqref{exp.decay} on $W_{e_i}$'s: for all $x$ satisfying $x\cd e_i'>0$ and $t\ge t_0$, 
\begin{equation*}
W_{e_i}\left(t,\left(\ws(e_i)t-\frac{n/2+1}{\ls(e_i')e\cd e_i'}\log\left(\frac{t+T}{T}\right)\right)\,e_i+x\right) \le Ct\e{-\lambda x\cd e_i'}\e{-\frac{|x|^2}{Kt}}
\end{equation*}
and
\begin{equation*}
W_{e_i}\left(t,\left(\ws(e_i)t-\frac{n/2+1}{\ls(e_i')e_i\cd e_i'}\log\left(\frac{t+T}{T}\right)\right)\,e_i+x\right) \le C x\cd e_i'\e{-\lambda x\cd e_i'},
\end{equation*} 
where $\lambda=\sup\limits_{e'\in\sn}\ls(e')$. Then at the location $$\left(\ws(e)t-\frac{n/2+1}{\ls(e)e\cd e'}\log\left(\frac{t+T}{T}\right)+r\right)\,e$$
we have
\begin{equation*}
\begin{array}{l}
\ds W_{e_i}\left(t,\left(\ws(e)t-\frac{n/2+1}{\ls(e)e\cd e'}\log\left(\frac{t+T}{T}\right)+r\right)\,e\right)\\ \ds=W_{e_i}\left(t,V(t;e_i) +t(\ws(e)e-\ws(e_i)e_i) +\log\left(\frac{t+T}{T}\right)\left(\frac{n/2+1}{\ls(e_i')e_i\cd e_i'}e_i-\frac{n/2+1}{\ls(e)e\cd e'}e\right)+re\right) 
\end{array}
\end{equation*}
where 
\begin{equation*}
V(t;e_i)=\left(\ws(e_i)t-\frac{n/2+1}{\ls(e_i')e_i\cd e_i'}\log\left(\frac{t+T}{T}\right)\right)\,e_i.
\end{equation*}

We want to know how many of the directions will have a non-zero effect there, and from those, what will the size of the effect be. Note that from elementary geometric considerations, the minimal angle between any two adjacent directions is of order $O\left(\frac{1}{\sqrt{t}}\right)$. $e_i$ affects $e$ at the location of order $O(1)$ from $\mathcal{W}'_t$ only if $\sqrt{t}\sin(\alpha)=O(1)$, where $\alpha$ is the angle between $e_1$ and $e'$. Since $\alpha\ge\frac{c}{\sqrt{t}}$ for some $c$ independent of $t$, the maximum number $N_0$ of $e_i$'s that can affect $e$ is independent of $t$. Then by increasing $L$ if necessary and using the second of the bounds on $W_{e_i}$ above, we can assure that at locations of $O(1)$ from $\mathcal{W}'_t$, $W$ remains bounded by $\epsilon$. Since $U(t,x)=\min\{1,W(t,x)\}$ agrees with $W$ there, then it is as well bounded by $\epsilon$. 

As we move further away from the boundary, these directions of small angle with $e$ will be $\epsilon/2$ by the second bound on $W_{e_i}$. However now we need to take into account the combined effect of all the other directions. But the extra directions are at angles of order larger than $1/\sqrt{t}$ from $e$, so now $x-\left(\ws(e_i)t-\frac{n/2+1}{\ls(e_i')e_i\cd e_i'}\log\left(\frac{t+T}{T}\right)\right)\,e_i$ is of size larger than $\sqrt{t}$. To $W_{e_i}$ for these directions we can then apply the first Gaussian bounds. This means that all but $N_0$ directions exhibit exponential decay. Then since the number of directions is algebraic in $t$, the cumulative effect is still decaying, in particular $\le\epsilon/2$, in these locations. This implies that $W(t,x)\le\epsilon$ for all $x$ at a distance greater than some $L$ from $\mathcal{W}_t'$. But then so is $U$ by its definition.

We have shown that $U(t,x)\le\epsilon$ for all $x=Re, R\ge\ws(e)t-\frac{n/2+1}{\ls(e)e\cd e'} \log\left(\frac{t+T}{T}\right)+L$. Since $U$ is a super-solution for $u$, then $u\le\epsilon$ in the same half-space, concluding so the proof of the main theorem.

\section{The equation in a linearly moving half-space - proofs}\label{linearspace}

In this section we obtain decay estimates on the solutions of the linearized Dirichlet problem in the half-space defined by $x\cd e'>\cs(e')t$. 
\subsection{The upper bound}

The first claim is a result of the heat kernel bounds for \eqref{p.linear} in the half-space. These can be obtained by using the kernel bounds in the whole space from \cite{norris}, applied to the equation in the following form.

\subsubsection*{Proof of the upper bound in the proposition}

The proof will be based on finding a weighted $L^1\longrightarrow L^2$ bound on the solution operator $S_t$ for \eqref{p.linear}, and extending this to an $L^1\longrightarrow L^{\infty}$ via a standard duality argument. 

For ease of notation, we denote by $\mathbb{H}^{e}_t:=\{x: x\cd e> \cs(e)t\}$ the half-space moving with speed $\cs(e)$ in the direction $e$. From \eqref{p.linear},
\begin{equation}\label{p^2} 
\frac{d}{dt}\ih{t}\nu(x)p^2(t,x)=-2\ih{t}\nu(x)|\nabla p(t,x)|^2\,dx.
\end{equation}
The right hand side will be bounded using the following Nash-type inequality.

\begin{lemma}  
There exists a constant $C$ such that for all $e\in\sn$ and all $f\in L^1(\he{0})\bigcap H^1(\he{0})$ satisfying $f\ge 0$ in $\he{0}$, $f=0$ on $\partial\he{0}$,
\begin{equation} \label{nash.eq}
\ihe{0}|f(x)|^2\,dx
\le C_n\left(\ihe{0}|\nabla f(x)|^2\,dx\right)^{\frac{n+2}{n+4}}\left(\ihe{0}x\cd e f(x)\,dx\right)^{\frac{4}{n+4}}.
\end{equation}
\end{lemma}
\begin{proof}
Rotate the $x$-axis so that the function $g(x)$ obtained from $f(x)$ with this change of variables is defined on the half-space $\{x_1\ge 0\}$. Let $G(x)$ be the odd extension of $g(x)$ to the whole $\rn$. Using Fourier isometry,
\begin{equation*} 
(2\pi)^n\int\limits_{\rn}|G(x)|^2\,dx =\int\limits_{\rn}|\hat{G}(\xi)|^2\,d\xi =\int\limits_{B_R}|\hat{G}(\xi)|^2\,d\xi+\int\limits_{B_R^c}|\hat{G}(\xi)|^2\,d\xi.
\end{equation*}
To bound the first term, note that
\begin{equation*} 
\left|\frac{\partial}{\partial\xi_1}\hat{G}(\xi)\right| =\left|\int\limits_{\rn}x_1\e{-ix\cd\xi}G(x)\,dx\right| \le\ds\int_{\rn}|x_1G(x)|\,dx.
\end{equation*}
Since $\hat{G}(0,\xi_2,\ldots,\xi_n)=0$, $\ |\hat{G(\xi)}|
=\left|\ds\int\limits_0^{\xi_1}\ds\frac{\partial}{\partial\xi_1}\hat{G}(\zeta_1,\xi_2,\ldots,\xi_n)\,d\zeta_1\right|\le |\xi_1|\ds\int_{\rn}|x_1G(x)|\,dx$. Therefore,
\begin{equation*} 
\int\limits_{B_R}|\hat{G}(\xi)|^2\,d\xi \le\int\limits_{B_R}|\xi_1|^2\,d\xi\left(\int\limits_{\rn}|x_1G(x)\,dx\right)^2 \le CR^{n+2} \left(\int\limits_{\rn}|x_1G(x)\,dx\right)^2.
\end{equation*}
On the other hand, 
\begin{equation*} 
\int\limits_{B_R^c}|\hat{G}(\xi)|^2\,d\xi \le\frac{1}{R^2}\int\limits_{B_R^c}|\xi|^2|\hat{G}(\xi)|^2\,d\xi =\frac{1}{R^2}\int\limits_{\rn}|\nabla G|^2\,dx.
\end{equation*}
Combining both bounds,
\begin{equation*}
\int\limits_{\rn}|G(x)|^2\,dx
\le CR^{n+2} ||x_1G||_1^2+C\frac{1}{R^2}||\nabla G||_2^2,
\end{equation*}
which can be optimized by choosing $R=\left(\frac{||\nabla G||_2}{||x_1G||_1}\right)^{\frac{2}{n+4}}$, to give
\begin{equation*} 
||G||_2\le C||\nabla G||_2^{\frac{n+2}{n+4}}||x_1G||_1^{\frac{2}{n+4}}.
\end{equation*}
Finally converting back to the original variables gives the desired inequality.
\end{proof}

After applying this inequality to $f(x)=p(t,x+\ws(e)t\,e)$ and a change of variables,
\begin{equation*} 
\ih{t}|p(t,x)|^2\,dx
\le C\left(\ih{t}|\nabla p(t,x)|^2\,dx\right)^{\frac{n+2}{n+4}}\left(\ih{t}(x\cd e'-\cs(e')t) p(t,x)\,dx\right)^{\frac{4}{n+4}}.
\end{equation*}
Recalling that $\nu(x)$ is uniformly bounded from above and below by positive constants, this implies for \eqref{p^2} that
\begin{equation}\label{p^2.2} 
\frac{d}{dt}\ih{t}\nu(x)p(t,x)^2\,dx
\le -C\left(\ih{t}\nu(x)p(t,x)^2\,dx\right)^{\frac{n+4}{n+2}}\left(\ih{t}(x\cd e'-\cs(e')t)p(t,x)\,dx\right)^{-\frac{4}{n+2}}
\end{equation}
If we could bound the second term on the right uniformly from above, then we would be left with an easy differential inequality for the $L^2$ norm of $p$. To this end, if we had a funtion $f$ that vanished when $x\cd e'=\cs(e')$, then
\begin{equation*} 
\frac{d}{dt}\int\limits_{\h{t}} \nu(x)f(t,x)p(t,x)\,dx=\int\limits_{\h{t}}\nu(x)p(t,x)\left[f_t(t,x)+\frac{1}{\nu(x)}\nabla\cd(\nu(x)\nabla f(t,x))+\frac{F(x)}{\nu(x)}\cd\nabla f(t,x)  \right]dx.
\end{equation*}
So if $f$ solves the backwards adjoint equation, then the integral on the left would be conserved. The existence of $f$ is established in the next lemma.

\begin{lemma} \label{lemma.backwards}
There exists a function $f(t,x)$ and a constant $M>1$ such that
\eal{backwards}{ll}{
f_t+\frac{1}{\nu(x)}\nabla\cd(\nu(x)\nabla f)+\frac{F(x)}{\nu(x)}\cd\nabla f = 0, \quad & x\in\h{t}, \ t\in\R,\\
f(t,x)=0, \quad & x\in\h{0}, \ t\in\R,
}
and satisfying 
\begin{equation}\label{fbounds}
f_t<0, \quad \text{ and } \quad M^{-1}(x-\ws(e)t\,e)\cd e' \le f(t,x) \le M(x-\ws(e)t\,e)\cd e' \ \text{ for all } x\in\h{t}, \ t\in\R.
\end{equation}
\end{lemma}
%
Note that if we define
\begin{equation}\label{lin.mom}
I(t):=\ih{t}\nu(x)f(t,x)p(t,x)\,dx,
\end{equation}
then $I(t)=I(0)$, and the linear moments of $p$ remain uniformly bounded, i.e.
\begin{equation}\label{lin.mom.bound}
C^{-1}I(0)\le\ih{0}x\cd e'p(t,x+\ws(e)t\,e)\,dx\le CI(0).
\end{equation}

\begin{proof}
The idea is to construct $f$ as the limit of a sequence of functions $f^{(n)}$ that satisfy approximate growth bounds on increasing time intervals. We start by noting that \eqref{backwards} possesses a solution of the form $G(t,x)=(x\cd e'-\cs(e')t) + g(x)$, where $g(x)$ satisfies the periodic equation
\begin{equation*} 
\nabla\cd(\nu(x)\nabla g(x)) + F(x)\cd\nabla g(x)
=\cs(e')\nu(x)-\nabla\nu(x)\cd e'-F(x)\cd e'.
\end{equation*}
The existence of $g$ is a consequence of the integral on the right being $0$ due to \eqref{normalized} and \eqref{drift}. Subtracting a constant from $g$ if necessary, we may assume that $G(t,x)\le 0$ on $\partial\h{t}=\{x\cd e'=\cs(e')t\}$.
$G$ solves \eqref{backwards}, but may not satisfy the boundary conditions. Instead, we use $G$ as a 
\ea{ll}{
f^{(n)}_t+\frac{1}{\nu(x)}\nabla\cd(\nu(x)\nabla f^{(n)})+\frac{F(x)}{\nu(x)}\cd\nabla f^{(n)} = 0, \quad & x\in\h{t}, \ t<n,\\
f^{(n)}(t,x)=0, \quad & x\in\h{t}, \ t\in\R,\\
f^{(n)}(n,x)=\max\{0,G(n,x)\}, \quad & x\in\h{n}.
}
Note that by the comparison principle, $f^{(n)}\ge G$ for all $x\in\h{t}, \ t<n$. Similarly, if $\alpha=\max_{\tn}g(x)$, then $G+\alpha\ge f^{(n)}$ both at the terminal time, and on the boundary of $\h{t}$, therefore the inequality holds everywhere in time and space. So
\begin{equation*} 
G(t,x)\le f^{(n)}(t,x)\le G(t,x)+\alpha, \ \forall x\in\h{t}, \ t\le n.
\end{equation*}
From this, when $x\cd e'-\cs(e')t\ge 2\alpha$,
\begin{equation*} 
f^{(n)}(t,x)\le x\cd e'-\cs(e')t+g(x)+\alpha
\le x\cd e'-\cs(e')t+2\alpha
\le 2(x\cd e'-\cs(e')t),
\end{equation*}
and
\begin{equation*}
f^{(n)}(t,x)\ge x\cd e'-\cs(e')t+g(x)
\ge x\cd e'-\cs(e')t-\alpha
\ge 2^{-1}(x\cd e'-\cs(e')t).
\end{equation*}
In particular, by the strong maximum principle, $\nabla f^{(n)}(t,x)\cd e'\ge c_0$ on the boundary of $\h{t}$ for all $n$ and $t\le n/2$. Therefore, by parabolic regularity, we can extract a subsequence converging to a limit $f(t,x)$ satisfying the PDE and the linear bounds.

Finally, $f^{(n)}(t,x)\ge 0$ and $f^{(n)}(t,x)\ge G$ for all $n$, therefore the same bounds hold for the limit $f(t,x)$. But then the parabolic maximum principle implies that $f_t(t,x)<0$.
\end{proof}
This implies that
\begin{equation*}
\begin{array}{ll}
\left(\ds\ih{t}(x\cd e'-\cs(e')t)p(t,x)\,dx\right)^{-1} &
\le C\left(\ds\ih{t}\nu(x)f(t,x)p(t,x)\,dx\right)^{-1}
=C\left(\ds\ih{0}\nu(x)f_0(x)p_0(x)\,dx\right)^{-1}\\
& \le C'\left(\ds\ih{0}x\cd e'\, p_0(x)\,dx\right)^{-1}.
\end{array}
\end{equation*}
Going back to \eqref{p^2.2} and solving the differential inequality, we obtain the following $L^1\longrightarrow L^2$ bound:
\begin{equation*} 
||p(t,\cd)||_2
\le \frac{C}{t^{\frac{n+2}{4}}}\ih{0}y\cd e' p_0(y)\,dy.
\end{equation*}
But the adjoint $S_t^*$ of the solution operator $S_t$ is the same except for having $-F(x)$ instead of $F(x)$ in \eqref{p.linear}, therefore by duality, it satisfies an $L^2\longrightarrow L^{\infty}$ bound with the same norm:
\begin{equation*} 
|p(t,x)|\le \frac{C(x\cd e'-\cs(e')t)}{t^{\frac{n+2}{4}}}||p_0||_2.
\end{equation*}
Since $S_t=S_{t/2}\circ S_{t/2}$, combining the bounds proves the first part of proposition \ref{linearbounds}.

\subsection{The lower bound}

The proof makes use of the estimates of the heat kernel bounds of the equation for $p$, together with a lower estimate for it in a small portion of a ball of radius $O(\sqrt{t})$. These are given respectively by lemma \ref{heat.kernel.bounds} and proposition \ref{lower.estimate} below.

We start with analyzing the heat kernel of the equation
\begin{equation} \label{p.equation}
\rho_t=\frac{1}{\nu(x;e')}\nabla\cd(\nu(x;e')\nabla \rho)-\frac{F(x;e')}{\nu(x;e')}\cd\nabla \rho
\end{equation}
on the cylinder
\begin{equation*} 
\mathcal{C}(s,x_0,R;e') := (s,\infty)\times B_R(x_0+\ws(e)s\,e)\subset \R_t\times\rn_x,
\end{equation*}
with Dirichlet boundary conditions. Let $\bar{\Gamma}(t,x,s,y;x_0,R,e')$ denote the heat kernel for \eqref{p.equation} in the cylinder. This means that $\bar{\Gamma}$ satisfies the equation in $(t,x)$ for $t>s$, $x\in B_R(x_0+\ws(e)t\,e)$ and $y\in B_R(x_0+\ws(e)s\,e)$, it vanishes when $x\in\partial B_R(x_0+\ws(e)t\,e)$, and has the initial condition $\lim\limits_{t\downarrow s} \bar{\Gamma}(t,x,s,y)=1/\nu(y)\delta_y(x)$.

\subsubsection{Heat kernel bounds}

Let us also denote by $\Gamma(t,x,s,y;e')$ the heat kernel for \eqref{p.equation} on the whole space. We will use the following result from \cite{norris} to get Gaussian bounds for $\Gamma$.

\begin{lemma}[Norris]\label{lemma.Norris}
Let L be an operator given by 
\begin{equation} \label{Norris}
Lf=\frac{1}{\nu(x)}\nabla\cd(\nu(x)(A(x)+B(x))\nabla f) +\frac{b}{\nu(x)}\cd\nabla f,
\end{equation}
where $A$ is a measurable, symmetric, uniformly elliptic matrix, $B$ is measurable, skew-symmetric, and they satisfy
\begin{equation*} 
\lambda^{-1}|\xi|^2\le \xi\cd (A(x)\xi) \le\lambda |\xi|^2, \ \forall x,\forall \xi, \qquad
|B(x)|\le \lambda, \ \forall x;
\end{equation*}
$\nu(x)$ satisfies $\lambda^{-1}\le\nu(x)\le\lambda$, and if $b\neq 0$, there exists a measurable, bounded vector field $V$ such that 
\begin{equation} \label{vf.condition}
\frac{1}{\nu(x)}\nabla\cd(\nu(x)V)=\frac{1}{\nu(x)}-1.
\end{equation}
Then there exists $K=K(n,\lambda)$ such that if $\Gamma$ is the heat kernel associated to $\partial_t-L$,
\begin{equation*} 
\frac{1}{Kt^{n/2}}\e{-K|x-y|^2/t} \le \Gamma(t,x,0,y+tb)
\le\frac{K}{t^{n/2}}\e{-|x-y|^2/Kt}.
\end{equation*}
\end{lemma}
%
For our operator, $A$ is the identity matrix, and $F(x)=(F(x)-\ws(e)e)+\ws(e)e$, where the first part is a mean zero, divergence-free vector field by virtue of \eqref{drift}. Then by Helmholtz decomposition, that can be expressed as the divergence of a skew-symmetric matrix $B(x)$, i.e. $F(x)-\ws(e)e=\nabla B(x)$, so
\begin{equation*}
\frac{F(x)}{\nu(x)}\cd\nabla\rho
=\frac{1}{\nu(x)}\nabla B(x)\cd\nabla\rho+\frac{\ws(e)e}{\nu(x)}\cd\nabla\rho
=\frac{1}{\nu(x)}\nabla\cd(\nu(x)(\frac{1}{\nu(x)} B(x)\cd\nabla\rho))+\frac{\ws(e)e}{\nu(x)}\cd\nabla\rho.
\end{equation*}
Therefore we can take $b=-\ws(e)e$, and $1/\nu(x)B(x)$ as the second matrix. Furthermore, \eqref{vf.condition} can be solved with $V=\nabla u$ for a periodic $u$, as the integral of the right side with respect to the $\nu(x)\,dx$ measure vanishes. Finally, all the functions involved have bounds that depend only on $\nu(x)$ and $\ws(e)$, and as all of these are bounded uniformly from above and below on $\sn$, then $\lambda$ exists and is uniform for all directions. Hence, this result can be applied to give
\begin{equation} \label{heat.rn}
\frac{1}{K(t-s)^{n/2}}\e{-K|x-y|^2/(t-s)} \le \Gamma(t,x,s,y-\ws(e)(t-s)\,e)
\le\frac{K}{(t-s)^{n/2}}\e{-|x-y|^2/K(t-s)}.
\end{equation}
We'll use these to prove the following result for $x$ and $y$ separated from the boundary of the cylinder.
%
\begin{lemma} \label{heat.kernel.bounds}
Given  $\delta\in(0,1)$, there exist constants $a$ and $K$ such that for all $R>0, x_0\in\rn, e\in\sn, t\in(s,s+R^2)$ and $x,y\in B_{\delta R}(x_0+\ws(e)t\,e)$,
\begin{equation*} 
\bar{\Gamma}(t,x,s,y-\ws(e)(t-s)\,e;x_0,R,e') \ge\frac{a}{K(t-s)^{\frac{n}{2}}}\e{-K|x-y|^2/(t-s)}.
\end{equation*}
\end{lemma} 
%
\begin{proof}
The proof is based on the analysis in \cite{fs}. Without loss of generality, we may assume $s=0$ and $x_0=0$. Let $\mathcal{C}_r=B_R(r\ws(e)\,e)$, and let $\phi(r,z)$ be a test function. From \eqref{p.equation}, 
\begin{equation*} 
\begin{array}{lll}
0&=& \ds\int\limits_{t_1}^{t_2}\ds\int\limits_{\disk{r}}(\nu(z)\rho_r-\nabla_z\cd(\nu(z)\nabla_z\rho)+F(z)\cd\nabla_z\rho)\phi\,dz\,dr  \\
&=& -\ds\int\limits_{t_1}^{t_2}\ds\int\limits_{\disk{r}}(\nu(z)\phi_r+\nabla_z\cd(\nu(z)\nabla_z\phi)+F(z)\cd\nabla_z\phi)\rho\,dz\,dr\\
&& + \ds\int\limits_{t_1}^{t_2}\ds\int\limits_{\partial\disk{r}}(\nu(z)\rho\nabla_z\phi - \nu(z)\phi\nabla_z\rho + \rho\phi F(z))\cd \hat{n}(z)\,dS_z\,dr + \left.\ds\int\limits_{\disk{r}} \nu(z)\rho\phi\,dz\right|_{t_1}^{t_2} ,
\end{array}
\end{equation*}
where $dS_z$ is the surface area of $\disk{r}$, and $\hat{n}(z)$ its outward unit normal. Now if $\rho=0$ on the boundary of $\disk{r}$, and $\phi$ solves the adjoint equation, we're left only with
\begin{equation} \label{heat.int} 
\ds\int\limits_{t_1}^{t_2}\ds\int\limits_{\partial\disk{r}}\nu(z)\phi\nabla_z\rho\cd \hat{n}(z)\,dS_z\,dr = \ds\int\limits_{\disk{t_2}} \nu(z)\rho\phi\,dz-\ds\int\limits_{\disk{t_1}} \nu(z)\rho\phi\,dz.
\end{equation}
Fixing $t>0$ and $x\in\disk{t}$, $y\in\disk{0}$, we can take $\rho(r,z)=\bar{\Gamma}(r,z,0,y)$, and $\phi(r,z)=\Gamma(t,x,r,z)$, as the heat kernel solves the backwards adjoint equation in its initial variables for $0<r<t$. Plugging these in \eqref{heat.int}, and keeping in mind that for any continuous $f$,
\begin{equation*} 
\lim\limits_{t_2\uparrow t}\int\limits_{\disk{t_2}}\nu(z)\Gamma(t,x,t_2,z)f(z)\,dz=f(x), \quad \text{ and }
\quad \lim\limits_{t_1\downarrow 0}\int\limits_{\disk{t_1}}\nu(z)\bar{\Gamma}(t_1,z,0,y)f(z)\,dz = f(y),
\end{equation*}
we obtain
\begin{equation} \label{heat.id} 
\bar{\Gamma}(t,x,0,y)=\Gamma(t,x,0,y)+\int\limits_0^t\int\limits_{\partial\disk{r}}\nu(z)\Gamma(t,x,r,z)\nabla_z\bar{\Gamma}(r,z,0,y)\cd \hat{n}(z)\,dS_z\,dr.
\end{equation}
Note that $\nabla_z\bar{\Gamma}(r,z,0,y)\cd \hat{n}(z)\le 0$ on $\partial\disk{r}$, therefore we can define a positive measure
\begin{equation*} 
d\mu(z,r)=-\nu(z)\nabla_z\bar{\Gamma}(r,z,0,y)\cd \hat{n}(z)\,dS_z\,dr \ \text{ on } \ \mathcal{C}(t)=\mathcal{C}(t,0,R;e').
\end{equation*}
To obtain information for this measure, we now take $\phi(r,z)=1$. \eqref{heat.int} then gives
\begin{equation*} 
\begin{array}{ll}
\mu(\mathcal{C}(t))& =-\ds\int\limits_{0}^{t}\ds\int\limits_{\partial\disk{r}}\nu(z)\nabla_z\bar{\Gamma}(r,z,0,y)\cd \hat{n}(z)\,dS_z\,dr = 1-\ds\int\limits_{\disk{t}}\nu(z)\bar{\Gamma}(t,z,0,y)\,dz <1.
\end{array}
\end{equation*}
Using this and \eqref{Norris} in \eqref{heat.id} for $x,y\in B_R(0)=D_0$ gives
\begin{equation*} 
\hat{\Gamma}(t,x,0,y-\ws(e)t\,e)
\ge\Gamma(t,x,0,y-\ws(e)t\,e)-\sup\limits_{z\in\disk{r}, 0\le r\le t} \Gamma(t,x+\ws(e)t\,e,r,z) .
\end{equation*}
In particular, if $x$ is away from the boundary of the cylinder, say $x\in B_{\delta R}(0)$, then 
\begin{equation*} 
\hat{\Gamma}(t,x,0,y-\ws(e)t\,e)\ge \frac{1}{Kt^{n/2}}\e{-K|x-y|^2/t}-\sup\limits_{0\le \tau\le t}\frac{K}{\tau^{n/2}}\e{-R^2(1-\delta)^2/K\tau}.
\end{equation*}
The function on the left attains its maximum at $\tau_{max}=\frac{2(1-\delta)^2}{Kn}R^2$, therefore for any $\epsilon<\frac{\sqrt{2}}{\sqrt{Kn}}(1-\delta)$, if $t\le\epsilon^2R^2$, then $t\le\tau_{max}$, so 
\begin{equation*} 
\begin{array}{ll}
\hat{\Gamma}(t,x,0,y-\ws(e)t\,e) &\ge \ds\frac{1}{Kt^{n/2}}\e{-\frac{K|x-y|^2}{t}}-\ds\frac{K}{t^{n/2}}\e{-\frac{R^2(1-\delta)^2}{Kt}}\\ &=\ds\frac{1}{Kt^{n/2}}\e{-\frac{K|x-y|^2}{t}}\left(1-K^2\e{-\frac{K}{t}\left(\frac{R^2(1-\delta)^2}{K^2}-|x-y|^2\right)}\right).
\end{array}
\end{equation*}
If furthermore $|x-y|\le\epsilon R$, and we shrink $\epsilon$ a bit to $\epsilon<\frac{1-\delta}{\sqrt{2Kn\ln(2K^2)}}$, then the expression in the brackets is at least $1/2$, leading to 
\begin{equation*} 
\hat{\Gamma}(t,x,0,y-\ws(e)t\,e) \ge \frac{1}{2Kt^{n/2}}\e{-\frac{K|x-y|^2}{t}}.
\end{equation*}
This bound was obtained for $x\in B_{\delta R}(0)$, with $|x-y|\le\epsilon^2R^2$ and $t\le\epsilon^2R^2$. Finally using a chaining argument analogous to the one in \cite{fs}, we can extend the bound to 
\begin{equation*} 
\bar{\Gamma}(t,x,s,y-\ws(e)t\,e;x_0,R,e') \ge\frac{a}{Kt^{n/2}}\e{-\frac{K|x-y|^2}{t}}.
\end{equation*}
for all $x,y\in B_{\delta R}(0)$ and $t\le R^2$.
\end{proof}

\subsubsection{Lower bound on a portion of the ball}

The next step towards a lower bound on $p$ on the whole ball of radius $\rho \sqrt{t}$ is the following bound that holds instead on a large enough portion of the ball.
%
\begin{prop}\label{lower.estimate}
There exists a time $T_0>0$ and constants $c_0>0, a>0$ and $A,A',B>0$ depending only on $u_0(x)$ such that for all $t\ge T_0$, there exists a set $D^+_t\subseteq \{x: A'\sqrt{t}\le x\cd e'\le A\sqrt{t}, x^\perp\in B_{B\sqrt{t}}\}$ with $|D^+_t|\ge \frac{a}{t^{\frac{n}{2}}}$, and
\begin{equation} \label{p.lower}
p(t,x+\ws(e)t\,e)\ge\frac{c_0}{t^{\frac{n+1}{2}}} \quad \text{ for all } x\in D^+_t.
\end{equation}
\end{prop}
%
Recall that $p$ is a function satisfying 
\eal{plinear}{ll}{
p_t=\ds\frac{1}{\nu(x;e')}\nabla\cd(\nu(x;e')\nabla p)-\ds\frac{F(x;e')}{\nu(x;e')}\cd\nabla p, &\quad x\in \h{t}, t>0,\\
p(t,x)=0, &\quad x\in\partial\h{t}, t>0.
}
The proof follows the same ideas as in \cite{hnrr2}. The idea is to find an upper bound for the exponential moments of $p$, and use them to control from above the $L^1$ norm of $x\cd e' p(t,x+\ws(e)t\,e)$ on $\{x\cd e'\ge N_u\sqrt{t}\}$. The same norm can be controlled from above on balls of radius $\sqrt{t}$ using the first part of proposition \ref{linearbounds}. Combining the two, we can ensure \eqref{p.lower} on a large portion of a ball of radius $O(\sqrt{t})$.\\

\noindent\textbf{Connection to the heat equation in $\R^{n+2}$.} Our fist observation is that for $x\sim O(\sqrt{t})$, $p(t,x+\ws(e)t\,e)$ should behave like $\frac{x\cd e'}{t^{\frac{n}{2}+1}}\e{-C|x|^2/t}$, and $\frac{1}{t^{\frac{n}{2}+1}}\e{-C|x|^2/t}$ appears to satisfy the asymptotics of a solution of the heat equation in $\mathbb{R}^{n+2}$. This suggests that we can try to understand the behavior of $p(t,x+\ws(e)t\,e)/(x\cd e')$. However as the function in the denominator may not work well with the operators in the problem, we'll instead look for a function that enjoys the same asymptotics, while solving a compatible equation.
%
\begin{lemma} \label{lemma.zeta} 
There exists a function $\zeta(t,x)$ and a constant $M>1$ such that 
\eal{zeta}{ll} 
{
\zeta_t=\frac{1}{\nu(x)}\nabla\cd(\nu(x)\nabla\zeta)-\frac{F(x)}{\nu(x)}\cd\nabla\zeta, & \quad x\in\h{t}, t\in\R,\\
\zeta(t,x)=0, & \quad x\in\partial\h{t}, t\in\R,
}
and
\begin{equation*} 
M^{-1}(x-\ws(e)t\,e)\cd e'\le \zeta(t,x)\le M(x-\ws(e)t\,e)\cd e'.
\end{equation*}
\end{lemma}
%
\begin{proof}
The idea is to start as in the proof of lemma \ref{lemma.backwards} with a solution of the form 
\begin{equation*} 
Z(t,x)=(x-\ws(e)t\,e)\cd e'+z(x),
\end{equation*}
with $z(x)$ periodic, and use this to control a subsequence of solutions that vanish in the boundary as well. Putting $Z$ in \eqref{zeta} implies that $\zeta$ should satisfy
\begin{equation*} 
\nabla\cd(\nu(x)z) -F(x)\cd \nabla z
= -\nabla\cd(\nu(x)e')+F(x)\cd e'-\nu(x)\ws(e)e\cd e'.
\end{equation*}
This equation has a periodic solution $z(x)$ since the integral of the right side is $0$ because of \eqref{normalized} and \eqref{drift}. Again, subtracting a constant if necessary, we may assume $z(x)\le 0$ on $\partial\h{t}$, and define the approximating sequence $\{\zeta^{(n)}(t,x)\}_n$ via
\ea{ll}{ 
\zeta^{(n)}_t=\frac{1}{\nu(x)}\nabla\cd(\nu(x)\nabla\zeta^{(n)})-\frac{F(x)}{\nu(x)}\cd\nabla\zeta^{(n)}, & \quad x\in\h{t}, t>-n,\\
\zeta^{(n)}(t,x)=0, & \quad x\in\partial\h{t}, t>-n,\\
\zeta^{(n)}(-n,x)=\max\{0,Z(-n,x)\}, & \quad x\in\h{-n}
}
The bounds for $\zeta^{(n)}$ and the extraction of a subsequence converging to a solution of \eqref{zeta} is a consequence of the maximum principle and parabolic regularity, and follows verbatim the proof of lemma \ref{lemma.backwards}.
\end{proof}

Now define $q(t,x)=\frac{p(t,x)}{\zeta(t,x)}$. This solves
\begin{equation*}
q_t=\frac{1}{\nu(x)}\nabla\cd(\nu(x)\nabla q)+\left(2\frac{\nabla\zeta}{\zeta}-\frac{F(x)}{\nu(x)}\right)\cd\nabla q.
\end{equation*}
Recall that the linear moment for $p$ remains bounded, as it is controlled by 
\begin{equation*} 
I(t)=\ih{t}\nu(x)f(t,x)p(t,x)\,dx =I(0),
\end{equation*}
where $f(t,x)$ is the function from lemma \ref{lemma.backwards}. In a similar way, we define the exponential moments
\begin{equation}\label{i.alpha} 
I_{\alpha}(t)=\ih{t}\nu(x)\eta_{\alpha}(t,x)p(t,x)\,dx,
\end{equation}
with $\eta_{\alpha}$ expected to grow like $\e{\alpha\cd x}$ for $|x|$ large and $x\cd e'$ bounded away from zero. 
Then $I_{\alpha}$ satisfies
\begin{equation*} 
\begin{array}{ll}
\ds\frac{dI_{\alpha}(t)}{dt} &= \ds\ih{t} \nu(x)(\eta_{\alpha})_t + \eta_{\alpha}(\nabla\cd(\nu(x)\nabla p)-F(x)\cd\nabla p)\,dx\\
&=\ds\ih{t}\nu(x)\left((\eta_{\alpha})_t+\ds\frac{1}{\nu(x)}\nabla\cd(\nu(x)\nabla\eta_{\alpha} +\ds\frac{F(x)}{\nu(x)}\cd\nabla\eta_{\alpha}\right)\,dx +\ds\int\limits_{\partial\h{t}}\nu(x)\eta_{\alpha}\nabla p\cd \hat{n}(x)\,dS_x.
\end{array}
\end{equation*}
The adjoint equation appeared before as well, but associated with linearly growing functions, therefore we expected that those functions to be exact solutions of the equation. In this case, we want $\eta_{\alpha}$ to be exponentially growing, therefore we expect it instead to be an eigenfunction for the adjoint operator. This result is described in the following lemma, whose proof will be postponed for later.
%
\begin{lemma} \label{lemma.eta} 
There exists $C>1, a_0>0$ and a positive-definite matrix $\Lambda$ such that for every $\alpha\in\rn$ satisfying
\begin{equation}\label{alpha.cond}
\alpha\cd e'>n^{-1/2}|\alpha|, \quad 2^{-1}\alpha\cd e'\le\frac{\alpha\cd\Lambda e'}{e'\cd\Lambda e'}\le 2\alpha\cd e' \quad \text{and} \quad |\alpha|\le a_0,
\end{equation}
there exist $\lambda(\alpha)$ and $\eta_{\alpha}(t,x)$ satisfying
\eal{eta.eq}{ll}{ 
(\eta_{\alpha})_t+\ds\frac{1}{\nu(x)}\nabla\cd(\nu(x)\nabla\eta_{\alpha}) +\ds\frac{F(x)}{\nu(x)}\cd\nabla\eta_{\alpha}=\lambda(\alpha)\eta_{\alpha}, & \quad t\in\R, \ x\in\h{t},\\
\eta_{\alpha}(t,x)=0, & \quad t\in\R, x\in\partial\h{t},
}
with
\begin{equation} \label{eta.bounds} 
C^{-1}\e{\alpha\cd x}\frac{1-\e{-2\alpha\cd e'x\cd e'}}{\alpha\cd e'} \le\eta_{\alpha}(t,x+\ws(e)t\,e) \le C\e{\alpha\cd x}\frac{1-\e{-2\alpha\cd e' x\cd e'}}{\alpha\cd e'},
\end{equation}
and
\begin{equation} \label{lambda}
\lambda(\alpha)=\alpha\cd\Lambda\alpha+O(|\alpha|^3).
\end{equation}
\end{lemma}
%
\noindent Note that because $\Lambda$ is positive definite, $e'\cd\Lambda e'$ is positive, and the comparability of $\alpha\cd e', \alpha\cd\Lambda e'$ and $|\alpha|$ is needed to ensure that the bounds in \eqref{eta.bounds} are uniform. If $\alpha$ satisfies merely $\alpha\cd\Lambda e'>0$, $\eta_\alpha$ would still exist but would instead be proportional to $\e{\alpha\cd x}\frac{1-\e{-2\frac{\alpha\cd\Lambda e'}{e'\cd\Lambda e'}x\cd e'}}{\alpha\cd\Lambda e'}$.\\

\noindent\textbf{Bounding the exponential moments.} Recall from \eqref{i.alpha} that
\begin{equation*} 
I_{\alpha}(t)=\ih{t}\nu(x)\eta_{\alpha}(t,x)p(t,x)\,dx =\ih{t}\eta_\alpha(t,x)\zeta(t,x)q(t,x)\,dx,
\end{equation*}
and with the above choice of $\eta_\alpha(t,x)$, $\frac{dI_\alpha(t)}{dt}=\lambda(\alpha)I_\alpha(t)$. Define the second exponential moment as
\begin{equation*} 
V_{\alpha}(t)=\ih{t}\nu(x)\eta_{2\alpha}(t,x)\zeta(t,x)q^2(t,x)\,dx.
\end{equation*}
Then
\begin{equation*} 
\begin{array}{lll}
\ds\frac{dV_\alpha(t)}{dt} &=& \ds\ih{t}\zeta q^2[-\nabla\cd(\nu\nabla\eta_{2\alpha})-F\cd\nabla\eta_{2\alpha}+\lambda(2\alpha)\nu\eta_{2\alpha}]\,dx +\ds\ih{t}\eta_{2\alpha}q^2[\nabla\cd(\nu\nabla\zeta)-F\cd\nabla\zeta]\,dx \\
&& +\ds\ih{t}2\eta_{2\alpha}\zeta q[\nabla\cd(\nu\nabla q)+2\nu\zeta^{-1}\nabla\zeta\cd\nabla q-F\cd\nabla q]\,dx \\
&=& \ds\ih{t}\nu q^2\nabla\zeta\cd\nabla\eta_{2\alpha} +2\nu\zeta q\nabla\eta_{2\alpha}\cd\nabla q +\eta_{2\alpha}q^2F\cd\nabla\zeta +2\eta_{2\alpha}\zeta qF\cd\nabla q +\lambda(2\alpha)\nu\eta_{2\alpha}\zeta q^2\,dx \\
&& +\ds\ih{t}-\nu q^2\nabla\zeta\cd\nabla\eta_{2\alpha} -2\nu\eta_{2\alpha}\nabla\zeta\cd\nabla q -\eta_{2\alpha}q^2F\cd\nabla\zeta \,dx \\
&& +\ds\ih{t}-2\nu\zeta q\nabla\eta_{2\alpha}\cd\nabla q -2\nu\eta_{2\alpha}q\nabla\zeta\cd\nabla q -2\nu\eta_{2\alpha}\zeta|\nabla q|^2 +4\nu\eta_{2\alpha}q\nabla\zeta\cd\nabla q -2\eta_{2\alpha}\zeta qF\cd\nabla q\,dx\\
&=& \lambda(2\alpha)V_\alpha(t)-2\ds\ih{t}\nu\eta_{2\alpha}\zeta|\nabla q|^2\,dx.
\end{array}
\end{equation*}
If we let
\begin{equation*} 
D_\alpha(t)=\ih{t}\nu(x)\eta_{2\alpha}(t,x)\zeta(t,x)|\nabla q(t,x)|^2\,dx,
\end{equation*}
we end up with the following ODEs:
\begin{equation} \label{ivd} 
\frac{dI_\alpha(t)}{dt}=\lambda(\alpha)I_\alpha(t) \quad \text{and} \quad \frac{dV_\alpha(t)}{dt}=\lambda(2\alpha)V_\alpha(t)-2D_\alpha(t).
\end{equation}
\textbf{Connecting $D_\alpha$ to $I_\alpha$ and $V_\alpha$.} Next we prove that there exists a constant $C$ depending only on the dimension, such that the following Nash type inequality holds for all $\alpha$ satisfying the conditions in \eqref{alpha.cond} and $t>0$:
\begin{equation}\label{Nash.n+2} 
D_\alpha(t)\ge V_\alpha(t)^{\frac{n+4}{n+2}}I_\alpha(t)^{-\frac{4}{n+2}}.
\end{equation}
Using the linear and exponential bounds in the lemmas \ref{lemma.zeta} and \ref{lemma.eta}, 
\begin{equation*} 
I_{\alpha}(t)=\ih{t}\nu(x)\eta_\alpha(t,x)\zeta(t,x)q(t,x)\,dx \ge C_I\ih{0}\e{\alpha\cd x}\frac{1-\e{-2\alpha\cd e'x\cd e'}}{\alpha\cd e'}x\cd e'q(t,x+\ws(e)t\,e)\,dx,
\end{equation*}
\begin{equation*} 
V_{\alpha}(t)=\ih{t}\nu(x)\eta_{2\alpha}(t,x)\zeta(t,x)q^2(t,x)\,dx \le C_V\ih{0}\e{2\alpha\cd x}\frac{1-\e{-4\alpha\cd e'x\cd e'}}{2\alpha\cd e'}x\cd e'q^2(t,x+\ws(e)t\,e)\,dx,
\end{equation*}
\begin{equation*} 
D_\alpha(t)=\ih{t}\nu(x)\eta_{2\alpha}(t,x)\zeta(t,x)|\nabla q(t,x)|^2\,dx \ge C_D\ih{0}\e{2\alpha\cd x}\frac{1-\e{-4\alpha\cd e'x\cd e'}}{2\alpha\cd e'}x\cd e'|\nabla q(t,x+\ws(e)t\,e)|^2\,dx.
\end{equation*}
In order to compare $I_\alpha, V_\alpha$ and $D_\alpha$, it is enough to compare the quantities in the right, with $q(t,x+\ws(e)t\,e)$ considered as a function of $x$ only. Given any any bounded, $C^1$ function $f:\h{t}\to\R$ that vanishes in the boundary and decays at infinity, we can define the analogous quantities
\begin{equation*} 
\hat{I}_\alpha=\ih{0}\e{\alpha\cd x}\frac{1-\e{-2\alpha\cd e'x\cd e'}}{\alpha\cd e'} x\cd e'f(x)\,dx, \quad
\hat{V}_\alpha=\ih{0}\e{2\alpha\cd x}\frac{1-\e{-4\alpha\cd e'x\cd e'}}{2\alpha\cd e'} x\cd e'f^2(x)\,dx,
\end{equation*}
\begin{equation*} 
\hat{D}_\alpha=\ih{0}\e{2\alpha\cd x}\frac{1-\e{-4\alpha\cd e'x\cd e'}}{2\alpha\cd e'}x\cd e'|\nabla f(x)|^2\,dx. 
\end{equation*}

\begin{lemma} 
There exists $C>0$ depending only on $n$, such that for all $0<\alpha\le\alpha_0$,
\begin{equation*} 
\hat{D}_\alpha\ge C\hat{V}_\alpha^{\frac{n+4}{n+2}}\hat{I}_\alpha^{-\frac{4}{n+2}}.
\end{equation*}
\end{lemma}
%
\begin{proof}
Assume first that $e'=e_1$. We will lift the function to $\mathbb{R}^{n+2}$ by considering $x_1\ge 0$ as a radial variable for the first $3$ new coordinates, and by taking $\tilde{\alpha}=(\alpha_1,0,0,\alpha_2,\ldots,\alpha_n)$. So for $z\in\mathbb{R}^{n+2}$, define 
\begin{equation*} 
F(z)=f(|(z_1,z_2,z_3)|,z_4,\ldots,z_{n+2}),
\end{equation*}
and let $G(z)$ be
\begin{equation*} 
G(z)=\e{\tilde{\alpha}\cd z}F(z).
\end{equation*}
Using spherical coordinates for $(z_1,z_2,z_3)$ and calling $\bar{z}=(z_4,\ldots,z_{n+2})$ and $\bar{\alpha}=(\tilde{\alpha}_4,\ldots,\tilde{\alpha}_{n+2})=(\alpha_2,\ldots,\alpha_n)$ we note that
\begin{equation*} 
\begin{array}{ll}
\ds\int\limits_{\mathbb{R}^{n+2}}G(z)\,dz &= \ds\int\limits_\R\cdots\ds\int\limits_\R\ds\int\limits_0^\infty\ds\int\limits_0^{2\pi}\int\limits_0^\pi \e{\alpha_1rcos\phi}\e{\bar{\alpha}\cd\bar{z}}f(r,z_4,\ldots,z_{n+2})r^2\sin\phi\,d\phi\,d\theta\,dr\,dz_2\ldots dz_{n+2}\\
&=2\pi\ds\int\limits_{\mathbb{R}^{n-1}}\ds\int\limits_0^\infty \ds\frac{\e{\alpha_1r}-\e{-\alpha_1r}}{\alpha_1}\e{\bar{\alpha}\cd\bar{z}} f(r,z_4,\ldots,z_{n+2})r\,dr\,d\bar{z}=2\pi\hat{I}_\alpha.
\end{array}
\end{equation*}
Similarly
\begin{equation*} 
\ds\int\limits_{\mathbb{R}^{n+2}}G^2(z)\,dz=2\pi\hat{V}_\alpha.
\end{equation*}
On the other hand, $|\nabla G(z)|^2=\e{2\tilde{\alpha}\cd z}\left[|\nabla f|^2+2\alpha_1\frac{z_1}{|(z_1,z_2,z_3)|}f\partial_1f+2\bar{\alpha}\cd\bar{\nabla}f+|\alpha|^2f^2\right]$, and
\begin{equation*} 
\begin{array}{l}
\ds\int\limits_{\mathbb{R}^{n+2}} \e{2\tilde{\alpha}\cd z} \left(2\alpha_1\frac{z_1}{|(z_1,z_2,z_3)|}f\partial_1f+2\bar{\alpha}\cd\bar{\nabla}f\right)\,dz = \ds\int\limits_{\mathbb{R}^{n-1}}\int\limits_0^\infty\ds\int\limits_0^{2\pi}\int\limits_0^\pi \e{2\alpha_1rcos\phi}\e{2\bar{\alpha}\cd\bar{z}}(\alpha_1\cos\phi(f^2)_r+\bar{\alpha}\cd\bar{\nabla}(f^2))r^2\sin\phi\,d\phi\,d\theta\,dr\,d\bar{z}\\
=-\ds\int\limits_{\mathbb{R}^{n-1}}\ds\int\limits_0^\infty\ds\int\limits_0^{2\pi}\int\limits_0^\pi 2\e{2\alpha_1rcos\phi}\e{2\bar{\alpha}\cd\bar{z}} (\alpha_1^2f^2r^2\cos^2\phi +\alpha_1f^2r\cos\phi +|\bar{\alpha}|^2f^2r^2) \sin\phi\,d\phi\,d\theta\,dr\,d\bar{z}\\
=-\ds\int\limits_{\mathbb{R}^{n-1}}\ds\int\limits_0^\infty\ds\int\limits_0^{2\pi} 2\alpha_1^2\frac{\e{2\alpha_1r}-\e{-2\alpha_1r}}{2\alpha_1r} \e{2\bar{\alpha}\cd\bar{z}}f^2r^2\,d\theta\,dr\,d\bar{z} +\ds\int\limits_{\mathbb{R}^{n-1}}\ds\int\limits_0^\infty\ds\int\limits_0^{2\pi}\int\limits_0^\pi 2\alpha_1^2\frac{1}{\alpha_1r}\e{2\alpha_1rcos\phi}\e{2\bar{\alpha}\cd\bar{z}}f^2r^2\cos\phi\sin\phi\,d\phi\,d\theta\,dr\,d\bar{z}\\
\hspace*{10pt}-\ds\int\limits_{\mathbb{R}^{n-1}}\ds\int\limits_0^\infty\ds\int\limits_0^{2\pi}\int\limits_0^\pi 2\alpha_1\e{2\alpha_1rcos\phi}\e{2\bar{\alpha}\cd\bar{z}}f^2r\cos\phi\sin\phi\,d\phi\,d\theta\,dr\,d\bar{z} -\ds\int\limits_{\mathbb{R}^{n-1}}\ds\int\limits_0^\infty\ds\int\limits_0^{2\pi} 2|\bar{\alpha}|^2\frac{\e{2\alpha_1r}-\e{-2\alpha_1r}}{2\alpha_1r} \e{2\bar{\alpha}\cd\bar{z}}f^2r^2\,d\theta\,dr\,d\bar{z}\\
\hspace*{10pt}=-\ds 2|\alpha|^2\int\limits_{\mathbb{R}^{n-1}}\ds\int\limits_0^\infty\ds\int\limits_0^{2\pi} \frac{\e{2\alpha_1r}-\e{-2\alpha_1r}}{2\alpha_1r} \e{2\bar{\alpha}\cd\bar{z}}f^2r^2\,d\theta\,dr\,d\bar{z} =-2|\alpha|^2\ds\int\limits_{\mathbb{R}^{n+2}}\e{2\alpha z\cd e_1}f^2\,dz.
\end{array}
\end{equation*}
Therefore,
\begin{equation*} 
\ds\int\limits_{\mathbb{R}^{n+2}}|\nabla G(z)|^2\,dz =\int\limits_{\mathbb{R}^{n+2}}\e{2\tilde{\alpha}\cd z}(|\nabla f|^2-|\alpha|^2f^2) \le 2\pi\hat{D}_\alpha.
\end{equation*}
Now the inequality follows directly from Nash inequality (refer to \cite{nash}) on $G$ in $\mathbb{R}^{n+2}$. If $e'\not\equiv e_1$, consider the orthonormal rotation matrix $E$ that satisfies $e'=Ee_1$, and apply the result to $f(Ex)$ on $x\cd e_1= Ex\cd e'>0$ instead.
\end{proof}

Going back to \eqref{Nash.n+2}, we see that the inequality follows from the lemma and the fact that
\begin{equation*} 
I_\alpha\ge C_I\bar{I}_\alpha, \quad V_\alpha\le C_V\bar{V}_\alpha, \quad I_\alpha\ge C_D\bar{D}_\alpha.
\end{equation*}
From \eqref{ivd}, $I_\alpha(t)=I_\alpha(0)\e{\lambda(\alpha)t}$, whereas $V_\alpha$ satisfies 
\begin{equation*} 
V_\alpha'=\lambda(2\alpha)V_\alpha-2D_\alpha \le\lambda(2\alpha)V_\alpha-CV_\alpha^{\frac{n+4}{n+2}}I_\alpha^{-\frac{4}{n+2}} \le\lambda(2\alpha)V_\alpha-CI_\alpha(0)^{-\frac{4}{n+2}}\e{-\frac{4}{n+2}\lambda(\alpha)t}V_\alpha^{\frac{n+4}{n+2}}.
\end{equation*}
If we let $V_\alpha(t)=\e{\lambda(2\alpha)t}Z_\alpha(t)$, and $\sigma(\alpha)=\frac{2}{n+2}\lambda(2\alpha)-\frac{4}{n+2}\lambda(\alpha)$, then
\begin{equation*} 
Z_\alpha'\le-CI_\alpha(0)^{-\frac{4}{n+2}}\e{(\frac{2}{n+2}\lambda(2\alpha)-\frac{4}{n+2}\lambda(\alpha))t} =-CI_\alpha(0)^{-\frac{4}{n+2}}\e{\sigma(\alpha)t},
\end{equation*}
so
\begin{equation*}
Z_\alpha(t)\le CI_\alpha(0)^2\left(\frac{\sigma(\alpha)}{\e{\sigma(\alpha)t}-\e{\sigma(\alpha)}}\right)^{\frac{n+2}{2}} \le CI_\alpha(0)^2\frac{1}{(t-1)^{\frac{n+2}{2}}}\e{-\frac{n+2}{2}\sigma(\alpha)},
\end{equation*}
because for small $|\alpha|$, $\sigma(\alpha)=\frac{2\alpha\cd\Lambda\alpha}{n+2}+O(|\alpha|^3)>0$ as a consequence of \eqref{lambda}, and for $\sigma>0, t>1$, $\e{\sigma t}-\e{\sigma}\ge\sigma\e{\sigma}(t-1)$. Furthermore the positivity of $\sigma(\alpha)$ implies that 
\begin{equation*}
V_\alpha(t)\le C\frac{I_\alpha(0)^2}{(t-1)^{\frac{n+2}{2}}}\e{\lambda(2\alpha)t},
\end{equation*}
i.e.
\begin{equation*}
\left(\ih{t}\nu(x)\eta_{2\alpha}(t,x)\zeta(t,x)q^2(t,x)\,dx\right)^{1/2} \le C\frac{\e{\frac{\lambda(2\alpha)}{2}t}}{(t-1)^{\frac{n+2}{4}}}\ih{0}\nu(x)\eta_\alpha(0,x)\zeta(0,x)q_0(x)\,dx.
\end{equation*}
Invoking once more the bounds of lemmas \ref{lemma.zeta} and \ref{lemma.eta}, we get that
\begin{equation} \label{int.bound}
\left(\ih{0}\e{2\alpha\cd x}\frac{1-\e{-4\alpha\cd e'x\cd e'}}{2\alpha\cd e'x\cd e'}p^2(t,x+\ws(e)t\,e)\,dx\right)^{1/2} \le C\frac{\e{\frac{\lambda(2\alpha)}{2}t}}{(t-1)^{\frac{n+2}{4}}}\ih{0} \e{\alpha\cd x}\frac{1-\e{-2\alpha\cd e'x\cd e'}}{\alpha\cd e'}p_0(x)\,dx.
\end{equation}

\subsubsection{Proof of lemma \ref{lemma.eta}}

\textbf{Eigenvalue asymptotics.} We expect 
$\eta_\alpha(t,x)\sim\e{\alpha\cd(x-\ws(e)t\,e)}\bar{\eta}_\alpha(x)$ for $x$ away from the boundary of $\h{t}$ so with this in mind, we start by the ansatz that 
$\e{\alpha\cd(x-\ws(e)t\,e)}\bar{\eta}_\alpha(x)$ solves \eqref{eta.eq} for some $\lambda(\alpha)$. This requires $\bar{\eta}_\alpha$ to solve
\begin{equation*} 
\begin{array}{ll}
\ds\frac{1}{\nu(x)}\nabla\cd(\nu(x)\nabla\bar{\eta}_\alpha)+\ds\left(\frac{F(x)}{\nu(x)}+\alpha\right)\cd\nabla\bar{\eta}_\alpha + \ds\frac{\alpha}{\nu(x)}\cd\nabla(\nu(x)\bar{\eta}_\alpha) +\ds\alpha\cd\left(\frac{F(x)}{\nu(x)}-\ws(e)e\right)\bar{\eta}_\alpha
=\bar{\lambda}(\alpha)\bar{\eta}_\alpha.
\end{array}
\end{equation*}
where $\bar{\lambda}(\alpha)=\lambda(\alpha)-|\alpha|^2$. Thus we consider $\bar{\eta}_\alpha$ to be the unique positive, periodic eigenfunction of this equation, normalized so that
\begin{equation} \label{normal.eta} 
\int\limits_{\tn}\nu(x)\bar{\eta}_\alpha(x)\,dx=1.
\end{equation}
First at $\alpha=0$, $\bar{\lambda}(0)=0$ and $\bar{\eta}_0(x)=1$ are respectively the principal eigenvalue and eigenfunction of the operator. The operator depends analytically on the parameter $\alpha$, so we can differentiate the equation in $\alpha_i$'s to find the quadratic approximation of $\bar{\lambda}$ for small $\alpha$. If $\bar{\eta}_i:=\frac{\partial\bar{\eta}_\alpha}{\partial\alpha_i}$ and $\bar{\lambda}_i(\alpha):=\frac{\partial\bar{\lambda}}{\partial\alpha_i}$, then 
\begin{equation*} 
\begin{array}{ll}
\ds\frac{1}{\nu(x)}\nabla\cd(\nu(x)\nabla\bar{\eta}_i) + \ds\left(\frac{F(x)}{\nu(x)}+\alpha\right)\cd\nabla\bar{\eta}_i + \ds\frac{\alpha}{\nu(x)}\cd\nabla(\nu(x)\bar{\eta}_i) + \alpha\cd\left(\frac{F(x)}{\nu(x)}-\ws(e)e\right)\bar{\eta}_i\\
+e_i\cd\nabla\bar{\eta}_\alpha + \ds\frac{e_i}{\nu(x)}\cd\nabla(\nu(x)\bar{\eta}_\alpha) + \ds e_i\cd\left(\frac{F(x)}{\nu(x)}-\ws(e)e\right)\bar{\eta}_\alpha = \bar{\lambda}_i(\alpha)\bar{\eta}_\alpha +\bar{\lambda}(\alpha)\bar{\eta}_i.
\end{array}
\end{equation*}
At $\alpha=0$, using that $\bar{\lambda}(0)=0$ and $\bar{\eta}_0=1$, we get
\begin{equation} \label{first.der}
\frac{1}{\nu(x)}\nabla\cd(\nu(x)\nabla\bar{\eta}_i) +\frac{F(x)}{\nu(x)}\cd \nabla\bar{\eta}_i +\frac{e_i}{\nu(x)}\cd\nabla\nu(x) +e_i\cd\left(\frac{F(x)}{\nu(x)}-\ws(e)e'\right) =\bar{\lambda}'(0).
\end{equation}
Integrating this with respect to $\nu(x)$ on the torus, and recalling \eqref{normalized} and \eqref{drift}, results in $\bar{\lambda}'(0)=0$. Differentiating once more in $\alpha_j$, and calling $\bar{\eta}_{ij}:=\frac{\partial^2\bar{\eta}_{\alpha}}{\partial\alpha_{ij}}$, $\bar{\lambda}_{ij}:=\frac{\partial^2\bar{\lambda}(\alpha)}{\partial\alpha_{ij}}$,
\begin{equation*} 
\begin{array}{ll}
\ds\frac{1}{\nu(x)}\nabla(\nu(x)\nabla\bar{\eta}_{ij}) +\ds\left(\frac{F(x)}{\nu(x)}+\alpha\right)\cd\nabla\bar{\eta}_{ij}  +\ds\frac{\alpha}{\nu(x)}\cd\nabla(\nu(x)\bar{\eta}_{ij}) +\ds\alpha\cd\left(\frac{F(x)}{\nu(x)}-\ws(e)e\right)\bar{\eta}_{ij}\\
+e_i\cd\nabla\eta_j+e_j\cd\nabla\eta_i 
+\ds\frac{e_i}{\nu(x)}\cd\nabla(\nu(x)\bar{\eta}_j) +\ds\frac{e_j}{\nu(x)}\cd\nabla(\nu(x)\bar{\eta}_i) +\ds e_i\cd\left(\frac{F(x)}{\nu(x)}-\ws(e)e\right)\bar{\eta}_j\\
+\ds e_j\cd\left(\frac{F(x)}{\nu(x)}-\ws(e)e\right)\bar{\eta}_i
=\bar{\lambda}_{ij}(\alpha)\bar{\eta}_\alpha+\bar{\lambda}_i(\alpha)\bar{\eta}_j+\bar{\lambda}_j(\alpha)\bar{\eta}_i+\bar{\lambda}(\alpha)\bar{\eta}_{ij}.
\end{array}
\end{equation*}
At $\alpha=0$, 
\begin{equation*} 
\begin{array}{ll}
\ds\frac{1}{\nu(x)}\nabla\cd(\nu(x)\nabla\bar{\eta}_{ij}) +\ds\frac{F(x)}{\nu(x)}\cd \nabla\bar{\eta}_{ij} +e_i\cd\nabla\eta_j+e_j\cd\nabla\eta_i  +\ds\frac{e_i}{\nu(x)}\cd\nabla(\nu(x)\bar{\eta}_j) +\ds\frac{e_j}{\nu(x)}\cd\nabla(\nu(x)\bar{\eta}_i)\\
+\ds e_i\cd\left(\frac{F(x)}{\nu(x)}-\ws(e)e\right)\bar{\eta}_j
+\ds e_j\cd\left(\frac{F(x)}{\nu(x)}-\ws(e)e\right)\bar{\eta}_i
=\bar{\lambda}_{ij}.
\end{array}
\end{equation*}
Integrating this with respect to $\nu(x)$ gives
\begin{equation*} 
\bar{\lambda}_{ij}(0)=\int\limits_{\tn}\nu(x)(e_i\cd\nabla\bar{\eta}_j+e_j\cd\nabla\bar{\eta}_i)+F(x)\cd e_i\bar{\eta}_j+F(x)\cd e_j\bar{\eta}_i\,dx.
\end{equation*}
Here we used the normalization \eqref{normal.eta} to get that $\int\nu(x)\eta_i\,dx=0$ for all $\alpha$ and $i$. Lastly, integrating \eqref{first.der} with respect to $\nu(x)\bar{\eta}_j$ gives
\begin{equation*} 
\int\limits_{\tn}F(x)\cd e_i\bar{\eta}_j\,dx =\int\limits_{\tn}\nu(x)\nabla\bar{\eta}_i\cd\nabla\bar{\eta}_j -F(x)\cd\nabla\bar{\eta}_i\bar{\eta}_j +\nu(x)e_i\cd\nabla\bar{\eta}_j\,dx.
\end{equation*}
Plugging this into the expression for $\bar{\lambda}_{ij}(0)$, and recalling that $\bar{\lambda}(\alpha)=\lambda(\alpha)-|\alpha|^2$
\begin{equation*} 
\Lambda_{ij}:=\lambda_{ij}(0) =2\int\limits_{\tn}\nu(x)(\nabla\bar{\eta}_i\cd\nabla\bar{\eta}_j+e_i\cd\nabla\bar{\eta}_j+e_j\cd\nabla\bar{\eta}_i)\,dx+2e_ie_j =2\int\limits_{\tn}\nu(x)(\nabla\bar{\eta}_i+e_i)\cd(\nabla\bar{\eta}_j+e_j)\,dx.
\end{equation*}
Note that $\Lambda$ is a positive-definite matrix, as $\nabla\bar{\eta}_i+e_i\not\equiv 0$ due to the periodicity of $\bar{\eta}_\alpha$. This proves \eqref{lambda}.\\

\noindent\textbf{Construction of $\eta_\alpha(t,x)$.} 
We want the weight $\eta_\alpha$ to be $0$ on the boundary as well as be asymptotically exponential away from it, therefore it is natural to look for a function comparable to $\frac{\e{\alpha\cd(x-\ws(e)t\,e)}-\e{(\alpha-2\alpha\cd e' e')\cd(x-\ws(e)t\,e)}}{\alpha\cd e'}$. The discussion above suggests that the first term should arise from $\frac{\e{\alpha\cd(x-\ws(e)t\,e)}}{\alpha\cd e'}\bar{\eta}_\alpha(x)$. Analogously the second term should be approximated by the same expression, evaluated at $\alpha-2\alpha\cd e'e'$. 

An important factor in how we choose these two sides is that their respective eigenvalues coincide. This means that $\lambda(\alpha)$ and $\lambda(\alpha-2\alpha\cd e'e')$ should coincide in order for the first the combined function to solve \eqref{eta.eq}. However note that their difference
\begin{equation*}
\lambda(\alpha)-\lambda(\alpha-2\alpha\cd e'e') =\alpha\cd\Lambda\alpha-(\alpha-2\alpha\cd e'e')\cd\Lambda(\alpha-2\alpha\cd e'e')+O(|\alpha|^3) =4\alpha\cd e'\alpha\cd(\Lambda e'-e'\cd\Lambda\cd e'e')+O(|\alpha|^3)
\end{equation*}
is of order $|\alpha|^2$ unless $\Lambda e'-e'\cd\Lambda\cd e'e'=0$, which means that to equate the eigenvalues, we would need to change $\alpha-2\alpha\cd e'e'$ by a large factor of order $O(|\alpha|)$. So instead we consider $\beta=\alpha-2\frac{\alpha\cd\Lambda e'}{e'\cd\Lambda e'}e'$ for the second term. Indeed, 
\begin{equation*}
\lambda(\alpha)-\lambda(\beta)  =\alpha\cd\Lambda\alpha-\left(\alpha-2\frac{\alpha\cd\Lambda e'}{e'\cd\Lambda e'}e'\right)\cd\Lambda\left(\alpha-2\frac{\alpha\cd\Lambda e'}{e'\cd\Lambda e'}e'\right)+O(|\alpha|^3) =O(|\alpha|^3)
\end{equation*}
The following lemma establishes the existence of $\beta$ as well as it's relation between between $\lambda$ and $\eta$ corresponding to $\alpha$ and $\beta$. 
%
\begin{lemma} \label{beta}
There exist 
$a_0>0$ such that for all $|\alpha|\le a_0$ satisfying the conditions in \eqref{alpha.cond} there exists $\beta(\alpha)$ satisfying 
\begin{equation*} 
\beta=\alpha-2\frac{\alpha\cd\Lambda e'}{e'\cd\Lambda e'}e'+O(|\alpha|^2)e',  \quad \quad \lambda(\beta)=\lambda(\alpha),
\end{equation*}
and
\begin{equation*} 
\sup\limits_{x\in\tn}|\bar{\eta}_\alpha(x)-1|\le C|\alpha|, \quad\quad \sup\limits_{x\in\tn}|\bar{\eta}_{\beta}(x)-1|\le C|\alpha|.
\end{equation*}
\end{lemma}
%
\begin{proof}
$\beta(\alpha)$ exists due to the asymptotics for $\lambda(\alpha)$ for $\alpha$ small. Indeed if $\beta=\alpha-2\frac{\alpha\cd\Lambda e'}{e'\cd\Lambda e'}e'+be'$, then
\begin{equation*}
\lambda(\beta)=\beta\cd\Lambda\beta+O(|\beta|^3)=\alpha\cd\Lambda\alpha-2b\alpha\cd\Lambda e'+b^2e'\cd\Lambda e'+O(|\alpha|^3)=\lambda(\alpha)-2b\alpha\cd\Lambda e'+O(|\alpha|^3),
\end{equation*}
which can be made to equal $\lambda(\alpha)$ for some $b=O(|\alpha|^2)$, since $\alpha\cd\Lambda e'$ and $|\alpha|$ are comparable. The last inequalities follow from elliptic regularity bounds on $\bar{\eta}_\alpha-\bar{\eta}_0=\bar{\eta}_\alpha-1$.
\end{proof}
Now we construct on increasing intervals an approximating sequence $\eta_{\alpha,T}(t,x)$ with terminal condition based on the ansatz. However, instead of using $-\alpha$ in the second term, we'll use $-\beta$ because the equations for both terms then will have the same coefficient for the linear term. 
\eal{approx.eta}{ll}{ 
\ds(\eta_{\alpha,T})_t+\frac{1}{\nu(x)}\nabla\cd(\nu(x)\nabla\eta_{\alpha,T}) +\frac{F(x)}{\nu(x)}\cd\nabla\eta_{\alpha,T} =\lambda(\alpha)\eta_{\alpha,T}, & \quad t<T, \ x\in\h{t},\\
\eta_{\alpha,T}(t,x)=0, & \quad t<T, \ x\in\partial\h{t},\\
\eta_{\alpha,T}(T,x)= \text{ to be determined} & \quad x\in\h{T},
}
with $\eta_{\alpha,T}(T,x)\ge 0$ to be given by functions of the form 
\begin{equation*} 
\frac{\e{\alpha\cd(x-\ws(e)t\,e)}}{\alpha\cd\Lambda e'}\bar{\eta}_\alpha(x) + C\frac{\e{\beta\cd(x-\ws(e)t\,e)}}{\beta\cd\Lambda e'}\bar{\eta}_{\beta}(x).
\end{equation*}
Note that $\beta\cd\Lambda e'=-\alpha\cd\Lambda e'+O(|\alpha|^2)$, and that  since $\lambda(\alpha)=\lambda(\beta)$, this satisfies \eqref{approx.eta} for every $C$. The free parameter $C$ will be adjusted such as to give an upper and lower bound for $\eta_{\alpha,T}$ at all times. These functions in turn will give the bounds analogous to \eqref{eta.bounds} that will allow us to extract a subsequence converging to $\eta_\alpha$.\\

\noindent Define 
\begin{equation*} 
H_l(t,x;\alpha)=\frac{\e{\alpha\cd(x-\ws(e)t\,e)}}{\alpha\cd\Lambda e'}\bar{\eta}_\alpha(x) + C_l(\alpha)\frac{\e{\beta\cd(x-\ws(e)t\,e)}}{\beta\cd\Lambda e'}\bar{\eta}_{\beta}(x).
\end{equation*}
and
\begin{equation*} 
H_u(t,x;\alpha)=\frac{\e{\alpha\cd(x-\ws(e)t\,e)}}{\alpha\cd\Lambda e'}\bar{\eta}_\alpha(x) + C_u(\alpha)\frac{\e{\beta\cd(x-\ws(e)t\,e)}}{\beta\cd\Lambda e'}\bar{\eta}_{\beta}(x).
\end{equation*}
with constants $C_l$ and $C_u$ for which 
\begin{equation*}
H_l(t,x)\le 0, \quad H_u(t,x)\ge 0 \quad \text{ for all } \ t\in\R, \ x\in\partial\h{t}.
\end{equation*}
A choice for them is 
\begin{equation*} 
C_l(\alpha)=\frac{\beta\cd\Lambda e'}{\alpha\cd\Lambda e'}\max\limits_{x\in\tn}\frac{\bar{\eta}_\alpha(x)}{\bar{\eta}_{\beta}(x)} \quad \text{and} \quad C_u(\alpha)=\frac{\beta\cd\Lambda e'}{\alpha\cd\Lambda e'}\min\limits_{x\in\tn}\frac{\bar{\eta}_\alpha(x)}{\bar{\eta}_{\beta}(x)}.
\end{equation*} 
Lemma \ref{beta} implies that for small $\alpha$,
\begin{equation*} 
C_l(\alpha)=1+O(|\alpha|)\quad \text{and} \quad C_u(\alpha)=1+O(|\alpha|).
\end{equation*}
Now choose 
\begin{equation*} 
\eta_{\alpha,T}(T,x)=\max\{0,H_l(T,x;\alpha)\}.
\end{equation*}
From the maximum principle, for all $t\le T$ and $x\in\h{t}$
\begin{equation} \label{eta.H} 
H_l(t,x;\alpha)\le\eta_{\alpha,T}(t,x)\le H_u(t,x;\alpha)
\end{equation}

\begin{lemma}
There exist $L>0$ and $M>1$ such that
\begin{equation} \label{eta.exp} 
M^{-1}\e{\alpha\cd x}\frac{1-\e{-2\frac{\alpha\cd\Lambda e'}{e'\cd\Lambda e'}x\cd e'}}{\alpha\cd\Lambda e'} \le\eta_{\alpha,T}(t,x+\ws(e)t\,e) \le M\e{\alpha\cd x}\frac{1-\e{-2\frac{\alpha\cd\Lambda e'}{e'\cd\Lambda e'}x\cd e'}}{\alpha\cd\Lambda e'}.
\end{equation}
for all $\alpha$ satisfying the conditions in \eqref{alpha.cond}, $T\in\R^+, t\le T$, and $x\in\{x\cd e'\ge L\}$.
\end{lemma}
\begin{proof}
To get the lower bound, from \eqref{eta.H} it is suffices to show 
\begin{equation*}
H_l(t,x+\ws(e)t\,e)\ge M_l\e{\alpha\cd x}\frac{1-\e{-2\frac{\alpha\cd\Lambda e'}{e'\cd\Lambda e'}x\cd e'}}{\alpha\cd\Lambda e'} 
\end{equation*}
for $x\cd e'$ greater then some constant $L$. By definition,
\begin{equation*} 
H_l(t,x+\ws(e)t\,e) =\frac{1}{\alpha\cd\Lambda e'}\bar{\eta}_{\alpha}(x+\ws(e)t\,e) \left(\e{\alpha\cd x} +C_l(\alpha)\frac{\alpha\cd\Lambda e'}{\beta\cd\Lambda e'}\frac{\bar{\eta}_{\beta}(x+\ws(e)t\,e)}{\bar{\eta}_{\alpha}(x+\ws(e)t\,e)}\e{\beta\cd x}\right).
\end{equation*}
From lemma \ref{beta}, for all $|\alpha|\le\alpha_0$
\begin{equation*} 
\beta=\alpha-2\frac{\alpha\cd\Lambda e'}{e'\cd\Lambda e'}e'+O(|\alpha|^2)e', \quad \bar{\eta}_{\alpha}(x+\ws(e)t\,e)=1+O(|\alpha|) \quad \text{and} \quad C_l(\alpha)\frac{\alpha\cd\Lambda e'}{\beta\cd\Lambda e'}\frac{\bar{\eta}_{\beta}(x+\ws(e)t\,e)}{\bar{\eta}_{\alpha}(x+\ws(e)t\,e)} =-1+O(|\alpha|).
\end{equation*}
Note that $\frac{1}{\alpha\cd\Lambda e'}\bar{\eta}_{\alpha}(x+\ws(e)t\,e)\ge\frac{1}{2}$ therefore decreasing $a_0$ if necessary, it suffices to show
\begin{equation*}
\frac{1}{2}\left(1+C_l(\alpha)\frac{\alpha\cd\Lambda e'}{\beta\cd\Lambda e'}\frac{\bar{\eta}_{\beta}(x+\ws(e)t\,e)}{\bar{\eta}_{\alpha}(x+\ws(e)t\,e)}\e{(\beta-\alpha)\cd x}\right) \ge M_l(1-\e{-2\frac{\alpha\cd\Lambda e'}{e'\cd\Lambda e'} x\cd e'}).
\end{equation*}
or equivalently,
\begin{equation} \label{low.M_l} 
1-2M_l \ge C_l(\alpha)\frac{\alpha\cd\Lambda e'}{\beta\cd\Lambda e'}\frac{\bar{\eta}_{\beta}(x+\ws(e)t\,e)}{\bar{\eta}_{\alpha}(x+\ws(e)t\,e)}\e{(\beta-\alpha)\cd x}-2M_l\e{-2\frac{\alpha\cd\Lambda e'}{e'\cd\Lambda e'} x\cd e'}.
\end{equation}
There are two cases to be considered, depending on whether $\frac{\alpha\cd\Lambda e'}{e'\cd\Lambda e'} x\cd e'$ is of order smaller or greater than $O(1)$. For small enough $|\alpha|$, $C_l(\alpha)\frac{\alpha\cd\Lambda e'}{\beta\cd\Lambda e'}\frac{\bar{\eta}_{\beta}(x+\ws(e)t\,e)}{\bar{\eta}_{\alpha}(x+\ws(e)t\,e)}=1+O(|\alpha|)\le 2$, and $(\beta-\alpha)\cd x=\left(-2\frac{\alpha\cd\Lambda e'}{e'\cd\Lambda e'}e'+O(|\alpha|^2)e'\right)\cd x\le-\frac{\alpha\cd\Lambda e'}{e'\cd\Lambda e'}x\cd e'$, so
\begin{equation*} 
C_l(\alpha)\frac{\alpha\cd\Lambda e'}{\beta\cd\Lambda e'}\frac{\bar{\eta}_{\beta}(x+\ws(e)t\,e)}{\bar{\eta}_{\alpha}(x+\ws(e)t\,e)}\e{(\beta-\alpha)\cd x} \le 2\e{-\frac{\alpha\cd\Lambda e'}{e'\cd\Lambda e'}x\cd e'}\le 1/2
\end{equation*}
for $x\cd e'\ge 4\frac{\ln 2e'\cd\Lambda e'}{\alpha\cd\Lambda e'}$, so \eqref{low.M_l} holds in this range for $M_l=1/4$.
On the other hand, if $x\cd e'\le 2\frac{\ln 2e'\cd\Lambda e'}{\alpha\cd\Lambda e'}$, then $\left(\beta-\alpha+2\frac{\alpha\cd\Lambda e'}{e'\cd\Lambda e'}e'\right)\cd x\le C|\alpha|^24\frac{\ln 2e'\cd\Lambda e'}{\alpha\cd\Lambda e'}\le C'|\alpha|$, so
\begin{equation*}\begin{array}{l} 
C_l(\alpha)\ds\frac{\alpha\cd\Lambda e'}{\beta\cd\Lambda e'}\frac{\bar{\eta}_{\beta}(x+\ws(e)t\,e)}{\bar{\eta}_{\alpha}(x+\ws(e)t\,e)}\e{(\beta-\alpha)\cd x}-M_l\e{-2\frac{\alpha\cd\Lambda e'}{e'\cd\Lambda e'} x\cd e'}
=\e{-2\frac{\alpha\cd\Lambda e'}{e'\cd\Lambda e'} x\cd e'}\left(C_l(\alpha)\ds\frac{\alpha\cd\Lambda e'}{\beta\cd\Lambda e'}\frac{\bar{\eta}_{\beta}(x+\ws(e)t\,e)}{\bar{\eta}_{\alpha}(x+\ws(e)t\,e)}\e{(\beta-\alpha+2\frac{\alpha\cd\Lambda e'}{e'\cd\Lambda e'}e')\cd x}-\frac{1}{2}\right)\\
\le (1-C_1|\alpha| x\cd e')((1+C_2|\alpha|)-1/2)\le 1/2 -(C_1x\cd e'-C_2)|\alpha|\le 1/2=1-M_l
\end{array}
\end{equation*}
for $x\cd e'\ge C_2/C_1\equiv L_l$, hence proving \eqref{low.M_l} for all $x\cd e'\ge L_l$.

Similarly for the upper bound we need
\begin{equation*} 
H_u(t,x+\ws(e)t\,e)\le M_u\e{\alpha\cd x}\frac{1-\e{-2\frac{\alpha\cd\Lambda e'}{e'\cd\Lambda e'}x\cd e'}}{\alpha\cd\Lambda e'},
\end{equation*}
or equivalently
\begin{equation*} 
\bar{\eta}_{\alpha}(x+\ws(e)t\,e) \left(1+C_u(\alpha)\ds\frac{\alpha\cd\Lambda e'}{\beta\cd\Lambda e'}\frac{\bar{\eta}_{\beta}(x+\ws(e)t\,e)}{\bar{\eta}_{\alpha}(x+\ws(e)t\,e)}\e{(\beta-\alpha)\cd x}\right) \le M_u\left(1-\e{-2\frac{\alpha\cd\Lambda e'}{e'\cd\Lambda e'}x\cd e'}\right).
\end{equation*}
In an analogous way, $\bar{\eta}_\alpha(x+\ws(e)t\,e)\le 2$, thus a sufficient condition would be
\begin{equation} \label{up.M_u} 
M_u\e{-2\frac{\alpha\cd\Lambda e'}{e'\cd\Lambda e'}x\cd e'} +2C_u(\alpha)\frac{\alpha\cd\Lambda e'}{\beta\cd\Lambda e'}\frac{\bar{\eta}_{\beta}(x+\ws(e)t\,e)}{\bar{\eta}_{\alpha}(x+\ws(e)t\,e)}\e{(\beta-\alpha)\cd x}\le M_u-2.
\end{equation}
For $M_u=4$ and $x\cd e\ge\frac{\ln 2\,e'\cd\Lambda e'}{4\alpha\cd\Lambda e'}$,
\begin{equation*} 
M_u\e{-\frac{\alpha\cd\Lambda e'}{e'\cd\Lambda e'}x\cd e'}\le 2=M_u-2.
\end{equation*}
For $x\cd e'\le\frac{\ln 2\,e'\cd\Lambda e'}{4\alpha\cd\Lambda e'}$, then $\left(\beta-\alpha+2\frac{\alpha\cd\Lambda e'}{e'\cd\Lambda e'}e'\right)\cd x\ge-C|\alpha|^24\frac{\ln 2e'\cd\Lambda e'}{\alpha\cd\Lambda e'}\ge -C'|\alpha|$, so
\begin{equation*}\begin{array}{l} 
M_u\e{-2\frac{\alpha\cd\Lambda e'}{e'\cd\Lambda e'}x\cd e'} +2C_u(\alpha)\frac{\alpha\cd\Lambda e'}{\beta\cd\Lambda e'}\frac{\bar{\eta}_{\beta}(x+\ws(e)t\,e)}{\bar{\eta}_{\alpha}(x+\ws(e)t\,e)}\e{(\beta-\alpha)\cd x}
=2\e{-2\frac{\alpha\cd\Lambda e'}{e'\cd\Lambda e'}x\cd e'}\left(2+C_u(\alpha)\ds\frac{\alpha\cd\Lambda e'}{\beta\cd\Lambda e'}\frac{\bar{\eta}_{\beta}(x+\ws(e)t\,e)}{\bar{\eta}_{\alpha}(x+\ws(e)t\,e)}\e{(\beta-\alpha+2\frac{\alpha\cd\Lambda e'}{e'\cd\Lambda e'}e')\cd x}\right)\\
\le 2(1-C_1|\alpha| x\cd e')(2-(1-C_2|\alpha|))\le 2(1-(C_1x\cd e'-C_2)|\alpha|)\le 2=M_u-2
\end{array}
\end{equation*}
for $x\cd e'\ge C_1/C_2\equiv L_u$.

So the result is true with $M=4$ and $L=\max\{L_l,L_u\}$.
\end{proof}

To extend \eqref{eta.exp} to all of $\h{t}$ we need to control $\eta_\alpha(t,x+\ws(e)t\,e)$ for $x\cd e'\in[0,L]$. Since we expect $\eta_{\alpha,T}(t,x)\simeq\e{\alpha\cd (x-\ws(e)t\,e)}$ near the boundary, let us express $\eta_{\alpha,T}(t,x)=\e{\alpha\cd(x-\ws(e)t\,e)}N_{\alpha,T}(t,x)$. Then $N_{\alpha,T}(t,x)$ solves
\ea{ll}{  
\ds(N_{\alpha,T})_t +\frac{1}{\nu(x)}\nabla\cd(\nu(x)\nabla N_{\alpha,T}) +\left(\frac{F(x)}{\nu(x)}+\alpha\right)\cd\nabla N_{\alpha,T} +\frac{\alpha}{\nu(x)}\cd\nabla(\nu(x)N_{\alpha,T})& \\
+\ds\alpha\cd\left(\frac{F(x)}{\nu(x)}-\ws(e)e\right)N_{\alpha,T} =(\lambda(\alpha)-|\alpha|^2)N_{\alpha,T}, & \quad t<T, \ x\in\h{t},\\
N_{\alpha,T}(t,x)=0, & \quad t<T, \ x\in\partial\h{t},\\
N_{\alpha,T}(T,x)=\e{-\alpha\cd(x-\ws(e)t\,e)}\max\{0,H_l(T,x;\alpha) \}  & \quad x\in\h{T}.
}
%
Note that the coefficients of the equation and $N_{\alpha,T}(T,x)$ are uniformly bounded for $|\alpha|\le a_0$, therefore by parabolic regularity there exists a constant $M>0$ such that for all $t\le T/2$ and $x\in\partial\h{t}$
\begin{equation*} 
|N_{\alpha,T}|\le M \quad \text{ and } \quad m\le\nabla N_{\alpha,T}(t,x)\cd e'\le M.
\end{equation*}
Using this, 
\begin{equation*} 
\nabla\eta_{\alpha,T}\cd e'=\e{\alpha\cd(x-\ws(e)t\,e)}\left(\alpha N_{\alpha,T}+\nabla N_{\alpha,T}\right)\cd e',
\end{equation*}
so taking $a_0\le 1/2$
\begin{equation*}
(2M)^{-1}\e{\alpha\cd(x-\ws(e)t\,e)} \le\nabla\eta_{\alpha,T}\cd e'\le 2M\e{\alpha\cd(x-\ws(e)t\,e)}
\end{equation*}

\noindent On the other hand
\begin{equation*} 
\nabla\left(\frac{\e{\alpha\cd(x-\ws(e)t\,e)} -\e{\alpha\cd(x-\ws(e)t\,e) -2\frac{\alpha\cd\Lambda e'}{e'\cd\Lambda e'}(x-\ws(e)t\,e)\cd e'}}{\alpha\cd\Lambda e'}\right)\cd e' =\frac{2}{e'\cd\Lambda e'}\e{\alpha\cd(x-\ws(e)t\,e)}
\end{equation*}
on $\partial\h{0}$, therefore by increasing $M$ if necessary, we get that
\begin{equation*} 
M^{-1}\e{\alpha\cd x}\frac{1-\e{-2\frac{\alpha\cd\Lambda e'}{e'\cd\Lambda e'}x\cd e'}}{\alpha\cd\Lambda e'} \le\eta_{\alpha,T}(t,x+\ws(e)t\,e) \le M\e{\alpha\cd x}\frac{1-\e{-2\frac{\alpha\cd\Lambda e'}{e'\cd\Lambda e'}x\cd e'}}{\alpha\cd\Lambda e'}
\end{equation*}
for all $T\in\R^+, t\le T/2, x\in\h{0}$. Now we can extract a subsequence $\eta_{\alpha,T_k}(t,x)$ that converges locally uniformly to $\eta_{\alpha}(t,x)$ while preserving all the bounds of lemma \ref{lemma.eta}.\\

Finally, to prove the bound in \eqref{eta.bounds} we need to compare $\frac{1-\e{-2\frac{\alpha\cd\Lambda e'}{e'\cd\Lambda e'}x\cd e'}}{\alpha\cd\Lambda e'}$ to $\frac{1-\e{-2\alpha\cd e'x\cd e'}}{\alpha\cd e'}$. This is where the condition $2^{-1}\alpha\cd e'\le\frac{\alpha\cd\Lambda e'}{e'\cd\Lambda e'}\le 2\alpha\cd e'$ in \eqref{alpha.cond} plays a role. In deed, for such $\alpha$ and $x\cd e'\ge 0$ the following inequalities hold:
\begin{equation*} 
\frac{1-\e{-2\frac{\alpha\cd\Lambda e'}{e'\cd\Lambda e'}x\cd e'}}{\alpha\cd\Lambda e'} \le\frac{2}{e'\cd\Lambda e'}\frac{1-\e{-4\alpha\cd e'x\cd e'}}{\alpha\cd e'} =\frac{2}{e'\cd\Lambda e'}\frac{(1-\e{-2\alpha\cd e'x\cd e'})(1+\e{-2\alpha\cd e'x\cd e'})}{\alpha\cd e'} \le\frac{4}{e'\cd\Lambda e'}\frac{1-\e{-2\alpha\cd e'x\cd e'}}{\alpha\cd e'},
\end{equation*}
and
\begin{equation*} 
\frac{1-\e{-2\frac{\alpha\cd\Lambda e'}{e'\cd\Lambda e'}x\cd e'}}{\alpha\cd\Lambda e'} \ge\frac{1}{2e'\cd\Lambda e'}\frac{1-\e{-\frac{1}{2}\alpha\cd e'x\cd e'}}{\alpha\cd e'} =\frac{1}{2e'\cd\Lambda e'}\frac{1-\e{-\alpha\cd e'x\cd e'}}{(1+\e{-\frac{1}{2}\alpha\cd e'x\cd e'})\alpha\cd e'} \ge\frac{1}{4e'\cd\Lambda e'}\frac{1-\e{-\alpha\cd e'x\cd e'}}{\alpha\cd e'}.
\end{equation*}
Combining the previous bounds on $\eta_\alpha$ with these inequalities, increasing $M$ if necessary, we conclude the proof of lemma \ref{lemma.eta}.

\subsubsection{Proof of proposition \ref{lower.estimate}}

To get the result in proposition \ref{lower.estimate} we first bound $\int x\cd e' p(t,x+\ws(e)t\,e)\,dx$ from above outside of a ball of radius $A\sqrt{t}$. This will be achieved by splitting the domain into $2n-1$ regions and then using the exponential moment bounds in \eqref{int.bound} in each of them. Let $\xi\in\rn$ be a vector satisfying $\xi\cd e'>0, \frac{1}{2}\xi\cd e'\le\frac{\xi\cd\Lambda e'}{e'\cd\Lambda e'}\le 2\xi\cd e'$, and $|\xi|\le\sqrt{n}\xi\cd e'$. Then for $\alpha=\frac{1}{\sqrt{t}}\xi$, \eqref{int.bound} gives
\begin{equation*} 
\left(\ih{0}\e{2\frac{x\cd\xi}{\sqrt{t}}}\frac{1-\e{-4\frac{\xi\cd e'x\cd e'}{\sqrt{t}}}}{2\xi\cd e'x\cd e'}p^2(t,x+\ws(e)t\,e)\,dx\right)^{1/2}
\le C\frac{\e{\frac{\lambda(2\xi/\sqrt{t})}{2}t}}{t^{\frac{n+1}{4}}}\ih{0}\e{\frac{x\cd\xi}{\sqrt{t}}}\frac{1-\e{-2\frac{\xi\cd e'x\cd e'}{\sqrt{t}}}}{\xi\cd e'}p_0(x)\,dx.
\end{equation*}
for $t\ge T_0$. By increasing $T_0$ if necessary for $x\in supp(p_0)$ we have 
\begin{equation*} 
\e{\frac{x\cd\xi}{\sqrt{t}}}\frac{1-\e{-2\frac{\xi\cd e'x\cd e'}{\sqrt{t}}}}{\xi\cd e'}\le 8\frac{x\cd e'}{\sqrt{t}}.
\end{equation*}
Moreover from lemma \ref{lemma.eta} for $T_0$ large, $\lambda(2\xi/\sqrt{t})\le 8\frac{\xi\cd\Lambda\xi}{t}$, so $\e{\frac{\lambda(2\xi/\sqrt{t})}{2}t}\le\e{4\xi\cd\Lambda\xi}$.  On the other hand if $\xi\cd e'x\cd e'\ge A\sqrt{t}$ for $A$ large enough, $1-\e{-4\frac{\xi\cd e'x\cd e'}{\sqrt{t}}}\ge 1/2$. So combining all together, for $A$ and $T_0$ large we obtain
\begin{equation*} 
\left(\ih{0}\frac{\e{2\frac{x\cd\xi}{\sqrt{t}}}}{\xi\cd e'x\cd e'}p^2(t,x+\ws(e)t\,e)\,dx\right)^{1/2}
\le C\frac{\e{4\xi\cd\Lambda\xi}}{t^{\frac{n+3}{4}}}\ih{0}{x\cd e'}p_0(x)\,dx \le C\frac{\e{4\xi\cd\Lambda\xi}}{t^{\frac{n+3}{4}}}I(0),
\end{equation*}
where $I(t)$ is defined as in \eqref{lin.mom}.

Without loss of generality, using a change of variables we may assume that $e'=e_1$ and $\Lambda e'=(\lambda_1,\ldots,\lambda_n)$, where $\lambda_1>0$ due to the positive definiteness of $\Lambda$. Let $\zeta=(\zeta_1,\ldots,\zeta_n)$ be such that $\zeta_i>0$ for all $i$ and $1/2\lambda_1\zeta_1\le\sum|\lambda_i\zeta_i|\le 2\lambda_1\zeta_1$, $|\zeta|\le\sqrt{n}\zeta_1$, so that any vector of the form $(\zeta_1,\pm\zeta_2,\ldots,\pm\zeta_n)$ would be a candidate for $\xi$ in the inequality above.

For any choice of signs $\sigma\in{1}\times\{\pm 1\}^{n-1}$, let $\xi^\sigma=(\sigma_1\zeta_1,\ldots,\sigma_n\zeta_n)$. We will divide the region of interest into $2n-1$ regions as follows: $R_t^\sigma=\{\xi^\sigma_ix_i\ge |\xi^\sigma_i|B\sqrt{t} : i=1,\ldots,n\}=\{x_1\ge 0, \sigma_ix_i\ge B\sqrt{t} : i=2,\ldots,n\}$ and $R_t^{e_1}=\{x_1\ge A\sqrt{t}, |x_i|\le B\sqrt{t}\}$. Then taking $\xi$ in the inequality above to be $\xi^\sigma$ in $R_t^\sigma$ we get
\begin{equation*} 
\begin{array}{l}
\ds\int\limits_{R_t^\sigma}\nu(x)f(t,x+\ws(e)t\,e)p(t,x+\ws(e)\,te)\,dx\le C\int\limits_{R_t^\sigma}x_1p(t,x+\ws(e)\,te)\,dx\\
\ds=C\int\limits_{R_t^\sigma}\frac{\e{\frac{x\cd\xi}{\sqrt{t}}}}{\sqrt{\xi_1x_1}}p(t,x+\ws(e)\,te)\sqrt{\xi_1}(x_1)^{3/2}\e{-\frac{x\cd\xi}{\sqrt{t}}}\,dx
\le C\left(\int\limits_{\h{0}}\frac{\e{2\frac{x\cd\xi}{\sqrt{t}}}}{\xi_1x_1}p^2(t,x+\ws(e)\,te)\,dx\right)^{1/2} \left(\int\limits_{R_t^\sigma}x_1^{3}\e{-2\frac{x\cd\xi}{\sqrt{t}}}\,dx\right)^{1/2}\\
\ds\le\frac{C}{t^{\frac{n+3}{4}}}I(0) \left(\int\limits_{\xi_1x_1\ge 0} x_1^{3}\e{-\frac{2\xi_1}{\sqrt{t}}x_1}\,dx_1 \int\limits_{\xi_2x_2\ge|\xi_2|B\sqrt{t}} \e{-\frac{2\xi_2}{\sqrt{t}}x_2}\,dx_2 \ldots \int\limits_{\xi_nx_n\ge|\xi_n|B\sqrt{t}} \e{-\frac{2\xi_n}{\sqrt{t}}x_n}\,dx_n\right)^{1/2}\\
\ds\le\frac{C}{t^{\frac{n+3}{4}}}I(0) \left(t^2\int\limits_{0}^\infty y_1^{3}\e{-y_1}\,dy_1 \times\sqrt{t}\int\limits_{B}^\infty\e{-y_2}\,dy_2 \times\ldots\times \sqrt{t}\int\limits_{B}^\infty \e{-y_n}\,dy_n\right)^{1/2}\le\phi(B)I(0),
\end{array}
\end{equation*}
where $\phi(B)$ is a function that decays to $0$ as $B\to\infty$. Similarly we may use $\xi=e_1$ in $R^{e_1}$ to get

\begin{equation*} 
\begin{array}{l}
\ds\int\limits_{R_t^{e_1}}\nu(x)f(t,x+\ws(e)t\,e)p(t,x+\ws(e)\,te)\,dx \le C\left(\int\limits_{\h{0}}\frac{\e{2\frac{x_1}{\sqrt{t}}}}{x_1}p^2(t,x+\ws(e)\,te)\,dx\right)^{1/2} \left(\int\limits_{R_t^{e_1}}x_1^{3}\e{-2\frac{x_1}{\sqrt{t}}}\,dx\right)^{1/2}\\
\ds\le\frac{C}{t^{\frac{n+3}{4}}}I(0) \left(\int\limits_{x_1\ge A\sqrt{t}} x_1^{3}\e{-\frac{2\xi_1}{\sqrt{t}}x_1}\,dx_1 \int\limits_{|x_2|\le B\sqrt{t}}1\,dx_2 \ldots \int\limits_{|x_n|\le B\sqrt{t}}1\,dx_n\right)^{1/2}\\
\ds\le\frac{C}{t^{\frac{n+3}{4}}}I(0) \left(t^2\int\limits_{A}^\infty y_1^{3}\e{-y_1}\,dy_1 \times\sqrt{t}\int\limits_{-B}^B1\,dy_2 \times\ldots\times \sqrt{t}\int\limits_{-B}^B1\,dy_n\right)^{1/2}\le\phi(A)\Phi(B)I(0),
\end{array}
\end{equation*}
where $\Phi(B)$ increases as $B\to\infty$. 

Finally let $R'_t=\{x: 0\le x_1\le A'\sqrt{t}, |x_i|\le B\sqrt{t}, i=2,\ldots,n\}$. Using the upper bound in proposition \ref{linearbounds} in this region,
\begin{equation*} 
\ds\int\limits_{R'_t}\nu(x)f(t,x+\ws(e)t\,e)p(t,x+\ws(e)\,te)\,dx \le C\ds\int\limits_{R'_t}x_1p(t,x+\ws(e)\,te)\,dx \le \frac{C}{t^{\frac{n+2}{2}}}\int\limits_{R'}x_1^2\,dx =\psi(A'),
\end{equation*}
with $\psi(A')$ decreasing as $A'\to 0$. \\

\noindent Combining all of the above together we have that on $$D_t(A',A,B)=\{A'\sqrt{t}\le x_1\le A\sqrt{t}, |x_i|\le B\sqrt{t}: i=2,\ldots,n\}=\h{0}\backslash\{R^{e_1}_t\cup R'_t\cup(\cup_{\sigma\in 1\times\{\pm 1\}^{n-1}} R^\sigma_t)\},$$
with $A, A'$ and $B$ to be chosen later, the same integral is bounded from below by 
\begin{equation*}
\begin{array}{l}
\ds\int\limits_{D_t}\nu(x)f(t,x+\ws(e)t\,e)p(t,x+\ws(e)\,te)\,dx  =I(t)-\int\limits_{R^{e_1}_t\cup R'_t\cup(\cup_{\sigma} R^\sigma_t)}\nu(x)f(t,x+\ws(e)t\,e)p(t,x+\ws(e)\,te)\,dx\\
\ge I(t)-(\phi(B)I(0)+\psi(A')+(2n-2)\phi(A)\Phi(B)I(0))\ge \frac{1}{2}I(0).
\end{array}
\end{equation*}
To get the last inequality we first pick $B$ large enough for $\phi(B)\le 1/6$, $A'$ small enough for $\psi(A')\le I(0)/6$, and finally $A$ large enough (depending on $B$) that $(2n-2)\phi(A)\Phi(B)\le 1/6$. Note that we also used the fact that $I(t)=I(0)$.

For this choice of $A,A',B$, divide $D_t$ into two subregions as follows: $D_t=D_t^+\cup D_t^-$, where $D_t^+=\{x\in D_t: x\ge \frac{c_0}{t^{\frac{n+1}{2}}}\}$ and $D_t^+=\{x\in D_t: x\ge \frac{c_0}{t^{\frac{n+1}{2}}}\}$, with $c_0>0$ to be picked later. Then 
\begin{equation*}
\frac{1}{2}I(0) \le\int\limits_{D_t^+}\nu(x)f(t,x)p(t,x)\,dx +\int\limits_{D_t^-}\nu(x)f(t,x)p(t,x)\,dx \le\int\limits_{D_t^+}\nu(x)f(t,x)p(t,x)\,dx +Cc_0A^2B^{n-1}, 
\end{equation*}
so choosing $c_0$ so that $Cc_0A^2B^{n-1}=I(0)/4$, we end up with 
\begin{equation*}
\frac{1}{4}I(0) \le\int\limits_{D_t^+}\nu(x)f(t,x)p(t,x)\,dx.
\end{equation*}
Using once again the upper bound in proposition \ref{linearbounds},
\begin{equation*}
\frac{1}{4}I(0) \le\int\limits_{D_t^+}\nu(x)f(t,x)p(t,x)\,dx \le C\int\limits_{D_t^+}\frac{I(0)}{t^{\frac{n+2}{2}}}x_1^2\,dx \le C\frac{1}{t^{\frac{n+2}{2}}}I(0)A^2t|D_t^+|,
\end{equation*}
which in turn implies that $|D_t^+|\ge \frac{1}{4CA^2}t^{\frac{n}{2}}$, concluding therefore the proof of proposition \ref{lower.estimate}.

\subsubsection{End of proof of lower bound in proposition \ref{linearbounds}}

To obtain a lower bound on $p(t,x)$ we will combine the result in proposition \ref{lower.estimate} with the heat kernel bounds from lemma \ref{heat.kernel.bounds}. Indeed we have that
\begin{equation*}
p(t,x)=\ih{s}\nu(y)p(s,y)\bar{\Gamma}(t,x,s,y)\,dy,
\end{equation*}
where
\begin{equation*}
\bar{\Gamma}(t,x+\ws(e)t\,e,s,y+\ws(e)s\,e)=\bar{\Gamma}(t,x,s,y-\ws(e)(t-s)\,e) \ge\frac{a}{K(t-s)^{\frac{n}{2}}}\e{-K\frac{|x-y|^2}{t-s}} 
\end{equation*}
for $\delta\in(0,1), R>0, x_0\in\rn, t-s\in [0,R^2], x,y\in B_{\delta R}(x_0)$, and
\begin{equation*}
p(s,y+\ws(e)s\,e)\ge\frac{c_0}{s^{\frac{n+1}{2}}} \quad \text{ for } y\in D_s^+\subseteq \{y\cd e'\in(A'\sqrt{s},A\sqrt{s}), |y^{\perp}|\le B\sqrt{s}\}.,
\end{equation*}
In order to use both bounds above, we need to pick $x_0$, $\delta$ and $R$ in such a way that $D_s^+\subseteq B_{\delta R}(x_0)$ and $s, t-s$ and $t$ are comparable. Let $s=\epsilon t$, with $\epsilon\in(0,1)$ to be determined later. Then $R=\sqrt{1-\epsilon}\sqrt{t}$. If $D_s^+\subseteq \{y\cd e'\in(A'\sqrt{s},A\sqrt{s}), |y^{\perp}|\le B\sqrt{s}\}\subseteq B_{\delta R}(x_0)$ is satisfied for some $\delta$ and $x_0$, then for $x\in B_{\delta R}(x_0)$,
\begin{equation*}
\begin{array}{l}
\ds p(t,x+\ws(e)t\,e)\ge \ds C\int\limits_{D_s^+}p(s,y+\ws(e)s\,e)\bar{\Gamma}(t,x+\ws(e)t\,e,s,y+\ws(e)s\,e)\,dy\\
\ds\ge C|D^+_s|\min\limits_{y\in D^+_s}\{p(s,y+\ws(e)s\,e)\bar{\Gamma}(t,x+\ws(e)t\,e,s,y+\ws(e)s\,e)\} \ge \frac{C}{\sqrt{s}(t-s)^{\frac{n}{2}}}=\frac{C}{t^\frac{n+1}{2}},
\end{array}
\end{equation*}
with the constant $C$ depending on $\epsilon$ but not on $t$. This is in particular true as long as
\begin{equation*}
\{y\cd e'\in(A'\sqrt{\epsilon}\sqrt{t},A\sqrt{\epsilon}\sqrt{t}), |y^{\perp}|\le B\sqrt{\epsilon}\sqrt{t}\}\subseteq B_{\delta\sqrt{1-\epsilon}\sqrt{t}}(x_0).
\end{equation*}
So we need to pick $x_0,\delta$ and $R$ to satisfy this, along with $B\sqrt{\epsilon}\sqrt{t}\}\subseteq B_{\delta\sqrt{1-\epsilon}\sqrt{t}}(x_0)\subseteq \h{0}$. Choosing $x_0=R=\sqrt{1-\epsilon}\sqrt{t}e'$, these conditions produce the system
\begin{equation*}
(1+\delta)R\ge A\sqrt{s}, \quad 0\le(1-\delta)R\le A'\sqrt{s}, \quad (\delta R)^2\ge(A\sqrt{s}-R)^2+(n-1)^2B^2s,
\end{equation*}
or equivalently
\begin{equation*}
1+\delta\ge A\sqrt{\frac{\epsilon}{1-\epsilon}}, \quad 0\le 1-\delta\le A'\sqrt{\frac{\epsilon}{1-\epsilon}}, \quad \delta^2\ge\left(A\sqrt{\frac{\epsilon}{1-\epsilon}}-1\right)^2+\left((n-1)\sqrt{\frac{\epsilon}{1-\epsilon}}B\right)^2.
\end{equation*}
Define $r=\sqrt{\frac{\epsilon}{1-\epsilon}}$. Choose $\epsilon\in(0,1)$ so that $r=\frac{1}{A+A'/4}$ and $\delta$ to be $\delta=\frac{A-A'/4}{A+A'/4}\in(0,1)$. For the third condition above to be satisfied we need $(A-A'/4)^2\ge(A'/4)^2+(n-1)^2B^2$. But we can always assure that this holds, because in the proof of proposition \ref{lower.estimate} we first choose $B$ and $A'$ to insure two of the integrals are small, then $A$ to be bigger than a constant depending on $B$. Therefore increasing $A$ if necessary we conclude that with this choice of $x_0,\epsilon$ and $\delta$ all the conditions are satisfied.

With the above choice of constants, we get that the lower bound $p(t,x+\ws(e)e)\ge\frac{C}{t^{\frac{n+1}{2}}}$ holds for all $x\in B_{\delta R}(x_0)=B_{\frac{A-A,/4}{A+A'/4}\sqrt{1-\epsilon}\sqrt{t}}(\sqrt{1-\epsilon}\sqrt{t}e')=B_{\rho\sqrt{t}}(\sigma\sqrt{t}e')$.

\section{Linearized problem in the log-corrected space - proofs} \label{logspace}

This section is dedicated to proving Proposition \ref{p.bar.est} by combining a decay estimate in positions of order $O(\sqrt{t})$ together with a comparison to bounds from an approximate solution in a bounded domain. The estimates for $\sqrt{t}$ locations generalize those in  Proposition \ref{linearbounds} and are proved in section \ref{sec.sqrtdecay}.

Section \ref{sec.approx} is dedicated to obtaining bounds on the solution of the equation in domains of diameter $\sqrt{t}$ located at the front of $\ws(e)t\,e$. This is done by first constructing a function $q$ that solves an approximate equation instead of \eqref{pbar}, and has the advantage that has the exact asymptotics that the solution of the exact equation is expected to have at domains of size $\sqrt{t}$.

In the final section \ref{sec.logprop} we combine Propositions  \ref{sqrttdecay} and \ref{approxsol} to prove the main result.

\subsection{Decay at positions of order $O(\sqrt{t})$}\label{sec.sqrtdecay}
The first step towards the proof of the main proposition is the observation that at locations of the form $(\ws(e)t+x_0+\epsilon\sqrt{t})e$, $\bar{p}$ is of order $\frac{1}{t^{\frac{n+1}{2}}}$.

\begin{prop}\label{sqrttdecay}
Given $x_0>0$ and $\epsilon>0$, there exist $\epsilon'>0$ and $C_\epsilon>0$ such that 
\begin{equation*}
\frac{1}{C_\epsilon t^{\frac{n+1}{2}}}\le\bar{p}(t,(\ws(e)t+x_0+\epsilon\sqrt{t})e+x)\le \frac{C_\epsilon}{t^{\frac{n+1}{2}}}
\end{equation*}
for all $x\in B_{\epsilon'\sqrt{t}}(0)$ and $t\ge T_0$.
\end{prop}
%
The proof of this result is similar to Proposition \ref{linearbounds}. Just as in that proof, we try to control the first moment of $\bar{p}$.

\begin{lemma} \label{int.bounds.pbar}
There exists $C>0$ such that 
\begin{equation*}
C^{-1}\le\ih{t}(x-\ws(e)t\,e)\cd e'\bar{p}(t,x)\,dx\le C.
\end{equation*}
\end{lemma}
\begin{proof}
Define $\bar{I}(t)$ to be
\begin{equation*}
\bar{I}(t):=(1-\beta(t))\ih{t}\nu(x)f(t,x)\bar{p}(t,x)\,dx,
\end{equation*}
where $f$ is the function from Lemma \ref{lemma.backwards} with bounds of the form $(x-\ws(e)t\,e)\cd e'$. Its derivative will be
\begin{equation*}
\frac{d\bar{I}(t)}{dt}=-\beta'(t)\ih{t}\nu\bar{p}f\,dx-\beta(t)\ih{t}\nu\bar{p}\partial_tf\,dx =O\left(\frac{1}{t^2}\right)\bar{I}(t)-O\left(\frac{1}{t}\right)\ih{t}\nu\bar{p}\partial_tf\,dx.
\end{equation*}
To bound the second integral we divide the domain $\h{t}$ into two regions:
\begin{equation*}
R^1_t:=\{ x\cd e'\in(\cs(e)t,\cs(e)t+t^{1/4})\} \quad \text{and}  \quad R^2_t:=\{x\cd e'\in(\cs(e)t+t^{1/4},\infty)\}.
\end{equation*}
Using the bounds on $f$, as well as the fact that $|\partial_tf|\le C$ by parabolic regularity we obtain
\begin{equation*}
|\bar{I}_2(t)|=\left|\int\limits_{R^2_t}\nu\partial_tf\bar{p}\,dx\right| \le\frac{C}{t^{1/4}}\int\limits_{R^2_t}\nu(x-\ws(e)t\,e)\cd e'\bar{p}\,dx \le\frac{C}{t^{1/4}}\int\limits_{R^2_t}\nu f\bar{p}\,dx \le\frac{C}{t^{1/4}}\bar{I}(t).
\end{equation*}
On the other hand on $R^1_t$ we compare $\bar{p}$ to the solutions of the same equation in the whole space, i.e.
\begin{equation*}
\ds(1-\beta(t))\bar{P}_t=\frac{1}{\nu(x)}\nabla\cd(\nu(x)\nabla\bar{P})-\frac{F(x)}{\nu(x)}\cd\nabla\bar{P}, \quad t>0, \ x\in\rn.
\end{equation*}
Using a change of variables $t\to\tau$, with $d\tau=(1-\beta(t))dt$ we can see that for large $t$, the heat kernel for this equation has the same bounds as for the case $\beta=0$, which is given by Norris in Lemma \ref{lemma.Norris}. Hence, for large $t$,
\begin{equation*}
\bar{P}(t,x+\ws(e)t\,e)\le\frac{C}{t^{\frac{n}{2}}}\e{-\frac{|x|^2}{Kt}}.
\end{equation*}
Then using this in $R^1_t$ we obtain
\begin{equation*}
|\bar{I}_1(t)|=\left|\int\limits_{R^1_t}\nu\partial_tf\bar{p}\,dx\right| \le C\int\limits_{R^1_t}\frac{C}{t^{\frac{n}{2}}}\e{-\frac{|x-\ws(e)t\,e|^2}{Kt}}\,dx \le\frac{C}{t^{\frac{1}{2}}}\int\limits_{t^{1/4}}^\infty\e{-\frac{x_1^2}{Kt}}\,dx_1 =C\int\limits_{\frac{1}{t^{1/4}}}^\infty\e{-\frac{y_1^2}{K}}\,dy_1 \le\frac{C}{t^{1/4}}.
\end{equation*}
Combining these,
\begin{equation*}
\frac{d\bar{I}(t)}{dt} =O\left(\frac{1}{t^2}\right)\bar{I}(t)+O\left(\frac{1}{t}\right)(\bar{I}_1(t)+\bar{I}_2(t)) \le O\left(\frac{1}{t^2}\right)\bar{I}(t)+O\left(\frac{1}{t^\frac{5}{4}}\right)+O\left(\frac{1}{t^\frac{5}{4}}\right)\bar{I}(t),
\end{equation*}
which implies $\bar{I}(t)\le C+O\left(\frac{1}{t^\frac{1}{4}}\right)$, i.e. $\bar{I}(t)\le C$ for all $t$. 

The lower bound on the other hand is easier, since $\partial_tf<0$ from \ref{lemma.backwards}, so $\bar{I}'(t)\ge O\left(\frac{1}{t^2}\bar{I}(t)\right)$, which implies $\bar{I}(t)\ge C\bar{I}(0)>0$ for all $t$.
\end{proof}

To obtain a decay of order $O(\frac{1}{t^{\frac{n+1}{2}}})$ we need to control the exponential moments of $\bar{p}$. Just as in the proof of Proposition \eqref{linearbounds} define 
\begin{equation*}
\bar{I}_\alpha(t):=(1-\beta(t))\ih{t}\nu(x)\eta_{\alpha}(t,x)\zeta(t,x)\bar{q}(t,x)\,dx, \quad \bar{V}_\alpha(t):=(1-\beta(t))\ih{t}\nu(x)\eta_{2\alpha}(t,x)\zeta(t,x)\bar{q}^2(t,x)\,dx,
\end{equation*}
\begin{equation*}
\bar{D}_\alpha(t):=\ih{t}\nu(x)\eta_{2\alpha}(t,x)\zeta(t,x)|\nabla \bar{q}(t,x)|^2\,dx,
\end{equation*}
where $\zeta(t,x)$ is the linearly growing function in Lemma \ref{lemma.zeta} and $\bar{q}(t,x)=\frac{\bar{p}(t,x)}{\zeta(t,x)}$. Then
\begin{equation*}
\frac{d\bar{I}_\alpha(t)}{dt} =(\lambda(\alpha)-\beta'(t))\bar{I}_\alpha(t) -\beta(t)\ih{t}\nu(\eta_\alpha)_t\zeta \bar{q}\,dx,
\end{equation*}
\begin{equation*}
\ds\frac{d\bar{V}_\alpha(t)}{dt} =(\lambda(2\alpha)-\beta'(t))\bar{V}_\alpha(t)-2\bar{D}_\alpha(t) +\beta(t)\ih{t}\nu\eta_{2\alpha}\zeta_t\bar{q}^2\,dx -\beta(t)\ih{t}\nu(\eta_{2\alpha})_t\zeta \bar{q}^2\,dx.
\end{equation*}
In order to proceed as in the case of $p$ we need to control the additional integrals in both derivatives, which amounts to controlling $(\eta_\alpha)_t$ and $\zeta_t$. This is proved in the following lemma:

\begin{lemma}\label{l1bound}
There exists $C>0$ such that for all $\alpha$ satisfying the conditions in \eqref{alpha.cond} and $x\in\h{t}$\\
\textbf{(i)} $|\zeta_t(t,x)|\le C$,\\
\textbf{(ii)} $|\bar{\eta}_\alpha(t,x)|\le C|\alpha|\eta_\alpha(t,x)$.
\end{lemma}
\begin{proof}
\textbf{(i)} This follows from parabolic regularity.\\
\textbf{(ii)} We will prove this initially for $\eta_{\alpha,T}$, then let $T\to\infty$. Note first that $(\eta_{\alpha,T})_t$ solves the same equation as $\eta_{\alpha,T}$. Taking into account the comparability conditions for $\alpha$, at the terminal condition
\begin{equation*}
(\eta_{\alpha,T})_t(T,x)=O(\e{\alpha\cd(x-\ws(e)t\,e)})+\e{\alpha\cd(x-\ws(e)t\,e)}d\mu_\alpha(x),
\end{equation*}
where $\mu_\alpha(x)$ is the measure coming from the set of zeros of $h_l$, of mass uniformly bounded in $\alpha$. Then at $T-1$,
\begin{equation*}
(\eta_{\alpha,T})_t(T-1,x) =O(\e{\alpha\cd(x-\ws(e)t\,e)})+\e{\alpha\cd(x-\ws(e)t\,e)}O(1) =O(\e{\alpha\cd(x-\ws(e)t\,e)}),
\end{equation*}
which in turn implies that for $t\le T-1$,
\begin{equation*}
|(\eta_{\alpha,T})_t(t,x)| =O(\e{\alpha\cd(x-\ws(e)t\,e)}) \le C\e{\alpha\cd(x-\ws(e)t\,e)}\bar{\eta}_\alpha(x).
\end{equation*}
Since this is true for all $T$, letting $T\to\infty$ 
\begin{equation*}
|(\eta_{\alpha})_t(t,x)|\le C\e{\alpha\cd(x-\ws(e)t\,e)}\bar{\eta}_\alpha(x) \le C|\alpha|\eta_{\alpha}(t,x).
\end{equation*}
\end{proof}
Using the bounds on $\beta$ and part (ii) of this lemma we bound the terms with $(\eta_{\alpha})_t$ in both equations above as follows 
\begin{equation*}
\left|\beta(t)\ih{t}\nu(\eta_{\alpha})_t\zeta \bar{q}\,dx\right|\le C|\beta(t)|\ih{t}\nu\eta_{\alpha}\zeta \bar{q}\,dx \le \frac{C}{t}\bar{I}_\alpha(t),
\end{equation*}
\begin{equation*}
\left|\beta(t)\ih{t}\nu(\eta_{2\alpha})_t\zeta \bar{q}^2\,dx\right|\le C|\beta(t)|\ih{t}\nu\eta_{2\alpha}\zeta \bar{q}^2\,dx \le \frac{C}{t}\bar{V}_\alpha(t).
\end{equation*}
Finally for the term involving $\zeta_t$ we use part (i) of the lemma, together with the following observations:\\
\textbf{(I)} In the region $(x-\ws(e)t\,e)\cd e'\ge 1$, $|\zeta_t(t,x)|\le C\le C(x-\ws(e)t\,e)\cd e'\le C\zeta(t,x)$.\\
\textbf{(II)} In the region $0\le(x-\ws(e)t\,e)\cd e'\le 1$ we note the following fact: for every $M>0$ there exists $C(M)>0$ such that for all functions $f\in C^1([0,1])$, $f\ge 0$,
\begin{equation*}
|f'(y)|\le M\int\limits_{0}^1f(z)\,dz, \forall y\in[0,1] \quad \implies \quad \int\limits_{0}^1f(y)\,dy \le C(M)\int\limits_{0}^1yf(y)\,dy.
\end{equation*}
This is true because if not, there would exist a sequence of functions $f_n$ with mass $1$, uniformly bounded derivatives, and first moments going to $0$, which by Arzela-Ascoli would lead to a subsequence having a continuous limiting function of mass $1$ and first moment $0$, a contradiction. We now use this fact on  $\ih{t}\nu\eta_{2\alpha}\zeta_t\bar{q}^2\,dx$ in $0\le(x-\ws(e)t\,e)\cd e'\le 1$ in the variable along $(x-\ws(e)t\,e)\cd e'$ to get
\begin{equation*}
\begin{array}{c}
\ds\left|\int\limits_{\stackrel{(x-\ws(e)t\,e)\cd e'}{\in(0,1)}} \nu\eta_{2\alpha}\zeta_t\bar{q}^2\,dx\right| \le C\int\limits_{\stackrel{(x-\ws(e)t\,e)\cd e'}{\in(0,1)}} \nu\eta_{2\alpha}\bar{q}^2\,dx \le C\int\limits_{\stackrel{(x-\ws(e)t\,e)\cd e'}{\in(0,1)}} \nu\eta_{2\alpha}(x-\ws(e)t\,e)\cd e'\bar{q}^2\,dx\\
\ds \le C\int\limits_{\stackrel{(x-\ws(e)t\,e)\cd e'}{\in(0,1)}} \nu\eta_{2\alpha}\zeta\bar{q}^2\,dx.
\end{array}
\end{equation*}

Combining \textbf{(I)} and \textbf{(II)} we get that 
\begin{equation*}
\left|\beta(t)\ih{t}\nu\eta_{2\alpha}\zeta_t\bar{q}^2\,dx\right|\le\frac{C}{t}\bar{V}_\alpha(t).
\end{equation*}
Similarly as in the proof of Proposition \ref{linearbounds} we have 
\begin{equation*}
\bar{I}_\alpha' =\left(\lambda(\alpha)+O\left(\frac{1}{t}\right)\right)\bar{I}_\alpha,
\end{equation*}
and invoking the Nash type inequality \eqref{Nash.n+2} together with Lemma \ref{int.bounds.pbar},
\begin{equation*}
\begin{array}{c}
\ds\bar{V}_\alpha'=\left(\lambda(2\alpha)+O\left(\frac{1}{t}\right)\right)\bar{V}_\alpha-2\bar{D}_\alpha \le \left(\lambda(2\alpha)+O\left(\frac{1}{t}\right)\right)\bar{V}_\alpha-C\bar{V}_\alpha^{\frac{n+4}{n+2}}\bar{I}_\alpha^{-\frac{4}{n+2}}\\ \ds=\left(\lambda(2\alpha)+O\left(\frac{1}{t}\right)\right)\bar{V}_\alpha-C\bar{V}_\alpha^{\frac{n+4}{n+2}}\bar{I}_\alpha(0)^{-\frac{4}{n+2}}\e{\sigma(\alpha)t}.
\end{array}
\end{equation*}
%
Fix $T=|\alpha|^{-2}$. Note that we only need the integral inequality for $|\alpha|=O\left(\frac{1}{\sqrt{t}} \right)$, so it is enough to prove the inequality up to $t\le T$. But note that for $M>0$ large enough, $Mt^{-\frac{n+2}{2}}$ is a super-solution for $t\le T$, so $\bar{V}_\alpha(t)\le Mt^{-\frac{n+2}{2}}$ for all $t\le T$. Using this, the rest of the proof follows exactly the lines of Proposition \ref{linearbounds}.

\subsection{An approximate solution}\label{sec.approx}

\noindent The second ingredient in the proof of the main result is the construction of an approximate solution $\bar{p}^{app}$ on domains of diameter $\sqrt{t}$ around locations of the form $(\ws(e)t+x_0+\epsilon\sqrt{t})e$.To do this, we will need the following function $q$ whose behavior is more or less explicit, and which solves an approximate version of the equation.

\begin{prop}\label{qfunction}
Let  $\chi_i(x;e')$ be the functions satisfying \eqref{chi} and $\chi_0\in\R$. There exists a function $q(t,x)$ such that for any $\sigma>0$,
\begin{equation*}
(1-\beta(t))q_t=\Delta q+\kappa(x)\cd\nabla q +O\left(\frac{1}{t^{\frac{n+5}{2}}} \right)
\end{equation*} 
for $x=y+(\ws(e)t+r)e$ satisfying $r\in (0,\sigma\sqrt{t})$, $y\cd e'=0, y\in B^\perp_{\sigma\sqrt{t}}(0)$ and $t>1$, and such that within this interval,
\begin{equation} \label{qbounds}
\begin{array}{c}
\ds\left| q(t,x)-\frac{(x-\ws(e)t\,e-\chi(x)/\ls(e')-\chi_0)\cd e'}{t^{\frac{n}{2}+1}} \e{-\frac{(x-\ws(e)t\,e)\cd S^{-1}(x-\ws(e)t\,e)}{4t}} \right|\\ \ds\le\frac{C}{t^{\frac{n}{2}+1}}\left(\frac{x\cd e'-\cs(e')t}{\sqrt{t}} \right)^2+ O\left(\frac{1}{t^{\frac{n+3}{2}}}\right)
\end{array}
\end{equation}
for some $C>0$.
\end{prop}
%
\noindent$\bar{p}^{app}$ will now be a solution of the exact equation on a bounded moving domain that is controlled by $q$ and the bounds from the previous proposition.
%
\begin{prop}\label{approxsol}
Let $\sigma\in (0,1)$ be fixed, $q(t,x)$ the function in Proposition \ref{qfunction} and $C_t$ the oblique cylinder
$$C_t:=\left\{x=y+(\ws(e)t+r)e: \ 
r\in\left(0,\sigma\sqrt{t}\right), \ |y|<\sigma\sqrt{t}, \ y\cd e'=0\right\}.$$
Let $\bar{p}^{app}(t,x)$ solve 
\eal{approxsoleq}{ll}{
\ds(1-\beta(t))\partial_t\bar{p}^{app}=\Delta\bar{p}^{app}+\kappa(x)\cd\nabla\bar{p}^{app}, \quad & x\in C_t, \ t>1, \\
\ds \bar{p}^{app}(t,x)=q(t,x), & x\in\partial C_t, \ t>1,\\
\ds \bar{p}^{app}(1,x)=q(1,x), & x\in C_1.
}
Then for all $t\ge 1$ and $x\in C_t$,
\begin{equation*}
|\bar{p}^{app}(t,x)-q(t,x)|\le\frac{C\sigma^{\frac{1}{n+1}}}{t^{\frac{n+2}{2}}}.
\end{equation*}
\end{prop}

\subsubsection{Proof of Proposition \ref{qfunction}}
The function will be constructed using a multiple-scale expansion in terms of the form $\frac{1}{t^{k/2}}$, so we start by making the following ansatz:
\begin{equation*} 
q(t,x)=A(t)v\left(t,x,\frac{x-\ws(e)t\,e}{R(t)}\right),
\end{equation*}
where $R(t)=\sqrt{t}$, $v$ is periodic in $x$, and $v(t,x,z)=0$ for $z\cd e'=0$. If $q$ was an exact solution of \eqref{pbareq}, i.e.
\begin{equation*}
(1-\beta(t))q_t=\Delta q+\kappa(x)\cd\nabla q,
\end{equation*}
using that $R'=\frac{1}{2R}$ we would have
\begin{equation*}
(1-\beta(t))\left(\frac{A'}{A}v+v_t-\frac{\ws(e)e}{R}\cd\nabla_zv-\frac{z}{2R^2}\cd\nabla_zv \right) =\Delta_xv +\frac{2}{R}\nabla_x\cd\nabla_zv +\frac{1}{R^2}\Delta_zv +\kappa(x)\cd\nabla_xv +\frac{\kappa(x)}{R}\cd\nabla_zv.
\end{equation*}
Furthermore we write $v$ using the expansion
\begin{equation*}
v(t,x,z):=v^0(z)+\frac{1}{R(t)}v^1(x,z)+\frac{1}{R(t)^2}v^2(x,z)+\frac{1}{R(t)^3}v^3(x,z)
\end{equation*}
with $v^1, v^2, v^3$ periodic in $x$, uniformly bounded in $z$ on compact sets. Then the equation above becomes
\begin{equation}\label{expansion} 
\begin{array}{l}
\ds (1-\beta)\frac{A'}{A} \left(v^0+\frac{1}{R}v^1+\frac{1}{R^2}v^2+\frac{1}{R^3}v^3\right) - (1-\beta)\frac{1}{2R^3}\left(v^1+\frac{2}{R}v^2+\frac{3}{R^2}v^3\right) \\
\ds -(1-\beta)\frac{\ws(e)e}{R}\cd \left(\nabla_zv^0+\frac{1}{R}\nabla_zv^1+\frac{1}{R^2}\nabla_zv^2+\frac{1}{R^3}\nabla_zv^3\right) \\
\ds -(1-\beta)\frac{z}{2R^2}\cd \left(\nabla_zv^0+\frac{1}{R}\nabla_zv^1+\frac{1}{R^2}\nabla_zv^2+\frac{1}{R^3}\nabla_zv^3\right)\\
\ds \quad\quad =\frac{1}{R}\left(\Delta_xv^1+\frac{1}{R}\Delta_xv^2+\frac{1}{R^2}\Delta_xv^3\right)
+\frac{2}{R^2}\left(\nabla_x\cd\nabla_zv^1+\frac{1}{R}\nabla_x\cd\nabla_zv^2+\frac{1}{R^2}\nabla_x\cd\nabla_zv^3\right) \\
\ds \quad\quad +\frac{1}{R^2}\left(\Delta_zv^0+\frac{1}{R}\Delta_zv^1+\frac{1}{R^2}\Delta_zv^2+\frac{1}{R^3}\Delta_zv^3\right) 
+\frac{\kappa(x)}{R}\cd\left(\nabla_xv^1+\frac{1}{R}\nabla_xv^2+\frac{1}{R^2}\nabla_xv^3\right) \\
\ds \quad\quad +\frac{\kappa(x)}{R}\left(\nabla_zv^0+\frac{1}{R}\nabla_zv^1+\frac{1}{R^2}\nabla_zv^2+\frac{1}{R^3}\nabla_zv^3\right).
\end{array}
\end{equation}
At this point we can set $A(t)=\frac{1}{t^m}$, with $m$ to be decided later. Then $\frac{A'}{A}=-\frac{m}{t}=-\frac{m}{R^2}$ in the expression above. On the other hand, 
$$\beta(t)=\frac{\alpha}{\cs(e')t}+O\left(\frac{1}{t^{3/2}} \right)=\frac{\alpha}{\cs(e')R^2}+O\left(\frac{1}{R^{3}} \right).$$
We will now choose the functions $v^i$ in such a way that the terms of order $\frac{1}{R}, \frac{1}{R^2}$ and $\frac{1}{R^3}$ cancel out. Note that for $\beta(t)$ plays no role for these terms.\\

\noindent Equating the terms of order $1/R$ we get
\begin{equation*} 
(\Delta_x +\kappa(x)\cd\nabla_x)v^1(x,z)=-(\kappa(x)+\ws(e)e)\cd\nabla_zv^0(z).
\end{equation*}
Note that the operator on the left is what we denoted by $L_x$ in the proof of Lemma \ref{nu.existance}. In the same lemma we constructed the functions $\chi_i$ in such a way that the vector $\chi=\sum_i\chi(x)_ie_i$ satisfied equation \eqref{chi}, namely
\begin{equation*} 
L_x\chi(x)=\ls(e')(\kappa(x)+\ws(e)e).
\end{equation*}
Using this, family of solutions to the equation for $v^1$ is given by
\begin{equation*}
v^1(x,z)=V^1(z)-\bar{\chi}(x)\cd\nabla_zv^0(z),
\end{equation*} 
where $\bar{\chi}(x)=\tilde{\chi}(x)+\chi_0$ for any constant $\chi_0$ and function $V^0(z)$, and $\tilde{\chi}=\frac{\chi(x)}{\ls(e')}$. \\

\noindent Next, equating the terms of order $1/R^2$ we get
\begin{equation*}
-mv^0-\ws(e)e\cd\nabla_zv^1-z/2\cd\nabla_zv^0 =\Delta_xv^2+2\nabla_x\cd\nabla_zv^1+\Delta_zv^0 +\kappa(x)\cd\nabla_xv^2+\kappa(x)\cd\nabla_zv^1,
\end{equation*}
which can be rewritten as
\begin{equation}\label{v2} 
L_xv^2= -(2\nabla_x+\kappa(x)+\ws(e)e)\cd\nabla_zv^1-(\Delta_z+z/2\cd\nabla_z+m)v^0.
\end{equation}
The criterion for the solvability of this equation is that the function on the right lies in the orthogonal complement of the kernel of the adjoint $L^*$ of the operator. Indeed the principal periodic eigenvalue of $L$ is $0$ and its corresponding eigenspace is 1-dimensional, spanned by the function $1$. Then by Krein-Rutman theorem, $L^*$ has $0$ as its principal eigenvalue, and the corresponding eigenspace is 1-dimensional as well. However from \eqref{kr}, $L^*\nu(x)=0$, so this eigenspace is spanned by $\nu(x)$. This shows that the kernel of $L^*$ is spanned by $\nu(x)$, so the solvability of \eqref{v2} is equivalent to the function on the right hand side being orthogonal to $\nu(x)$. This translates into
\begin{equation*}
\ds 0=\ds\int \nu(x)(2\nabla_x+\kappa(x)+\ws(e)e)\cd\nabla_zv^1(x,z)\,dx+ (\Delta_z+z/2\cd\nabla_z+m)v^0(z)\int\nu(x)\,dx.
\end{equation*} 
The term involving $V^1(z)$ vanishes because by \eqref{chi} and \eqref{kr},
\begin{equation*}
\int \nu(x)(\kappa(x)+\ws(e)e)\,dx=\int\nu(x) L\tilde{\chi}(x)\,dx=-\int L^*\nu(x)\tilde{\chi}(x)\,dx=0.
\end{equation*} 
On the other hand, $\int\nu(x)\,dx=1$, so the last term is an elliptic equation in $z$ for $v^0$. Finally, the $\bar{\chi}\cd\nabla_zv^0$ contributes by
\begin{equation*}
\sum\limits_{i,j}\left[\int\nu(x)(2\nabla_x+\kappa(x)+\ws(e)e)\cd e_i\tilde{\chi}_j(x)\,dx\right]\partial_{z_iz_j}v^0(z):=\sum\limits_{i,j}M_{ij}\partial_{z_iz_j}v^0(z).
\end{equation*}
Note that the constant $\chi_0$ vanishes as well for the same reason as $V^0(z)$. Therefore we end up with the equation
\ea{ll}{
\ds\nabla_z\cd((I-M)\nabla_zv^0) +z/2\cd\nabla_zv^0(z)+mv^0=0, \quad & z\cd e'>0,\\
\ds v^0(z)=0, & z\cd e'=0,\\
\ds v^0(z)>0, & z\cd e'>0.
}
To solve this we first show that this equation is elliptic by substituting $I-M$ with a symmetric positive-definite matrix $S$. Then $m$ will be simply the principal eigenvalue of the operator $\nabla\cd S\nabla+z/2\cd\nabla$. Indeed let
\begin{equation*}
S_{ij}:=I_{ij}-\frac{1}{2}(M_{ij}+M_{ji}).
\end{equation*}
\begin{equation*} 
\begin{array}{ll}
M_{ij} &= \ds \int 2\nu\nabla\tilde{\chi}_j\cd e_i+\int \nu(\kappa+\ws(e)e)\cd e_i\tilde{\chi}_j =\int 2\nu\nabla\tilde{\chi}_j\cd e_i +\int\nu(\Delta\tilde{\chi}_i+k\cd\nabla\tilde{\chi}_i )\tilde{\chi}_j\\
&\ds=\int 2\nu\nabla\tilde{\chi}_j\cd e_i -\int\nu\nabla\tilde{\chi}_i\cd\nabla\tilde{\chi}_j -\int(\nabla\nu\cd\nabla\tilde{\chi}_i-\nu\kappa\cd\nabla\tilde{\chi}_i)\tilde{\chi}_j\\
&\ds=\int 2\nu\nabla\tilde{\chi}_j\cd e_i -\int\nu\nabla\tilde{\chi}_i\cd\nabla\tilde{\chi}_j -\int F\cd\nabla\tilde{\chi}_i\tilde{\chi}_j,
\end{array}
\end{equation*}
where for the last equality we used \eqref{nu}. Then using the fact that $F$ is divergence free, we get
\begin{equation*}
\begin{array}{ll}
S_{ij}&\ds=\delta_{ij} -\int\nu(\nabla\tilde{\chi}_j\cd e_i+\nabla\tilde{\chi}_i\cd e_j) -\int\nu\nabla\tilde{\chi}_i\cd\nabla\tilde{\chi}_j +\frac{1}{2}\int F\cd(\nabla\tilde{\chi}_i\tilde{\chi}_j+\nabla\tilde{\chi}_j\tilde{\chi}_i)\\
&\ds=\int\nu(e_i\cd e_j-(\nabla\tilde{\chi}_i\tilde{\chi}_j+\nabla\tilde{\chi}_j\tilde{\chi}_i)+\nabla\tilde{\chi}_i\cd\nabla\tilde{\chi}_j)\\ &\ds=\int\nu(e_i-\nabla\tilde{\chi}_i)\cd(e_j-\nabla\tilde{\chi}_j).
\end{array}
\end{equation*}
Since $\nabla\tilde{\chi}_i$'s are not constant, $S$ is a symmetric positive-definite matrix, and therefore $v^00$ solves the elliptic equation $\nabla\cd S\nabla v^0+z/2\cd\nabla v^0+mv^0=0$. The solution of this can be written down explicitly as 
\begin{equation*}
v^0(z)=z\cd e'\e{-\frac{z\cd S^{-1}z}{4}}.
\end{equation*}
This forces $m=\frac{n+1}{2}$ as the principal eigenvalue of the operator. So up to now we have explicitly determined $m, v^0$,and $v^1(x,z)$, have freedom in $V^1(z)$ and have established the existence of $v^2(x,z)$ in a way that cancels terms of order $1/R$ and $1/R^2$. Using the exact expression for $v^0$ in \eqref{v2} we get that
\begin{equation*} 
\begin{array}{ll}
L_xv^2(x,z)&\ds=-\nabla_z\cd(M \nabla_zv^0) +(2\nabla_x+\kappa+\ws(e)e)\cd\nabla_z(\bar{\chi}\cd\nabla_zv^0) -(2\nabla_x+\kappa+\ws(e)e)\cd\nabla_zV^1\\
&\ds=\sum\limits_{i,j}\left[ -M_{ij}+(2\partial_{x_i}+\kappa_i(x)+\ws(e)e\cd e_i)\bar{\chi}_j(x) \right]\partial_{z_iz_j}v^0(z) -(\kappa(x)+\ws(e)e)\cd\nabla_zV^1(z).
\end{array}
\end{equation*}
Note that we can find functions $\chi^{ij}(x)$ that are periodic in $x$ and solve 
\begin{equation*} 
L_x\chi^{ij}(x)=-M_{ij}+(2\partial_{x_i}+\kappa_i(x)+\ws(e)e\cd e_i)\bar{\chi}_j(x) ,
\end{equation*}
because the solvability criterion for this is that the right hand side is orthogonal to $\nu(x)$, which follows directly from the definition of $M_{ij}$. Using this, a family of solutions for $v^2$ can be given by
\begin{equation*} 
v^2(x,z)=\sum_{ij}\chi^{ij}(x)\partial_{z_iz_j}v^0(z)-\bar{\chi}(x)\cd\nabla_zV^1(z)+V^2(z).
\end{equation*}
Finally, equating the terms of order $1/R^3$ we get
\begin{equation} \label{v3}
L_xv^3= \frac{\alpha}{\cs(e')}\ws(e)e\cd\nabla_zv^0 - \left(\Delta_z+\frac{z}{2}\cd\nabla_z+\left(m+\frac{1}{2}\right)\right)v^1 - (2\nabla_x+\kappa+\ws(e)e)\cd\nabla_zv^2
\end{equation}
Substituting in the solutions for $v^1$ and $v^2$ we obtain
\begin{equation*} 
\begin{array}{l}
L_xv^3 = \ds \frac{\alpha}{e\cd e'}e\cd\nabla_zv^0(z) -\left(\Delta_z+\frac{z}{2}\cd\nabla_z+\left(m+\frac{1}{2}\right)\right)V^1(z)\\ \ds+\sum\limits_{i}\left(\Delta_z+\frac{z}{2}\cd\nabla_z+\left(m+\frac{1}{2}\right)\right)\partial_{z_i}v^0(z)\bar{\chi}_i(x)
\ds -\sum\limits_{ijk}\partial_{z_iz_jz_k}v^0(z)(2\partial_{x_i}+\kappa_i(x)+\ws(e)e\cd e_i)\chi^{jk}(x)\\ \ds+\sum\limits_{ij}\partial_{z_iz_j}V^1(z)(2\partial_{x_i}+\kappa_i(x)+\ws(e)e\cd e_i)\bar{\chi}_j(x) -(\kappa(x)+\ws(e)e)\cd\nabla_zV^2(z).
\end{array}
\end{equation*}
To solve for $v^3$ we need again that the right hand side is orthogonal to $\nu(x)$. This translates into
\begin{equation*}
\begin{array}{l}
\ds \frac{\alpha}{e\cd e'}e\cd\nabla_zv^0(z) -\left(\Delta_z+\frac{z}{2}\cd\nabla_z+\left(m+\frac{1}{2}\right)\right)V^1(z) +\sum\limits_{i} \left(\int\nu\bar{\chi}_i\right)\left(\nabla_z\cd(I-S)\nabla_z+\frac{1}{2}\right)\partial_{z_i}v^0(z)\\
\ds -\sum\limits_{ijk}\left(\int\nu(2\partial_{x_i}+\kappa_i+\ws(e)e\cd e_i)\chi^{jk}\right) \partial_{z_iz_jz_k}v^0(z) +\nabla_z\cd(I-S)\nabla_zV^1(z)=0
\end{array}
\end{equation*}
Since $V^2$ does not appear in this condition, then we can take it to be $0$, ie set $V^2\equiv 0$. Let
\begin{equation*}
\begin{array}{l}
\ds a_i=\int\nu(x)\bar{\chi}_i(x)\,dx, \quad a=[a_1,\ldots, a_n]\\
\ds a_{ijk}=\int\nu(x)(2\partial_{x_i}+\kappa_i(x)+\ws(e)e\cd e_i)\chi^{jk}(x)\,dx
\end{array}
\end{equation*}
Then the solvabilitiy condition turns into an elliptic equation for $V^1$:
\begin{equation*}
\nabla_z\cd(S\nabla_zV^1) + \frac{z}{2}\cd\nabla_zV^1+\left(m+\frac{1}{2}\right)V^1 =-\sum\limits_{ijk}(a_{ijk}+a_i(I_{jk}-S_{jk})) \partial_{z_iz_jz_k}v^0 +\left(\frac{\alpha}{e\cd e'}e+\frac{1}{2}a\right)\cd\nabla_zv_0
\end{equation*}
Since we are interested in approximate solutions that are zero on the boundary $x\cd e'=\cs(e')$ of the moving domain, we look for $V^1$ satisfying the equation above in $z\cd e'>0$, together with the boundary conditions $V^1=0$ and $e'\cd\nabla_zV^1=0$ if $z\cd e'=0 $. Its existence is ensured by Cauchy-Kovalevskaya, since the coefficients and the function on the right hand side are all real analytic.
%
So $q(t,x)$ becomes
\begin{equation*}
\begin{array}{l} 
\ds q(t,x)=\ds\frac{1}{t^{\frac{n+1}{2}}} v^0(z)+\frac{1}{t^{\frac{n+2}{2}}} v^1(x,z)+\frac{1}{t^{\frac{n+3}{2}}} v^2(x,z)+\frac{1}{t^{\frac{n+4}{2}}} v^3(x,z)\\
\ds=\ds\frac{1}{t^{\frac{n+1}{2}}} \frac{(x-\ws(e)t\,e)\cd e'}{\sqrt{t}}\e{-\frac{(x-\ws(e)t\,e)\cd S^{-1}(x-\ws(e)t\,e)}{4t}} +\frac{1}{t^{\frac{n+2}{2}}}V^1\left(\frac{x-\ws(e)t\,e}{\sqrt{t}}\right)\\ \ds-\frac{1}{t^{\frac{n+2}{2}}}\left(\frac{\chi(x)}{\ls(e')}+\chi_0\right)\cd\left(e'-\frac{1}{2t}(x-\ws(e)t\,e)\cd e'S^{-1}(x-\ws(e)t\,e)\right)\e{-\frac{(x-\ws(e)t\,e)\cd S^{-1}(x-\ws(e)t\,e)}{4t}}\\
\ds +\frac{1}{t^{\frac{n+3}{2}}}\sum\limits_{ij}\chi^{ij}(x)\left(\frac{x-\ws(e)t\,e}{2\sqrt{t}}\cd S^{-1}(e_i\cd e'e_j+e_j\cd e'e_i)\right)\e{-\frac{(x-\ws(e)t\,e)\cd S^{-1}(x-\ws(e)t\,e)}{4t}}\\
\ds +\frac{1}{t^{\frac{n+3}{2}}}\sum\limits_{ij}\chi^{ij}(x) \left(\frac{1}{4t^{\frac{3}{2}}}(x-\ws(e)t\,e)\cd e'(x-\ws(e)t\,e)\cd S^{-1}e_i\cd e'(x-\ws(e)t\,e)\cd S^{-1}e_i\right)\times\\
\hspace*{300pt}\e{-\frac{(x-\ws(e)t\,e)\cd S^{-1}(x-\ws(e)t\,e)}{4t}}\\ \ds-\frac{1}{t^{\frac{n+3}{2}}}\left(\frac{\chi(x)}{\ls(e')}+\chi_0\right)\cd\nabla_zV^1\left(\frac{x-\ws(e)t\,e}{\sqrt{t}}\right) +\frac{1}{t^{\frac{n+4}{2}}} v^3\left(x,\frac{x-\ws(e)t\,e}{\sqrt{t}}\right).
\end{array}
\end{equation*}
Hence, if $x\in B_{\sigma\sqrt{t}}(\ws(e)t\,e)$ and $(x-\ws(e)t\,e)\cd e'\ge 0$, \eqref{qbounds} holds due to the periodicity in $x$ and uniform bounds in $z$ of $v^2, v^3, V^1$.

\subsubsection{Proof of Proposition \ref{approxsol}}
Let $\phi(t,x):=\bar{p}^{app}(t,x)-q(t,x)$. Using the equations for $\bar{p}^{app}$ and $q(t,x)$ and Lemma \ref{nu.existance}, we get that $\phi$ solves
\ea{ll}{
\ds(1-\beta(t))\partial_t\phi=\frac{1}{\nu(x)}\nabla\cd(\nu(x)\nabla\phi)-\frac{F(x)}{\nu(x)}\cd\nabla\phi+O\left(\frac{1}{t^{\frac{n+5}{2}}}\right), \quad & x\in C_t,\ t>1\\
\ds\phi(t,x)=0, & x\in\partial C_t,\ t>1,\\
\ds\phi(1,x)=0, & x\in\partial C_1.
}
%
%
Note that $|\beta(t)|\le 1/2$ for all $t$, and $F$ and $\nu$ are uniformly bounded, so we may apply the ABP maximum principle (see \cite{krylov}). Denote by $D_t$ the parabolic cylinder $D_t:=\{(s,x): s\in(1,t), x\in C_s\}$, and by $\partial'D_t:=\{(s,x):\ s=1 \ \text{ or } \ s\in(1,t), x\in\partial D_s\}$ its parabolic boundary. Then
\begin{equation*}
\sup\limits_{D_t}|\phi|\le \sup\limits_{\partial'D_t}|\phi|+ C(diam(C_t))^{\frac{n}{n+1}}\left|\left|O\left(\frac{1}{t^{\frac{n+5}{2}}}\right)\right|\right|_{L^{n+1}(D_t)},
\end{equation*}
where $C$ depends only on the dimension. Then using the fact that $\phi=0$ on $\partial'D_t$ we get that for all $t\ge 1$,
\begin{equation*}
|\phi(t,x)|\le \frac{C\sigma^{\frac{1}{n+1}}}{t^{\frac{n+2}{2}}}.
\end{equation*}

\subsection{Proof of the main proposition \ref{p.bar.est}}\label{sec.logprop}
From the previous two sections we have control of $\bar{p}$ at positions of order $O(\sqrt{t})$ from the boundary, as well as explicit bounds for $\bar{p}^{app}$ with boundary data on $P_t$ agreeing with $q$. Note that we have a degree of freedom in the definition of $q$ in \eqref{qbounds} coming from $\chi_0$, and we will pick this appropriately so that $\bar{p}^{app}$ and $\bar{p}$ are comparable at the boundary of $P_t$, where $P_t$ is defined as
$$P_t:=\{x=x'+(\ws(e)t+r)e: r\in(0,\epsilon\sqrt{t}), x'\cd e'=0, |x'|\le\epsilon'\sqrt{s}, \text{ for } s<t \text{ so that } \ws(e)t+r=\ws(e)s+\epsilon\sqrt{s}\}.$$
$\epsilon$ and $\epsilon'$ are taken so that the result of Proposition \ref{sqrttdecay} holds for a fixed $x_0$. Note that 
\begin{equation}\label{s.lowerbound}
\ws(e)t<\ws(e)t+r =\ws(e)s+\epsilon\sqrt{s} \le(\ws(e)+\epsilon)s \text{ for } s\ge 1 \implies s\ge\frac{\ws(e)}{\ws(e)+\epsilon}t, \forall t\ge t_0,
\end{equation}
where $t_0$ is so that $s\ge 1$ is satisfied. We split the boundary of $P_t$ in three components:
\begin{equation*}
\partial^0P_t:=\partial P_t\cap\{r=0 \}, \quad \partial^{\epsilon\sqrt{t}}P_t=\partial P_t\cap \{r=\epsilon\sqrt{t} \}, \quad \partial^sP_t=\partial P_t\cap \{r\in(0,\epsilon\sqrt{t}), |y|=\epsilon'\sqrt{s} \}
\end{equation*}

The strategy is to control $\bar{p}$ in $P_t$ by keeping it between positive multiples of $\bar{p}^{\pm}$, defined as solutions of \eqref{approxsoleq} corresponding to $q^\pm$ in Proposition \ref{qfunction} with $\chi_0=\chi_0^{\pm}$ chosen appropriately. If $\chi_0\cd e'> \left|\left|\frac{\chi\cd e'}{\ls(e')}\right|\right|_{L^\infty}$, then $q(t,x)<0$ for $x\cd e'=\cs(e')t$ and $t$ large enough, therefore let $\chi_0^-=2\left|\left|\frac{\chi\cd e'}{\ls(e')}\right|\right|_{L^\infty}e'$. Similarly $q>0$ on the same boundary if $\chi_0\cd e'<- \left|\left|\frac{\chi\cd e'}{\ls(e')}\right|\right|_{L^\infty}$, so let $\chi_0^+=-2\left|\left|\frac{\chi\cd e'}{\ls(e')}\right|\right|_{L^\infty}e'$.

From Proposition \ref{qfunction},
\begin{equation*}
\left|q^{\pm}(t,x)-\frac{C\sigma}{t^{\frac{n+1}{2}}}\right|\le \frac{C}{t^{\frac{n+2}{2}}}
\end{equation*}
for all $x=y+(\ws(e)t+r)e, r\in\left(\frac{\sigma}{2}\sqrt{t},\sigma\sqrt{t}\right), |y|<\sigma\sqrt{t}, y\cd e'=0$. Then from Proposition \ref{sqrttdecay} with $\sigma=\epsilon$, by decreasing $\epsilon'$ if necessary, we have that on $\partial^{\epsilon\sqrt{t}}P_t$ and $\partial^sP_t$ for $s$ corresponding to $r\in(\epsilon/2,\epsilon)$,
\begin{equation*}
\bar{p}(t,x)\le\frac{C_\epsilon}{t^{\frac{n+1}{2}}}\le Mq^+(t,x)=M\bar{p}^+(t,x),
\end{equation*}
\begin{equation*}
\bar{p}(t,x)\ge\frac{C_\epsilon}{t^{\frac{n+1}{2}}}\ge M^{-1}q^-(t,x)=M^{-1}\bar{p}^-(t,x).
\end{equation*}
On the other hand, from the choice of $\chi_0$, for all $x=y+\ws(e)t\,e$ with $y\cd e'=0, |y|\le\sigma/2\sqrt{t}$ and $t>1$,
\begin{equation*}
\bar{p}(t,x)=0\le Mq^+(t,x)=M\bar{p}^+(t,x),
\end{equation*}
\begin{equation*}
\bar{p}(t,x)=0\ge M^{-1}q^-(t,x)=M^{-1}\bar{p}^-(t,x).
\end{equation*}
Hence, we only need to prove the same inequalities for $x\in\partial^sP_t\cap\{r\in(0,\frac{\epsilon}{2}\sqrt{t}) \}$. In this range,
\begin{equation*}
\ws(e)s+\epsilon\sqrt{s}\le\ws(e)t+\frac{\epsilon}{2}\sqrt{t} \implies \frac{\epsilon}{2}\sqrt{t} \le\ws(e)(t-s)+\epsilon(\sqrt{t}-\sqrt{s}) =\ws(e)(t-s)+\epsilon\frac{t-s}{\sqrt{t}+\sqrt{s}} \le(\ws(e)+\epsilon)(t-s),
\end{equation*}
where the last inequality holds because $s,t\ge 1$. Hence
\begin{equation}\label{t-s.bound}
t-s\ge\frac{\epsilon/2}{\ws(e)+\epsilon}\sqrt{t}.
\end{equation}
The upper bound follows from
\begin{equation*}
\bar{p}(t,x+\ws(e)t\,e)\le C\frac{x\cd e'}{t^{n+\frac{1}{2}}}, \quad \forall x\in\rn.
\end{equation*}
This can be proved following the same steps as in Proposition \ref{linearbounds}, using Lemma \ref{l1bound}. Using this and \eqref{qbounds} for $x\in\partial^sP_t$, and increasing $M$ if necessary,
\begin{equation*}
\bar{p}(t,x=x'+(\ws(e)t+r)e) \le C\frac{r}{t^{\frac{n}{2}+1}} \le Mq_+(t,x)=Mp^+.
\end{equation*}
The lower bound follows from a combination of the result of Proposition \ref{sqrttdecay} at time $s$ and heat kernel bounds. It is easy to show that the result of Lemma \ref{heat.kernel.bounds} is still true for \eqref{pbareq}, so denoting by $\bar{\Gamma}$ the heat kernel corresponding to this equation, we have that for $x\in\partial^sP_t\cap\{r\in(0,\frac{\epsilon}{2}\sqrt{t})\}$,
\begin{equation*}
\begin{array}{ll}
\ds\bar{p}(t,x=x'+(\ws(e)t+r)e) &\ds=\bar{p}(t,x'+(\ws(e)s+\epsilon\sqrt{s})e) \ge\int\bar{p}(s,y)\bar{\Gamma}(t,x'+(\ws(e)t+r)e,s,y)\,dy  \\
\ds &\ds =\int\bar{p}(s,y+\ws(e)s\,e)\bar{\Gamma}(t,x'+(\ws(e)t+r)e,s,y+\ws(e)s\,e)\,dy.
\end{array}
\end{equation*}
From Lemma \ref{heat.kernel.bounds},
\begin{equation*}
\bar{\Gamma}(t,x'+(\ws(e)t+r)e,s,y+\ws(e)s\,e)\ge \frac{a}{K(t-s)^\frac{n}{2}}\e{-K\frac{|x'+re-y|^2}{t-s}}
\end{equation*}
for all $s<t<s+R^2$, $x'+re,y\in B_{\delta R}(x_0)$, with $a,K$ depending only on $\delta\in(0,1)$, but independent of $x_0$ and $R$.\\
On the other hand, from Proposition \ref{sqrttdecay},
\begin{equation*}
\bar{p}(s,y+\ws(e)s\,e)\ge \frac{1}{Cs^{\frac{n+1}{2}}}, \quad \forall y\in B_{\epsilon'\sqrt{s}}((\epsilon\sqrt{s}+x_0')e), \text{ where } x_0'>0.
\end{equation*}
Now choose $x_0=(\epsilon\sqrt{s}+x_0')e$; $R=\sqrt{\frac{\epsilon}{\ws(e)}}\sqrt{s}$, so that $t-s\le R^2$ from \eqref{s.lowerbound}; $\delta=\epsilon'\sqrt{\frac{\ws(e)}{\epsilon}}$ so that $\delta R=\epsilon'\sqrt{s}$. By decreasing $\epsilon'$ if necessary, $\delta<1$. These imply that
\begin{equation*}
\begin{array}{ll}
\ds\bar{p}(t,x=x'+(\ws(e)t+r)e) &\ds\ge \int\limits_{B_{\epsilon'\sqrt{s}}((\epsilon\sqrt{s}+x_0')e)}\bar{p}(s,y+\ws(e)s\,e)\bar{\Gamma}(t,x'+(\ws(e)t+r)e,s,y+\ws(e)s\,e)\,dy.\\
&\ds\ge C \frac{1}{s^{\frac{n+1}{2}}} \frac{1}{(t-s)^{\frac{n}{2}}} \int\limits_{B_{\epsilon'\sqrt{s}}((\epsilon\sqrt{s}+x_0')e)}\e{-K\frac{|x'+re-y|^2}{t-s}}.
\end{array}
\end{equation*}
To bound the integral, we use the fact that
\begin{equation*}
\frac{1}{(t-s)^{\frac{n}{2}}} \int\limits_{\rn}\e{-K\frac{|x'+re-y|^2}{t-s}}= \text{constant}\implies \frac{1}{(t-s)^{\frac{n}{2}}} \int\limits_{B_{\gamma\sqrt{t-s}}(x'+re)}\e{-K\frac{|x'+re-y|^2}{t-s}}\ge c_\gamma
\end{equation*}
for every $\gamma\in\R^+$. Thus we could bound the integral from below if for some $\gamma$,
\begin{equation*}
B_{\gamma\sqrt{t-s}}(x'+re)\subseteq B_{\epsilon'\sqrt{s}}((\epsilon\sqrt{s}+x_0')e).
\end{equation*}
But this is true by \eqref{t-s.bound} and the fact that $|x'|=\epsilon'\sqrt{s}$, so
\begin{equation*}
\bar{p}(t,x=x'+(\ws(e)t+r)e)\ge c\frac{1}{s^\frac{n+1}{2}} =c\frac{\sqrt{s}}{s^\frac{n+2}{2}}\ge c\frac{r}{t^{\frac{n+2}{2}}}.
\end{equation*}
The last inequality is due to $r\le\epsilon\sqrt{s}$ and \eqref{s.lowerbound}.

Now from the bounds \eqref{qbounds} of $q^-$ we have that on the same part of the boundary of $P_t$, $q^-(t,x=x'+(\ws(e)t+r)e)\le C\frac{r}{t^{\frac{n+2}{2}}}$, thus  increasing $M$ again if necessary, $\bar{p}\ge M^{-1}q^-$.

This shows that on the boundary of $P_t$,
\begin{equation*}
Mp^+\ge \bar{p}\ge M^{-1}p^-,
\end{equation*}
therefore by Proposition \ref{approxsol}, we obtain the bounds in Proposition \ref{p.bar.est}.


%



\begin{thebibliography}{99}
\small

\bibitem{aw1} D.G. Aronson and H.F. Weinberger, {\em Multidimensional nonlinear diffusions arising in population genetics}, Adv. Math. {\bf 30} (1978), 33--76.

\bibitem{aw2} D.G. Aronson, H.F. Weinberger, {\em Nonlinear diffusion in population genetics, combustion, and nerve pulse propagation}, Proc. of the Tulane Program in PDEs and related topics, Springer-Verlag, Berlin-Heidelberg-New York (1975), 5--49.

\bibitem{bh1} H. Berestycki and F. Hamel, {\em Front propagation in periodic excitable media}, Comm. Pure Appl. Math. {\bf 55} (2002), 949--1032.

\bibitem{bh2} H. Berestycki and F. Hamel, {\em Generalized travelling waves for reaction-diffusion equations}, Perspectives in Nonlinear Partial Differential Equations. In honor of H.~Brezis, Amer. Math. Soc., Contemp. Math. {\bf 446}, 2007, 101--123.

\bibitem{bhn1} H. Berestycki, F. Hamel and N. Nadirashvili, {\em The speed of propagation for KPP type problems. I.~Periodic framework}, J. Eur. Math. Soc. {\bf 7} (2005), 173--213. 

\bibitem{bhn2} H. Berestycki, F. Hamel and N. Nadirashvili, {\em Elliptic eigenvalue problems with large drift and applications to nonlinear propagation phenomena}, Comm. Math. Phys. {\bf 253} (2005), 451--480.

\bibitem{bbd} J. Berestycki, \'E. Brunet and B. Derrida, {\em Exact solution and precise asymptotics of a Fisher-KPP type front}, J. Phys. A: Math. Theor. {\bf 51} (2018).

\bibitem{bram1} M.D. Bramson, {\em Maximal displacement of branching Brownian motion}, Comm. Pure Appl. Math. {\bf 31} (1978) 531--581.

\bibitem{bram2} M.D. Bramson, {\em Convergence of solutions of the Kolmogorov equation to travelling waves}, Mem. Amer. Math. Soc. {\bf 44} (1983). 

\bibitem{es1} U. Ebert, W. van Saarloos, {\em Front propagation into unstable states: universal algebraic convergence towards uniformly translating pulled fronts}, Phys. D {\bf 146} (2000), 1--99.

\bibitem{es2} U. Ebert, W. van Saarloos, B. Peletier, {\em Universal algebraic convergence in time of pulled fronts: the common mechanism for difference-differential and partial differential equations}, European J. Appl. Math. {\bf 13} (2002), 53--66.

\bibitem{evsoug} L.C. Evans, P.E. Souganidis, {\em A PDE approach to certain large deviation problems for systems of parabolic equations}, Analyse non lin\'eaire (Perpignan, 1987). Ann. Inst. H. Poincar\'e Anal. Non Lin\'eaire {\bf 6} (1989), suppl., 229--258.

\bibitem{fs} E. B. Fabes, D. W. Stroock, {\em A new proof of Moser's parabolic Harnack inequality using the old ideas of Nash}, Arch. Ration. Mech. Anal. {\bf 96} (1986), 327--338.

\bibitem{fisher} R.A. Fisher, {\em The wave of advance of advantageous genes}, Ann. Eugenics {\bf 7} (1937), 353--369.

\bibitem{freidlin} M.I. Freidlin, {\em Functional integration and partial differential equations}, Annals of Mathematics Studies {\bf 109}, Princeton University Press, Princeton, NJ, 1985.

\bibitem{fg} M.I. Fre\u{\i}dlin, J. G\"artner, {\em The propagation of concentration waves in periodic and random media}, Dokl. Acad. Nauk SSSR {\bf 249} (1979), no. 3, 521--525.

\bibitem{gartner} J. G\"artner, {\em Location of wave fronts for the multi-dimensional K-P-P equation and Brownian first exit densities}, Math. Nachr. {\bf 105} (1982), 317--351.

\bibitem{hnrr1} F. Hamel, J. Nolen, J.-M. Roquejoffre and L. Ryzhik, {\em A short proof of the logarithmic Bramson correction in Fisher-KPP equations}, Networks and Heterogeneous Media, {\bf 8} (2013), no. 1, 275--289.

\bibitem{hnrr2} F. Hamel, J. Nolen, J.-M. Roquejoffre and L. Ryzhik, {\em The logarithmic delay of KPP fronts in a periodic medium}, Jour. Eur. Math. Soc. {\bf 18} (2016), 465--505.

\bibitem{henderson}  C. Henderson, {\em Population stabilization in branching Brownian motion with absorption and drift}, Comm. Math. Sci. {\bf 14} (2015), no. 4, 973--985.

\bibitem{kpp} A.N. Kolmogorov, I.G. Petrovsky and N.S. Piskunov, {\em \'Etude de l'\'equation de la diffusion avec croissance de la quantit\'e de mati\`ere et son application \`a un probl\`eme biologique}, Bull. Univ. \'Etat Moscou, S\'er. Inter.~A {\bf 1} (1937), 1--26.

\bibitem{kr} M.G. Kr\u ein, M.A. Rutman, {\em Linear operators leaving invariant a cone in a Banach space}, Amer. Math. Soc. Translation 1950 {\bf 26} (1950), 128 pp. 

\bibitem{krylov} N.V. Krylov, {\em Nonlinear elliptic and parabolic equations of the second order}, Mathematics and its Applications, Reidel, Norwell, Massachusetts, 1987.

\bibitem{lau} K.-S. Lau, {\em On the nonlinear diffusion equation of Kolmogorov, Petrovskii and Piskunov}, J. Diff. Eqs. {\bf 59} (1985), 44--70.

\bibitem{ltz} E. Lubetzky, C. Thornett, O. Zeitouni, {\em Maximum of branching Brownian motion in a periodic envirenment}, Preprint, 2018.

\bibitem{mck} H.P. McKean, {\em Application of Brownian motion to the equation of Kolmogorov-Petrovskii-Piskunov}, Comm. Pure Appl. Math. {\bf 28} (1975), 323--331.

\bibitem{nash} J. Nash, {\em Continuity of solutions of parabolic and elliptic equation}, Amer. J. Math. {\bf 80} (1958), 931--954.

\bibitem{nrr1} J. Nolen, J.-M. Roquejoffre, L. Ryzhik, {\em Refined long time asymptotics for the Fisher-KPP fronts}, Comm. Cont. Math. {\bf 21} (2019), no. 7.

\bibitem{nrr2} J. Nolen, J.-M. Roquejoffre, L. Ryzhik, {\em Convergence to a single wave in the Fisher-KPP equation}, Chinese Ann. Math., Series B {\bf 38} (2017), no. 2, 629--646.

\bibitem{norris} J. Norris, {\em Long-time behavior of heat flow: global behavior and exact asymptotics}, Arch. Ration. Mech. Anal. {\bf 140} (1997), 161--195.

\bibitem{uchiyama} K. Uchiyama, {\em The behavior of solutions of some nonlinear diffusion equations for large time}, J.~Math. Kyoto Univ. {\bf 18} (1978), 453--508.

\bibitem{weinberger} H. Weinberger, {\em On spreading speeds and traveling waves for growth and migration models in a periodic habitat}, J. Math. Biol. {\bf 45} (2002).

\bibitem{xin} J. Xin, {\em An introduction to fronts in random media}, Surveys and Tutorials in the Applied Mathematical Sciences, 5. Springer, New York, 2009.

\end{thebibliography}
\end{document}